\documentclass{article}

\usepackage[utf8]{inputenc}
\usepackage{amsmath, amssymb, amsthm, bm}
\usepackage{graphicx, booktabs, subfigure}
\usepackage{hyperref}
\usepackage{tikz}
\usepackage{tkz-euclide}
\usetikzlibrary{chains, scopes, positioning, backgrounds, shapes, fit, shadows, calc}
\usepackage{multirow} 
\usepackage{booktabs}
\usepgflibrary{shapes.geometric}
\usepackage{float}

\usepackage{txfonts, cite}

\numberwithin{equation}{section}
\newtheorem{thm}{Theorem}[section]
\newtheorem{remark}{Remark}[section]
\newtheorem{lem}{Lemma}[section]
\newtheorem{definition}{Definition}[section]
\theoremstyle{definition}
\newtheorem{example}{\noindent Example}
\graphicspath{{Figures/}}

\begin{document}
\title{Simultaneous recovery of an obstacle and its excitation sources from near-field scattering data}

\author{
	Yan Chang\footnote{School of Mathematics, Harbin Institute of Technology, Harbin, P. R. China. {\it 21B312002@stu.hit.edu.cn} } \ and
	Yukun Guo\footnote{School of Mathematics, Harbin Institute of Technology, Harbin, P. R. China. {\it ykguo@hit.edu.cn} (Corresponding author)}
	}
	\date{}
	\maketitle
	
\begin{abstract}
This paper is concerned with the inverse problem of determining an obstacle and the corresponding incident point sources in the Helmholtz equation from near-field scattering data. An optimization method is proposed to simultaneously recover both the obstacle and source locations. Moreover, a two-step sampling scheme with novel indicator functions is proposed to produce a good initial guess for solving the optimization problem. Theoretically, we analyze the convergence properties of the optimization method and the indicating behaviors of the indicator functions. Several numerical examples are presented to show the effectiveness of the proposed method.
\end{abstract}
	

\section{Introduction}
Inverse problems for wave equations arise widely in non-destructive testing \cite{Bao1, Arridge}, underwater acoustics \cite{Baum, LPH}, antenna synthesis \cite{Devaney, Ramm} and extensive other application areas. Being two prototypical models in the field of inverse problems, the time-harmonic inverse obstacle scattering and inverse source problem have received significant investigations in the literature. For the nonlinear inverse obstacle scattering problem, the emanating source (incident wave) is given and the inhomogeneous scatterer serves as the target. Reciprocally, in the linear inverse source problem, the source is unknown while the passive scatterering inhomogeneity is vacant. Classical algorithms for inverse obstacle scattering problems involve the optimization methods \cite{Colton, Guo1, Kirsch3}, the singular source method \cite{Potthast} and the recursive linearization algorithm \cite{Bao}. In addition, there are some  sampling-type methods with qualitative features, for instance, the factorization method \cite{Kirsh2}, the linear and direct sampling methods \cite{LLZ08, LXD}, the single-shot schemes \cite{LJZ1, LJZ2} and the reverse time migration approach \cite{Chen, Chen1}. Recently, a variety of reconstruction methods have also been developed to deal with the inverse source problems. Bao et al. \cite{Bao2} applied a recursive algorithm to recover unknown sources of acoustic field with multi-frequency measurement data. Zhang et al. \cite{Zhang} used the direct sampling method to locate multipolar acoustic sources for the Helmholtz equation from near-field Cauchy data. Bousba et al. \cite{Bousba} utilized the direct sampling method to identify point sources for the Helmholtz equation from far-field data. The algebraic method has also been applied to the Helmholtz equation for determining monopolar sources \cite{Badia}. 
	
In this paper, we are concerned with an inverse problem to simultaneously reconstruct the acoustic obstacle and its excitation source points from the near-field measurements. This co-inversion problem can be mathematically formulated as an intrinsically coupled problem comprising the aforementioned inverse obstacle scattering and inverse source problem. In comparison with the vast studies on the conventional inverse obstacle/source problems, focuses on identification of both unknown source and scatterer are relatively rare. Over recent years, some interesting attempts have been made in dealing with such co-inversion problems. In \cite{LHY3}, Liu et al. established sufficient conditions to recover both the sound speed of the medium and the source in thermo- and photo-acoustic tomography. The simultaneous recovery of an embedded obstacle and its surrounding medium by formally-determined scattering data was studied in \cite{LHY1}. In \cite{Hu}, the authors proved the uniqueness in simultaneously recovering the source functions and the obstacle for the wave equation. An adjoint state method for recovering both the source and attenuation in the attenuated X-ray transform was proposed in \cite{Luo}. In addition, we also refer to \cite{LLM19, LLM21} for some recent studies on the determination of both potential and source in a random Schr\"{o}dinger equation.
	 
In contrast to the standard problem of reconstructing either the source or obstacle solely, it is substantially more challenging to identify these concurrently unknown targets. It is worthwhile noting that, in the scenario of co-reconstruction, only the total field is available and the contributions from inaccessible source and scatterer are superimposed in the total field measurements. Hence, decoupling these {\it incident} and {\it scattered} components would inevitably suffer from inherent difficulties. As a result, the usual nonlinearity, ill-posedness as well as the additional lack of information constitute the mathematical and numerical challenges of our current study.
	
In this work, we propose an optimization method to recover an obstacle and its illuminating point sources simultaneously from the measured total field. The salient features of our method can be summarized as follows. First, the proposed method is capable of simultaneously and quantitatively determining both the obstacle and source points since the reconstruction does not rely on any alternating iteration between the source and scatterer. Second, by the idea of the decomposition method and reformulating the inverse problem as a nonlinear optimization problem, the solution process for the direct problem is unnecessary. Third, by designing proper indicator functions, we provide an easy-to-implement and effective strategy for choosing the initial guesses for the optimization method, which avoids the a prior information of the unknown targets. Finally, the proposed method is insensitive to the noise-contaminated data and works well even with limited-aperture measurement data despite the severe ill-posedness.
		
We now introduce the precise formulation of the co-inversion problem under consideration.  Let $D\subset \mathbb{R}^2$ be an open and simply connected domain with $C^2$ boundary $\partial D.$  Given the source point located at $z\in\mathbb{R}^2\backslash\overline{D},$ the incident field $u^i$ due to $z$ is given by
	\begin{equation}\label{eq:fundamental_solution}
		u^i(x;z)=\Phi(x,z)=\frac{\mathrm{i}}{4}H_0^{(1)}(k|x-z|),\quad x\in\mathbb{R}^2\backslash\left(\overline{D}\cup\{z\}\right),
	\end{equation}
	where $k>0$ is the wave number,  $H_0^{(1)}$ is the Hankel function of the first kind of order zero. Then the forward scattering problem for the sound-soft obstacle $D$ is to find the scattered field $u^s\in H_{\rm loc}^1(\mathbb{R}^2\backslash\overline{D})$ such that
	\begin{align}
		\label{eq:Helmholtz} \Delta u^s+k^2 u^s  &\, =0\quad{\rm in }\  \mathbb{R}^2\backslash\overline{D},\\
		\label{eq:Dirichlet}  u &\,=  0\quad{\rm on}\ \partial D,\\
		\label{eq:Sommerfeld}  \lim\limits_{r:=|x|\to\infty}\sqrt{r}&\left(\frac{\partial u^s}{\partial r} -\mathrm{i}k u^s\right)=0,
	\end{align}
	where $u=u^i+u^s$ is the total field. The Sommerfeld radiation condition \eqref{eq:Sommerfeld} holds uniformly with respect to all directions $\hat{x}=x/|x|\in\mathbb{S}:=\{x\in\mathbb{R}^2:|x|=1\}$, and \eqref{eq:fundamental_solution} is also known as the fundamental solution to the Helmholtz equation \eqref{eq:Helmholtz}. It is well known (see, e.g., \cite{Colton1}) that the direct problem \eqref{eq:Helmholtz}-\eqref{eq:Sommerfeld} admits a unique solution $u^s\in H^1_{\rm loc}\left(\mathbb{R}^2\backslash\overline{D}\right)$.
		
	In this paper, we assume that $D\subset B_R:=\{x\in\mathbb{R}^2:|x|<R\}$ and $B_R\backslash\overline{D}$ is connected.
	Let $S:=\cup_{j=1}^N\{z_j\}\subset B_R\backslash\overline{D}$ be a set of distinct source points with $N\in\mathbb{N}_+$ the number of source points. 	
	Denote by $u^s(x; z_j)$ the scattered field corresponding to the incident field $u^i(x; z_j)$ of the form \eqref{eq:fundamental_solution}, and we collect the total field  $u(x; z_j)=u^i(x;z_j)+u^s(x;z_j)$ on the measurement curve $\Gamma_R:=\partial B_R=\{x\in\mathbb{R}^2: |x|=R\}$. Define the measured total field data by
		\[
		\mathbb{U}:=\{u(x; z):x\in \Gamma_R, z\in S\},
		\]
	then the inverse problem under consideration is to determine the obstacle-source pair $(\partial D, S)$
	from the measurements $\mathbb{U},$ namely,
	\begin{equation}\label{inverse_problem}
		\mathbb{U}\to(\partial D, S).
	\end{equation}
	
	We refer to Figure \ref{fig:Setup} for an illustration of the geometry setting of the inverse obstacle and source problem \eqref{inverse_problem}.
	
		\begin{figure}
			\centering
			\begin{tikzpicture}[thick]
				\draw [blue] circle (3.2cm); 
				\draw node at (2.5, 2.5) {$\Gamma_R$};
				
				\pgfmathsetseed{3}	 
				\clip (0, 0) circle (3cm);
				\foreach \p in {1,...,20}
				{\fill [red] (3*rand, 3*rand) circle (0.06);
				}
			
				\pgfmathsetseed{8}			
				\draw plot [smooth cycle, samples=6, domain={1:6}] (\x*360/7+2*rnd:0.1cm+2cm*rnd) node at (0, -.5) {$D$}[fill=lightgray];
				
				\draw node at (1.5, 1.5) {$B_R$};
		   \end{tikzpicture}
			\caption{Illustration of the co-inversion for imaging the obstacle and source points. }\label{fig:Setup}
		\end{figure}
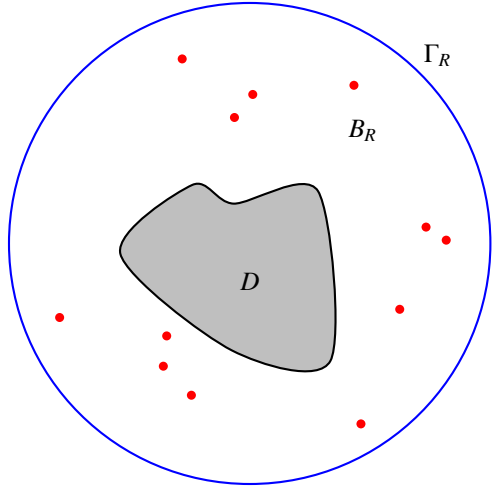
	
	The rest of this paper is organized as follows: In Section \ref{sec:optimization}, an optimization method is proposed to reconstruct the source-obstacle pair and the corresponding convergence theory is established. In Section \ref{sec:DSM}, an effective  strategy based on the sampling-type method is provided to select the initial guess for the optimization method.  Numerical experiments are given in Section \ref{sec:examples} to verify the performance of the method. Finally, some conclusions are given in Section \ref{sec:conclusions}.
		
	\section{Optimization method}\label{sec:optimization}
	
	Based on the decomposition method for inverse obstacle scattering problem \cite{Kirsch3}, we propose a novel optimization method in this section by incorporating the source locations as new variables in the cost functional in order to reconstruct the obstacle-source pair $(\partial D, S)$. The basic idea of the optimization method is to seek for the positions of source points and the boundary of obstacle as the locations where the Dirichlet boundary condition holds, namely, the total field $u^i+u^s$ vanishes. We would like to emphasize that, since the admissible source locations and obstacles are radically distinct, the new optimization problem will thereby encounter nontrivial difficulties in the analysis and implementation. 
		
	\subsection{Optimization method for the co-inversion problem}
	In this subsection, we present the optimization method for the inverse obstacle-source problem. According to the initial guess by the sampling results (see the next section for more details) about the obstacle-source pair $(\partial D,S),$ we can place a closed $C^2$ curve $\Lambda$ inside $D$ such that $k^2$ is not a Dirichlet eigenvalue for the negative Laplacian in the interior of $\Lambda$. Choosing some proper admissible set $U\times V $ of the obstacle-source pair $(\Gamma,z')$ such that $\Gamma\in U,$ $z'=(z'_1,z'_2,\cdots,z'_N)\in V,$ the scattered field $u^s(x;z'_j)$ corresponding to $z'_j$ can be represented by a single-layer potential
		\begin{align}\label{eq:single}
			u^s(x;z'_j)=\int_\Lambda\Phi(x,y)\varphi_j(y)\mathrm{d}s(y),
		\end{align}
		with the unknown density $\varphi_j\in L^2(\Lambda).$ Given the measured total field $u(x;z_j),$ the density $\varphi_j$ are supposed to satisfy the integral equation of the first kind
		\begin{align}\label{eq:density}
			(\mathcal{S}\varphi_j)(x)=u(x;z_j)-u^i(x;z'_j), \ \ x\in\Gamma_R,
		\end{align}
		where $u^i(x;z'_j)$ is the incident wave due to source point $z'_j$ and the compact integral operator $\mathcal{S}:L^2(\Lambda)\to L^2(\Gamma_R)$ is defined by
		\[
		(\mathcal{S}\varphi)(x):=\int_\Lambda\Phi(x,y)\varphi(y)\mathrm{d}s(y),\quad x\in\Gamma_R.
		\]
		The fact that the kernel of $\mathcal{S}$ is analytic leads to the severe ill-posedness of \eqref{eq:density}, which motivates us to seek for a stable numerical solution by adopting the Tikhonov regularization strategy to replace \eqref{eq:density} by
		\begin{align}\label{eq:reg_density}
			\alpha\varphi_j^\alpha + \mathcal{S}^*\mathcal{S}\varphi_j^\alpha=\mathcal{S}^*\Big(u(\cdot;z_j)-u^i(\cdot;z'_j)\Big),\quad{\rm on }\ \ \Gamma_R,
		\end{align}
		where $\alpha>0$ is a regularization parameter and the adjoint operator $\mathcal{S}^*:L^2(\Gamma_R)\to L^2(\Lambda)$ of $\mathcal{S}$ is given by
		\begin{align*}
			(\mathcal{S}^*g)(y):=\int_{\Gamma_R}\overline{\Phi(x,y)}g(x)\mathrm{d}s(x),\ \ y\in\Lambda,
		\end{align*}
		where the overline signifies the complex conjugate.
		
		Given the approximate scattered field $u^s_\alpha(x;z'_j)$ by inserting a solution $\varphi_j^\alpha$ of \eqref{eq:reg_density} into the single-layer potential representation \eqref{eq:single}, the obstacle-source pair $(\partial D,z)$ is then sought for by minimizing the defect
		\begin{align}\label{eq:defect}
			\sum_{j=1}^N\left\|u^i(x;z'_j)+u^s_\alpha(x;z'_j)\right\|_{L^2(\Gamma)}
		\end{align}
		over some suitable class $U\times V$ of admissible obstacle-source pairs $(\Gamma,z').$
		In \eqref{eq:defect}, $u^i(x;z'_j)$ is the incident field due to the source point $z'_j.$
		
		For the minimization problem, the admissible class $U$ of $\Gamma$ is chosen to be a compact set (with respect to the $C^{1,\beta}$ norm, $0<\beta<1$) of all star-like closed $C^2$ curves, described by
		\begin{align}\label{eq:Gamma}
			\Gamma=\{r(\hat{x})\hat{x}:\hat{x}\in\mathbb{S}\},
		\end{align}
		where $r\in C^2(\mathbb{S})$ is the radial function satisfying the prior information
		$$0<a\le r(\hat{x})\le b,\ \ \hat{x}\in\mathbb{S},$$
		with given constants $a$ and $b.$ The admissible class $V$ of $z'=(z'_1,z'_2,\cdots,z'_N)$ is chosen to be a compact set (with respect to the Euclidean norm in $\mathbb{R}^2$) of all $N$ ordered points such that $z'_j\in B_R\backslash \overline{D_\Gamma}, j=1,\cdots,N$, where $D_\Gamma$ is the domain enclosed by $\Gamma$ such that $\Gamma=\partial D_\Gamma$.
		
		\begin{remark}
			It deserves noting that, the purpose for choosing the admissible curve to be star-like is to provide convenience for the subsequent analysis and calculation, while the exact boundary curve $\partial D$ does not necessarily require to be star-like. In fact, even if $\partial D$ is not star-like, the obstacle-source pair may also be well reconstructed, which will be illustrated by the numerical experiments in Section \ref{sec:examples}.
		\end{remark}
	
\subsection{Convergence analysis}
		In this subsection, we will investigate the convergence properties of the optimization method with an analogous analysis to that of the decomposition methods \cite{Guo1,Colton}.
		
		For later use, we define the acoustic single-layer operator $\mathcal{S}_1:L^2(\Lambda)\to L^2(\Gamma)$ by
		\begin{align}\label{eq:S1}
			(\mathcal{S}_1\varphi)(x):=\int_\Lambda\Phi(x,y)\varphi(y)\mathrm{d}s(y),\quad x\in\Gamma,
		\end{align}
		where $\Gamma$ is a closed $C^2$ curve containing $\Lambda$ in its interior. To start with, we state some properties of the operators $\mathcal{S}$ and $\mathcal{S}_1$ (cf.\cite{Colton}).
		\begin{lem}\label{lem:single}
			The single-layer operators $\mathcal{S}$ and $\mathcal{S}_1$, defined by \eqref{eq:single}  and \eqref{eq:S1}, respectively, are injective and have dense range provided that $k^2$ is not a Dirichlet eigenvalue for the negative Laplacian in the interior of $\Lambda.$
		\end{lem}
		Combining the minimization of the Tikhonov functional for \eqref{eq:defect} and the defect minimization \eqref{eq:density}, we define the cost functional $\mu:(L^2(\Lambda))^N\times U\times V\to \mathbb{R}$ by
		\begin{align}\label{eq:cost}
			\mu(\varphi,\Gamma,z';\alpha)=&\sum_{j=1}^N \left\{\left\|u(x;z_j)-u^i(x;z'_j)-(\mathcal{S}\varphi_j
			)(x)\right\|^2_{L^2(\Gamma_R)}+\alpha\left\|\varphi_j(x)\right\|^2_{L^2(\Lambda)}\right.\nonumber\\&\left.+\left\|u^i(x;z'_j)+(\mathcal{S}_1\varphi_j)(x)\right\|^2_{L^2(\Gamma)}\right\},
		\end{align}
		where $\varphi=(\varphi_1,\cdots,\varphi_N)\in(L^2(\Lambda))^N,$ $z'=(z'_1,z'_2,\cdots,z'_N)\in V,$ $\alpha>0$ is the regularization parameter in \eqref{eq:reg_density}.
		
		Now, we turn to analyze the convergence properties of the optimization method. To this end, we first introduce the following definition of optimal obstacle-source pair.
		\begin{definition}\label{def:optimal}
			Given the measured total field $u(x;z_j),j=1,\cdots,N$ and a regularization parameter $\alpha>0,$ an obstacle-source pair $(\Gamma^*,z^*)\in U\times V$ is called optimal if there exists $\varphi^*\in (L^2(\Lambda))^N$ such that $\varphi^*$ and $(\Gamma^*, z^*)$ minimize the cost functional \eqref{eq:cost} simultaneously over all $\varphi\in(L^2(\Lambda))^N$ and $(\Gamma,z')\in U\times V,$ i.e.,
			\begin{equation}\label{eq:optimization}
				\mu(\varphi^*,\Gamma^*,z^*;\alpha)=m(\alpha),
			\end{equation}
			where
			\[
			m(\alpha)=\inf\limits_{\substack{\varphi\in(L^2(\Lambda))^N,\\\Gamma\in U,\, z'\in V}}\mu(\varphi,\Gamma,z';\alpha).
			\]
		\end{definition}
	In terms of Definition \ref{def:optimal}, the following theorem establishes the existence of the optimal solution.
		\begin{thm}
			For each $\alpha>0,$ there exists an optimal obstacle-source pair $(\Gamma^*, z^*)\in U\times V.$
		\end{thm}
		\begin{proof}
			Let $(\varphi^n,\Gamma^n,z^n)$ be a minimizing sequence in $\Big(L^2(\Lambda)\Big)^N\times U\times V,$ i.e.,
			\[
			\lim\limits_{n\to\infty}\mu(\varphi^n,\Gamma^n,z^n;\alpha)=m(\alpha).
			\]
			Since $U,V$ are compact, we can assume that $\Gamma^n\to\Gamma^*,$ $z^n\to z^*,$ $n\to\infty.$ In this paper, the convergence $\Gamma^n\to\Gamma^*$ is understood in the sense that $\|r^n-r^*\|_{C^{1,\beta}(\mathbb{S})}\to0,$ $n\to\infty,$ with the functions
			$r^n$ and $r^*$ representing $\Gamma^n$ and $\Gamma^*$ via \eqref{eq:Gamma}, respectively. From
			\[
			\alpha\|\varphi^n\|_{(L^2(\Lambda))^N}\le\mu(\varphi^n,\Gamma^n,z^n;\alpha)\to m(\alpha),\ \ n\to\infty,
			\]
			with $\alpha>0,$ we conclude that the sequence $\{\varphi^n\}$ is bounded, i.e., there exists some constant $C$ such that $\|\varphi^n\|_{(L^2(\Lambda))^N}\le C,\forall n.$
			Further, we can assume that $\varphi^n\rightharpoonup \varphi^*\in(L^2(\Lambda))^N,n\to\infty.$ Since $\mathcal{S}$ is compact, it follows that $\mathcal{S}\varphi_j^n\to\mathcal{S}\varphi_j^*,n\to\infty,j=1,\cdots,N.$ The fact that the Hankel function $H_0^{(1)}$ is continuous leads to the following convergence
			\[
			u^i(x;z_j^n)=\frac{\mathrm{i}}{4}H_0^{(1)}(k|x-z_j^n|)\to\frac{\mathrm{i}}{4}H_0^{(1)}(k|x-z_j^*|)
			=u^i(x;z_j^*),\ \ j=1,\cdots,N.
			\]
			Hence,
			\[
			\left\|u(x;z_j)-u^i(x;z_j^n)-(\mathcal{S}\varphi_j^n)(x)\right\|_{L^2(\Gamma_R)}\to
			\left\|u(x;z_j)-u^i(x;z_j^*)-(\mathcal{S}\varphi_j^*)(x)\right\|_{L^2(\Gamma_R)},\ \ n\to\infty.
			\]
			By Taylor's formula, one can estimate that for all $\hat{x}\in\mathbb{S},$
			\[
			\left|\Phi(r^n(\hat{x})\hat{x},y)-\Phi(r^*(\hat{x})\hat{x},y)\right|\le K|r^n(\hat{x})-r^*(\hat{x})|,
			\]
			with
			$$K=\sup_{U\times\Lambda}\left|\nabla_x\Phi(x, y)\right|.$$
			From the Schwarz inequality, it holds that
			\[
			\left|\int_\Lambda\left[\Phi(r^n(\hat{x})\hat{x},y)-\Phi(r^*(\hat{x})\hat{x},y)\right]\varphi_j^n(y)\mathrm{d}s(y)\right|
			\le CK|\Lambda||r_n(\hat{x})-r^*(\hat{x})|,
			\]
			for all $\hat{x}\in\mathbb{S}.$ Denote the dependence of $\mathcal{S}_1:L^2(\Lambda)\to L^2(\Gamma^n)$ on $n$ by writing $\mathcal{S}_1^n,$ and the uniform convergence $\mathcal{S}_1^n\varphi_j^n\to\mathcal{S}_1\varphi^*_j(n\to\infty$) holds. Therefore, we can deduce that
			\[
			\left\|u^i(x;z_j^n)+(\mathcal{S}_1^n\varphi_j^n)(x)\right\|^2_{L^2(\Gamma^n)}\to
			\left\|u^i(x;z_j^*)+(\mathcal{S}_1\varphi_j^*)(x)\right\|^2_{L^2(\Gamma^*)},\quad n\to\infty.
			\]
			Further,
			\begin{align*}
				\alpha\sum_{j=1}^N\left\|\varphi_j^n\right\|^2_{L^2(\Lambda)}\to & m(\alpha)-\sum_{j=1}^N\left\|u(x;z_j)-u^i(x;z_j^*)-(\mathcal{S}\varphi_j^*)(x)\right\|^2_{L^2(\Gamma_R)}                                                                  \\
				& -\sum_{j=1}^N\left\|u^i(x;z_j^*)+(\mathcal{S}_1\varphi_j^*)(x)\right\|_{L^2(\Gamma^*)}^2
				\le\alpha\sum_{j=1}^N\left\|\varphi_j^*\right\|^2_{L^2(\Lambda)},\ \ n\to\infty.
			\end{align*}
			Since $\varphi_j^n\rightharpoonup \varphi_j^*,n\to\infty,j=1,\cdots,N,$ it follows that
			\[
			\lim\limits_{n\to\infty}\sum_{j=1}^N\left\|\varphi_j^n-\varphi_j^*\right\|^2_{L^2(\Lambda)}
			=\lim\limits_{n\to\infty}\sum_{j=1}^N\left\|\varphi_j^n\right\|^2_{L^2(\Lambda)}
			-\sum_{j=1}^N\left\|\varphi_j^*\right\|^2_{L^2(\Lambda)}\le0.
			\]
			Thus, we have the norm convergence $\left\|\varphi_j^n-\varphi_j^*\right\|\to0, n\to\infty$ for $j=1,\cdots,N.$ Finally, it holds that
			\[
			\mu\Big(\varphi^*,\Gamma^*,z^*;\alpha\Big)=\lim\limits_{n\to\infty}\mu\Big(\varphi^n,\Gamma^n,z^n;\alpha\Big)=m(\alpha),
			\]
			and the proof is complete.
		\end{proof}
		Before heading for the convergence analysis of the optimization method, we shall state the convergence property of the cost functional first.
		\begin{thm}\label{thm:convergent}
			Let $z_j,$ $j=1,\cdots,N$ be the exact source points, $u(x;z_j)$ be the exact total fields corresponding to the sound-soft obstacle $D$ such that $\partial D$ belongs to $U.$ Then the cost functional $m(\alpha)$ is convergent in the sense that
			\begin{align}
				\lim\limits_{\alpha\to0}m(\alpha)=0.
			\end{align}
		\end{thm}
		\begin{proof}
			By Lemma \ref{lem:single}, given $\varepsilon_j>0,$ there exist $z'_j\in V,\varphi_j\in L^2(\Lambda)$ such that
			\[
			\left\|u^i(x;z'_j)+(\mathcal{S}_1\varphi_j)(x)\right\|_{L^2(\partial D)}\leq\varepsilon_j,\ \ j=1,\cdots,N.
			\]
			Since the scattered field of a radiating solution to the Helmholtz equation depends continuously on the boundary (cf. \cite{Colton}), it holds that
			\begin{align*}
				\left\|u(x; z_j)-u^i(x; z'_j)-(\mathcal{S}\varphi_j)(x)\right\|_{L^2(\Gamma_R)} & \le C\left\|u(x; z_j)-u^i(x; z'_j)-(\mathcal{S}_1\varphi_j)(x)\right\|_{L^2(\partial D)} \\
				& =C\left\|(\mathcal{S}_1\varphi_j)(x)+u^i(x;z'_j)\right\|_{L^2(\partial D)}             \\
				& \le C\varepsilon_j,
			\end{align*}
			where we have used the fact that $u(x;z_j)=0,x\in\partial D.$ Further, we deduce that
			\[
			\mu(\varphi,\partial D,z;\alpha)\le(1+C^2)\sum_{j=1}^N\varepsilon_j^2+\alpha\sum_{j=1}^N\left\|\varphi_j\right\|_{L^2(\Lambda)}^2\to
			(1+C^2)\sum_{j=1}^N\varepsilon_j^2,\ \ \alpha\to0.
			\]
			The theorem is accomplished due to the arbitrariness of each $\varepsilon_j, j=1,2,\cdots,N.$
		\end{proof}
		
		We are now in a position to state the main convergence result.
		\begin{thm}\label{thm:convergence}
			Let $\{\alpha_n\}$ be a null sequence, $(\Gamma^n,z^n)$ be the corresponding sequence of optimal obstacle-source pairs for the regularization parameter $\alpha_n.$ Then there exists a convergent subsequence of $\left\{(\Gamma^n,z^n)\right\}.$ Assume that $u$ is the exact total field of a domain $D$ such that $\partial D$ is contained in $U.$ Then every limit point $(\Gamma^*,z^*)$ of $(\Gamma^n, z^n)$ represents an optimal obstacle-source pair. In other words, if the incident wave due to the source point $z^*$ is scattered by the sound-soft obstacle $D^*$ bounded by $\Gamma^*$, then the total field vanishes on $\Gamma^*$.
		\end{thm}
		\begin{proof}
			The existence of convergent subsequences of $(\Gamma^n,z^n)$ follows from the compactness of $U\times V$. Let $(\Gamma^*,z^*)$ be the limit point. Without loss of generality, we assume that $(\Gamma^n,z^n)\to(\Gamma^*,z^*),n\to\infty,$ i.e., $\Gamma^n\to\Gamma^*,$ $z^n\to z^*,$ $n\to\infty.$ Let $u_j^*$ be the solution to the direct scattering problem with the incident wave emitted by the source point $z_j^*$ for the obstacle with boundary $\Gamma^*,$ i.e., the boundary condition reads
			\begin{align}\label{29}
				u_j^*(\cdot)+u^i(\cdot; z_j^*)=0,\quad\text{on}\ \Gamma^*.
			\end{align}
			Since $(\Gamma^n, z^n)$ is optimal for the parameter $\alpha_n$, there exists $\varphi^n\in (L^2(\Lambda))^N$ such that
			$$
			\mu(\varphi^n,\Gamma^n, z^n;\alpha_n)=m(\alpha_n),
			$$
			for $n=1,2,\cdots$. From Theorem \ref{thm:convergent}, we conclude that the boundary data satisfies
			\begin{align}\label{30}
				&\left\|u^i(x;z^n_j)+(\mathcal{S}^n_1\varphi^n_j)(x)\right\|_{L^2(\Gamma^n)}\to0,\quad n\to\infty.
			\end{align}
			From the condition that $\Gamma^n\to\Gamma^*,z_j^n\to z_j^*,n\to\infty,$
			we obtain that $u_j^s=u_j^*$
			with $u_j^s=\lim\limits_{n\to\infty}\mathcal{S}_1^n\varphi_j^n,$ and  the total field $u^s_j(\cdot)+u^i(\cdot; z^*_j)$ must vanish on $\Gamma^*$ as a result of \eqref{29}.
		\end{proof}
	    
	    Based on the previous discussions, the rudiments of our optimization procedure is summarized in {\bf Algorithm 1}.
	\begin{table}[htp]
		\centering
		\begin{tabular}{cp{.8\textwidth}}
			\toprule
			\multicolumn{2}{l}{{\bf Algorithm 1:}\quad Optimization method for the co-inversion problem.} \\
			\midrule
			{\bf Step 1} &  Given the total field data $u(x; z_j)|_{\Gamma_R}, j=1,\cdots,N$, choose some suitable class $U\times V$ of admissible obstacle-source pairs $(\partial D, S)$; \\
			{\bf Step 2} &  Select an initial curve for the boundary $\partial D$ and an initial guess for the $N$ ordered points $z'=(z'_1,z'_2,\cdots,z'_N)$; \\
			{\bf Step 3} & Place a suitable auxiliary curve $\Lambda$ inside $D$;\\
			{\bf Step 4} & Determine the obstacle-source pair $(\partial D, S)$ by minimizing \eqref{eq:defect}, which is formulated by the following implementations for $j=1,\cdots,N$:
			\begin{itemize}
				\item[(a)] Subtract the incident field $u^i(x; z'_j)$ due to the source point $z'_j$ from the measured total field $u(x; z_j)$, where the incident field $u^i(x; z'_j)$ is given by \eqref{eq:fundamental_solution};
				\item[(b)] Solve the regularized equation \eqref{eq:reg_density} with an appropriate regularization parameter $\alpha$ to find the density $\varphi_j^\alpha$;
				\item[(c)] Represent the approximate scattered field $u_\alpha^s(x;z'_j)$ in terms of the single-layer representation \eqref{eq:single} with density $\varphi_j^\alpha$.
			\end{itemize} \\
		   {\bf Step 5 } & When an error tolerance is fulfilled, the minimizer $(\Gamma^*, z^*)$ of \eqref{eq:defect} can be viewed as the desired reconstruction.\\
			\bottomrule
		\end{tabular}
	\end{table}
	    
	    At the end of this section, we would like to remark that, the nonlinear optimization problems generally suffer from the notorious difficulty of local minimums. Specifically, performance of any iterative solver for the optimization problem \eqref{eq:optimization} would be unavoidably affected by the initial guess $(\Gamma^0, z^0)$. Therefore, choosing an appropriate initial value is highly crucial for the success of our optimization method. Nevertheless, the a priori information of the rough $(\Gamma^0, z^0)$ is not straightforwardly available. And it is thus vital to develop a fast and cheap strategy for resolving this crux, which will be the main concern of the next section. 
	
	\section{Direct sampling for choosing the initial guess}\label{sec:DSM}
	
	In the previous section, we develop an optimization method to reconstruct the obstacle-source pair $(\partial D, S)$ from the total field $\mathbb{U}$ and it has been pointed out that the initial guess $(\Gamma^0, z^0)$ should play an important role in the quality of reconstruction. Noticing that the direct sampling methods (DSM) are independent of any a priori information of the unknown objects and easy to implement, we wish to utilize the sampling-type methods to seek the initial guess for the source $S$ and the obstacle $D$, respectively.
	
	\subsection{Determine the initial guess of source points}\label{subsec4.1}
	
	Motivated by the direct sampling schemes for inverse obstacle/source scattering problems, we first propose the {\it total field} version of indicator functions to reconstruct the source points by knowledge of $\mathbb{U}$.
	Let $\Omega_1\subset B_R$ be a bounded sampling domain such that $D\cup S\subset\Omega_1.$ For any sampling point $y\in\Omega_1$, we define $N$ indicator functions as follows:
	\begin{equation}\label{eq:Ij}
		I_j(y)=\left|\int_{\Gamma_R}\overline{u(x; z_j)}\Phi(x, y)\mathrm{d}s(x)\right|,\ \ j=1,\cdots, N.
	\end{equation}
	Note that the indicators rely on the inner product of the fundamental solution with respect to the total field data rather than the usual incident or scattered data. We expect to locate the peak of $I_j(y)$ for each $j$ and take the corresponding maximum point $\tilde{z}_j$ as an approximation of the exact source $z_j$.
		
	Before analyzing the characteristics of $I_j(y),$ we first claim two lemmas which will play a key role in the subsequent analysis.
		
	\begin{lem}\label{lem:4.1}\cite[Lemma 3.1]{Chen}
			Let $g \in H^{1 / 2}\left(\partial{D}\right).$ Then the scattering problem
			\begin{align}
				\label{eq3.2}\Delta w+k^{2} w= & 0 \quad {\rm in}\ \ \mathbb{R}^2 \backslash \overline{D}, \\
				\label{eq3.3}w= & g \quad {\rm on }\ \ \partial{D}, \\
				\label{eq3.4}\sqrt{r}\left(\frac{\partial w}{\partial r}-\mathrm{i} k w\right) & \rightarrow 0 \quad {\rm as }\ \ r \rightarrow \infty,
			\end{align}
			admits a unique solution $w \in H_{\mathrm{loc}}^{1}\left(\mathbb{R}^2 \backslash \overline{D}\right) $. Moreover, there exists a constant $C>0$ such that
			$$
			\left\|\frac{\partial w }{\partial \nu}\right\|_{H^{-1 / 2}\left(\partial{D}\right)} \leq C\|g\|_{H^{1 / 2}\left(\partial{D}\right)},
			$$
			with $\nu$ the unit outer normal to the boundary $\partial D$.
	\end{lem}
		
	For any $g \in H^{1 / 2}\left(\partial{D}\right)$, the Dirichlet-to-Neumann mapping $\mathbb{T}: H^{1 / 2}\left(\partial{D}\right) \rightarrow H^{-1 / 2}\left(\partial{D}\right)$ for the scattering problem \eqref{eq3.2}-\eqref{eq3.4} is defined as 
	$$ 
		\mathbb{T}(g)=\frac{\partial w}{\partial \nu}.
	$$ 
	By Lemma \ref{lem:4.1}, $\mathbb{T}$ is a bounded linear operator and we will denote $\|\mathbb{T}\|$ its operator norm in this paper.
		
	Analogous to \cite[Lemma 3.2]{Chen}, one can immediately obtain the following estimate concerning the scattered field.
	\begin{lem}\label{lem:us_estimation}
			Let $L:=\min_{1\le j\le N}{\rm dist}\left(z_j,\overline{D}\right)$. Then for each $j=1,\cdots,N$, the following estimate holds
			\[
			|u^s(x; z_j)|\le C_{js}(1+\|\mathbb{T}\|)k^{-1}(RL)^{-1/2},
			\]
			where $x\in\Gamma_R$ is the receiver and the constant $C_{js}$ is dependent of $k, R$ and $L$.
	\end{lem}
	
	The following theorem sheds light on the indicating properties of $I_j(y), j=1,\cdots,N$. 
		\begin{thm}
			Let  $I_j(y), j=1,\cdots,N,$  be defined by \eqref{eq:Ij} and assume that 
			\begin{equation}\label{eq:assumption}
				L_j:={\rm dist}(z_j, \Gamma_R)>\sup_{y\in\Omega_1}|y-z_j|.
			\end{equation}
			Then there exists a constant $C_j>0$ such that for each $y\in \Omega_1,\ \ y\ne z_j,$
			\begin{align}
				I_j(y)\le C_j k^{-1}|y-z_j|^{-1/2},\quad j=1,\cdots,N.
			\end{align}
		\end{thm}
		\begin{proof}
		We first recall the following estimate for the Hankel function \cite[(3.8)]{Chen}:
		\begin{align}
			\left|H_0^{(1)}(t)\right|\le\sqrt{\frac{2}{\pi t}},\ \ \forall t>0.
		\end{align}
	Then from the representation of incident field $u^i(x; z_j),$ it can be easily derived that
			\begin{align}
				\nonumber\int_{\Gamma_R}\left|\overline{u^i(x;z_j)}\Phi(x,y)\right|\mathrm{d}s(x) & =\int_{\Gamma_R}\left|\overline{\frac{\mathrm{i}}{4}H_0^{(1)}(k|x-z_j|)}\frac{\mathrm{i}}{4}H_0^{(1)}(k|x-y|)\right|\mathrm{d}s(x) \\
				& \nonumber\le\frac{1}{16}\int_{\Gamma_R}\sqrt{\frac{2}{\pi k|x-z_j|}}\sqrt{\frac{2}{\pi k|x-y|}}\mathrm{d}s(x) \\
				& =\frac{1}{8k\pi}\int_{\Gamma_R}\frac{1}{\sqrt{|x-z_j||x-y|}}\mathrm{d}s(x). \label{eq:ui}
			\end{align}
		
	Further, by Lemma \ref{lem:us_estimation}, we deduce that there is a positive constant $C_{js}$ such that
			\begin{align}
				\nonumber\int_{\Gamma_R}\left|\overline{u^s(x; z_j)}\Phi(x, y)\right|\mathrm{d}s(x) & =\int_{\Gamma_R}\left|{u^s(x; z_j)}\right|\left|\frac{\mathrm{i}}{4}H_0^{(1)}(k|x-y|)\right|\mathrm{d}s(x) \\
				& \le \frac{C_{js}\sqrt{2}}{4k^{3/2}(\pi RL)^{1/2}}\int_{\Gamma_R}\frac{1+\|T\|}{\sqrt{|x-y|}}\mathrm{d}s(x). \label{eq:us}
			\end{align}	
	Combining \eqref{eq:ui} and \eqref{eq:us}, together with the aid of $|x-y|\ge\left||x-z_j|-|y-z_j|\right|$, we have 
	\begin{align*}
		I_j(y)  & =  \left|\int_{\Gamma_R}\overline{u^i(x;z_j)}\Phi(x,y)\mathrm{d}s(x)+\int_{\Gamma_R}\overline{u^s(x;z_j)}\Phi(x,y)\mathrm{d}s(x)\right| \\
		        & \le  \int_{\Gamma_R}\left|\overline{u^i(x;z_j)}\Phi(x,y)\right|\mathrm{d}s(x)+\int_{\Gamma_R}\left|\overline{u^s(x;z_j)}\Phi(x,y)\right|\mathrm{d}s(x) \\
		        & \le  C_{j1}\left(\int_{\Gamma_R}k^{-1}(|x-z_j||x-y|)^{-1/2}\mathrm{d}s(x)+\int_{\Gamma_R}k^{-3/2}(1+\|T\|){|x-y|}^{-1/2}\mathrm{d}s(x)\right)\\
		       & \le   C_{j1}\Bigg(\int_{\Gamma_R}{k^{-1}\Big(|x-z_j|\left||x-z_j|-|y-z_j|\right|\Big)^{-1/2}}\mathrm{d}s(x)\\
		       &\quad +\int_{\Gamma_R}k^{-3/2}(1+\|\mathbb{T}\|)\left||x-z_j|-|y-z_j|\right|^{-1/2}\mathrm{d}s(x)\Bigg),
	\end{align*}
	where
	$$
	C_{j1}=\max\left\{\frac{1}{8\pi}, \frac{C_{js}\sqrt{2}}{4(\pi R L)^{1/2}}\right\}.
	$$
	
	Assumption \eqref{eq:assumption} implies that there exists a positive constant $C_{j2}$ such that $L_j\ge (1+C_{j2})\sup_{y\in\Omega_1}|y-z_j|$. Then
	\begin{align*}
		I_j(y) & \le \frac{2\pi C_{j1}R}{k\sqrt{L_j}\sqrt{C_{j2}|y-z_j|}}+\frac{2C_{j1}\pi(1+\|\mathbb{T}\|)R}{k^{3/2}\sqrt{C_{j2}|y-z_j|}} \\
		& \le C_j k^{-1}|y-z_j|^{-1/2},
	\end{align*}
	where
	$$
	C_j=\frac{2\pi R C_{j1}}{\sqrt{C_{j2}}}\left(\frac{1}{\sqrt{L_j}}+\frac{1+\|\mathbb{T}\|}{\sqrt{k}}\right),
	$$ 
	and this completes the proof.
	\end{proof}

	In virtue of the above theorem, each function $I_j(y)$ should decay as the sampling point $y$ recedes from the corresponding source point $z_j$. Hence,  the initial detection can be conveniently chosen by locating the point $\tilde{z}_j$ where $I_j(y)$ reaches its apex. Then in the optimization process, the exact source location $z_j$ could be effortlessly and accurately found in the vicinity of $\tilde{z}_j$. This indicating behavior will be numerically verified by the experiments in the next section.
	
	\subsection{Determine the initial guess of the obstacle}\label{sec:approximate reverse time migration method}
	The indicator function proposed in the previous subsection has the capability of locating the approximate source points, but it usually fails to image the shape of obstacle. Thus a further sampling step for identifying the rough obstacle should be taken into account. By integrating the approximate source points into the reverse time migration (RTM) approach \cite{Chen1}, we propose an approximate reverse time migration method to reconstruct the obstacle in this subsection.
		
	To reconstruct the obstacle, information about the scattered field is indispensable in the classical RTM. However, the scattered field can not be measured directly due to the existence of unknown source points. So we are going to approximate the scattered field by subtracting the incident field $u^i(x; \tilde{z}_j)$ due to the numerical source point $\tilde{z}_j$ from the measured total field $u(x; z_j),$  i.e.,
	\begin{align}\label{eq:us_j}
		u_j^s(x)=u(x;z_j)-u^i(x;\tilde{z}_j),\quad x\in\Gamma_R,
	\end{align}
	and develop an imaging functional $I_D(y)$ using the approximate scattered field $u_j^s(x)$. To this end, confining the imaging domain to $\Omega_2\subset\Omega_1$ such that $\overline{D}\subset \Omega_2,$ $S\subset\Omega_1\backslash\overline{\Omega}_2,$ and we define the imaging function for $y\in\Omega_2$ by
	\begin{align}\label{eq:I_D}
		I_D(y):=-k^2{\rm Im}\left(\sum_{j=1}^N\Phi(y,z_j)\int_{\Gamma_R}\overline{u^s_j(x)}\Phi(x,y)\mathrm{d}s(x)\right).
	\end{align}
	Note that, in \eqref{eq:I_D}, instead of the exact scattered field $u^s(x;z_j)$ corresponding to the $j$-th source point $z_j,$ we use an approximate scattered field $u^s_j(x)$ for $z_j$. That is why this method is called the approximate reverse time migration method. Once the source points are retrieved, the indicating behavior of the approximate RTM immediately follows the original RTM \cite{Chen, Chen1}.  The imaging function $I_D(y)$ provides a qualitative information of the obstacle $D$,  which will be used as a useful clue to discover the obstacle in the optimization method. 
		
	We end this section with a brief description of {\bf Algorithm 2}, which can be used to determine the initial guess for the optimization method proposed in Section \ref{sec:optimization}.
	\begin{table}[htp]
		\centering
		\begin{tabular}{cp{.8\textwidth}}
			\toprule
			\multicolumn{2}{l}{{\bf Algorithm 2:}\quad Determine the initial guess by the direct sampling methods.} \\
			\midrule
			{\bf Step 1} &  For each $j=1,\cdots,N$, collect the total field data $u(x; z_j)|_{\Gamma_R}$; \\
			{\bf Step 2} &  Choose a proper sampling domain $\Omega_1\subset B_R$ such that $D\cup S\subset \Omega_1$ and generate the sampling grid $\mathcal{T}_1$ over $\Omega_1$;\\
			{\bf Step 3} & Locate the maximum of $I_j(y)$ for each $j$ and take the corresponding $\tilde{z}_j$ as an approximation of $z_j$. The locations $\{\tilde{z}_j\}_{j=1}^N$ could be found one by one in this way and will be reused as the initial guess for the source points in the optimization method; \\
			{\bf Step 4} &  For $j=1,\ldots, N$, compute the approximate scattered field $u^s_j$ by \eqref{eq:us_j}; \\
			{\bf Step 5} & Choose a proper imaging domain $\Omega_2\subset\Omega_1$ such that $\overline{D}\subset\Omega_2$, $S\subset \Omega_1\backslash\overline{\Omega}_2$ and generate the sampling grid $\mathcal{T}_2$ over $\Omega_2$; \\
			{\bf Step 6} & Evaluate $I_D(y)$ for the imaging point $y\in\mathcal{T}_2$ by \eqref{eq:I_D} and accordingly select an approximate initial guess for the boundary curve in the optimization method. \\
			\bottomrule
		\end{tabular}
	\end{table}
		
	\section{Numerical experiments}\label{sec:examples}
	In this section, we will present several numerical experiments to demonstrate the performance of the proposed method. In our experiments, synthetic total field data are generated by reformulating the direct problem \eqref{eq:Helmholtz}-\eqref{eq:Sommerfeld} as a boundary integral equation and the integral equation is solved by the Nystr\"{o}m method \cite{Colton}. Here we use a numerical quadrature rule with $64$ equidistant grid points on $[0,2\pi]$. The receivers are set to be $x_r=4(\cos\theta_r,\sin\theta_r),$ $\theta_r={r\theta}/{N_R},$ $r=1,\cdots,N_R$ with $\theta\in (0, 2\pi]$ the observation aperture. For the full-aperture case, $\theta=2\pi,$ $N_R=120$. For the limited-aperture case, we consider two cases with $\theta=3\pi/2,$ $N_R=90$ and $\theta=\pi,$ $N_R=60$, respectively.
	The forward solver provides us with the total field data
	\[
	u(x_r; z_j),\quad r=1,\cdots,N_R,\quad j=1,\cdots,N.
	\]
	
	To test the stability of the numerical method, random noise of different levels is added to the total field data $u(x; z_j)$ to produce the noisy total field data
	\[
	u^\varepsilon:=u+\varepsilon r_1|u|\mathrm{e}^{\mathrm{i}\pi r_2},
	\]
	where $r_1,$ $r_2$ are two uniformly distributed random numbers ranging from $-1$ to $1$ and $\varepsilon>0$ is the noise level. We use $\varepsilon=10\%$ in Examples \ref{example1}-\ref{example3} and $\varepsilon=5\%$ in Example \ref{example4}.
		
	In all the subsequent examples, the sampling domains $\Omega_1$ are chosen as squares centered at the origin with $\mathcal{T}_1$ being a $200\times200$ uniform grid. The imaging domains $\Omega_2$ are individualized case by case, according to the sampling results for the point sources. The imaging points $\mathcal{T}_2$ are selected as $100\times100$ uniform grids distributed over $\Omega_2$.	
	For the optimization method, we choose the auxiliary curve $\Lambda$ to be a circle centered at the origin with radius $r_\Lambda=0.6.$ The integral over the unit circle $\mathbb{S}$ is numerically approximated by the trapezoidal rule with $64$ grid points. To solve the nonlinear least-squares problem, the Levenberg-Marquardt algorithm is adopted with both functional value stopping criterion and successive iterate stopping criterion chosen to be $10^{-6}$ for the full-aperture case and $10^{-4}$ for the limited-aperture problems. The starting curve $\Gamma_0$ for the Levenberg-Marquardt algorithm is a circle centered at the origin with radius $r_{\Gamma_0}=\left|\arg\max_{y\in\Omega_2}I_D(y)\right|$. For $j=1,\cdots,N$, the starting guesses for the source points $\{z_j^0\}_{j=1}^N$ are chosen to be $\{\tilde{z}_j\}_{j=1}^N$. The regularization parameter $\alpha$ is selected by trial and error. The ansatz function $r(t)$ is chosen in the $(2M+1)$-dimensional subspace $U_M\subset U$ and can be expanded by trigonometric polynomials of degree less than or equal to $M,$ i.e.,
	\begin{align}\label{eq:rM}
		r_M(t)=r_\Lambda+a_0+\sum_{m=1}^M(a_m\cos mt+b_m\sin mt),
	\end{align}
	with some $M\in\mathbb{N}_+$. Unless otherwise specified, $M=8$ is used in what follows. 
		
	To evaluate the reconstruction quantitatively, we introduce an equidistant set of knots  on $[0,2\pi]$ by $t_i:=2\pi i/N_1,i=0,1,\cdots,N_1-1,$ with $N_1$ the number of knots chosen to be 256 in our numerical experiments. The discrete relative $L^2$ error of the boundary curve is defined by
	\begin{align}\label{eq:error1}
		E_D:=\frac{\left(\sum\limits_{i=1}^{N_1}|r_M(\tau_i)-r^*(\tau_i)|^2\right)^{1/2}}{\left(\sum\limits_{i=1}^{N_1}|r^*(\tau_i)|^2\right)^{1/2}},
	\end{align}
	with $r^*$ standing for the exact boundary curve and $\tau_i=\tau(t_i)$ being $t_i$ or some transformation of $t_i$ depending on whether the exact boundary is star-like or not, respectively.
		
	In the following figures regarding geometrical settings of the problem, the red and blue points denote respectively the exact source locations and receivers, the green solid line denotes the boundary of the obstacle, the black and purple dashed lines denote the boundary of the sampling domains $\Omega_1$ and $\Omega_2,$ respectively. In the figures about the results of the numerical experiments, the small black `+' markers are the exact source points, the small red circles stand for the reconstructed source points, the green solid line denotes the exact boundary of the obstacle, the blue dashed line denotes the reconstructed boundary.
	In addition, for ease of comparison, the indicator functions $I_j(y),$ $j=1,\cdots,N,$ and the imaging function $I_D(y)$ are normalized so that the maximum value of each function is $1.$
		
	\begin{example}\label{example1}
	In the first example, we consider the reconstruction of a starfish-shaped obstacle and the influence of source quantities. The exact boundary of the starfish-shaped obstacle is parameterized by
	$$
		x_S(t)=(1+0.2\cos 5t)(\cos t,\sin t),\quad0\le t\le2\pi.
	$$
	
	We first consider the case where the five source points are respectively located at 
	$$
	S_1=\{(2,2),(-0.5,1.8),(-1.6,-1),(0,-2.2),(2,-1.2)\}.
	$$ 
	
	The sampling domains are chosen to be $\Omega_1=[-2.5,2.5]\times[-2.5,2.5]$, $\Omega_2=[-1.5,1.5]\times[-1.5,1.5]$. The regularization parameter and the wave number are set to be $\alpha=10^{-8}$ and $k=5$, respectively. 
			
	Figure \ref{fig:starfish2}(a) exhibits the geometry setting. Figure \ref{fig:starfish2}(b) shows the reconstructed locations of $S_1$ by the DSM and we can find that the reconstructed locations of the source points are very close to the exact locations. Figure \ref{fig:starfish2}(c) shows the image of the imaging function $I_D(y)$ over the sampling domain $\Omega_2.$ We find that in Figure \ref{fig:starfish2}(c), the shape of the obstacle can be roughly captured.
	In Figure \ref{fig:starfish2}(d), reconstruction of the optimization method is depicted, where the initial guess for the source points are chosen to be the numerical source points obtained by the DSM and the initial curve is represented by the black dashed circle in \ref{fig:starfish2}(c). Further, we compute the error of boundary curve by \eqref{eq:error1} with $\tau_i=t_i,i=1,\cdots,N$ and in this case $E_D=5.73\%$.		
	In Table \ref{tab:star_5'_location}, we list the exact and  reconstructed locations of source points $S_1$ by the DSM and the optimization method, respectively. According to Table \ref{tab:star_5'_location} and Figure \ref{fig:starfish2}, we want to point out that both the DSM and optimization method could produce satisfactory reconstructions of the source points. A promising advantage of the optimization method is that the obstacle can also be reconstructed quantitatively, i.e, a significantly more sharper profile can be identified by the optimization method.
	
	\begin{figure}[htpb]
		\centering
		\subfigure[]{\includegraphics[width=0.24\linewidth]{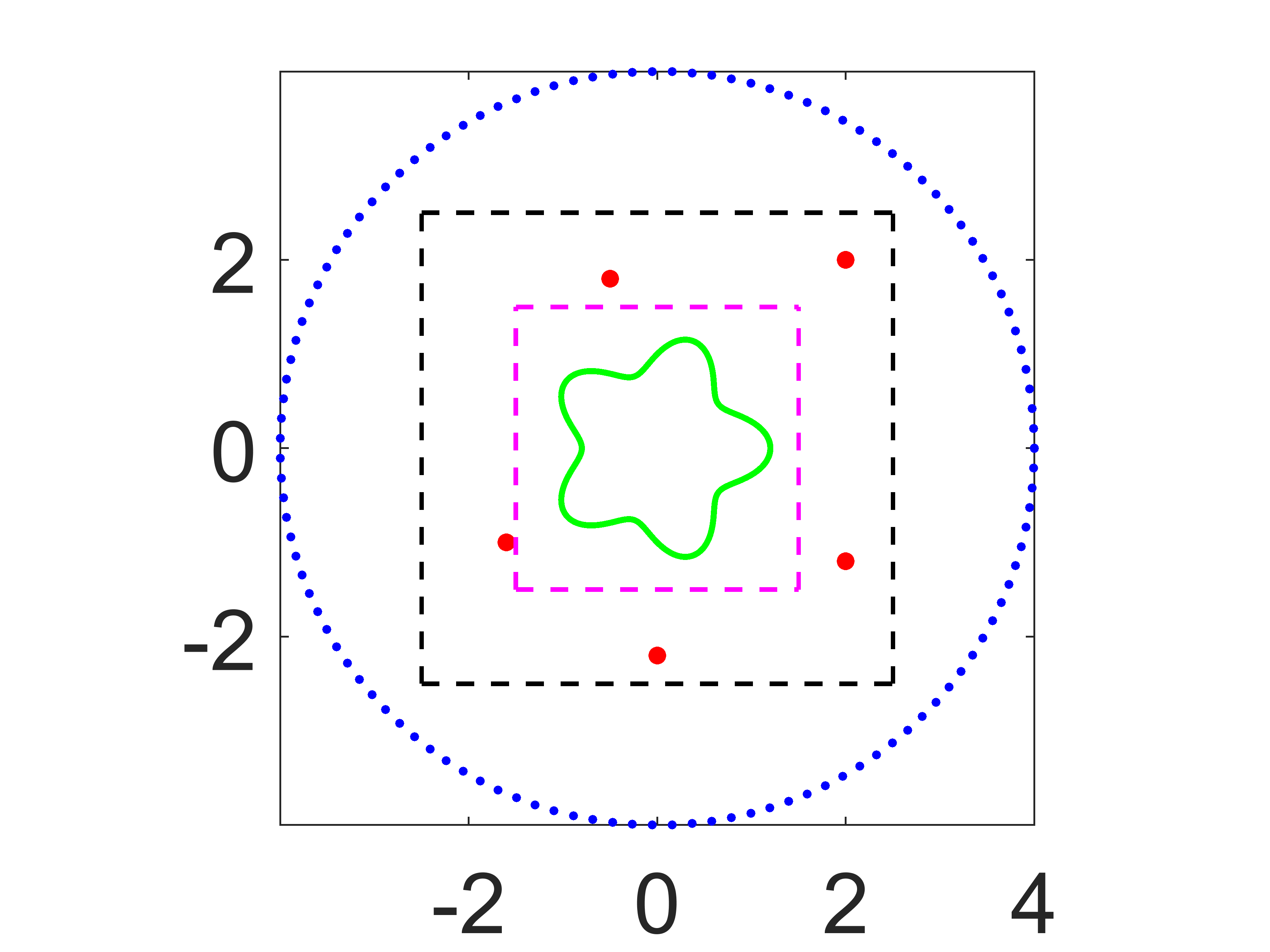}}
		\subfigure[]{\includegraphics[width=0.24\linewidth]{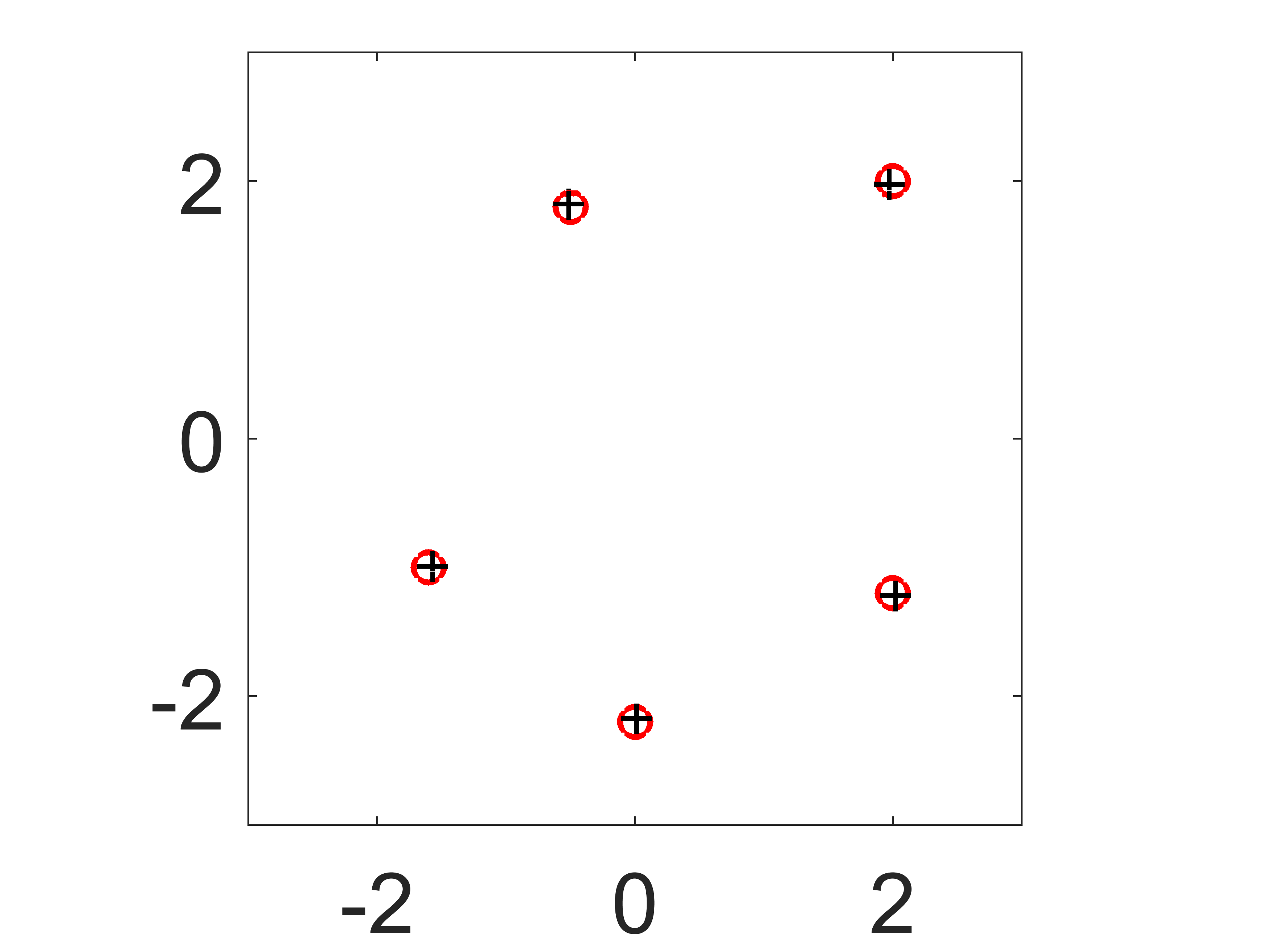}}
		\subfigure[]{\includegraphics[width=0.24\linewidth]{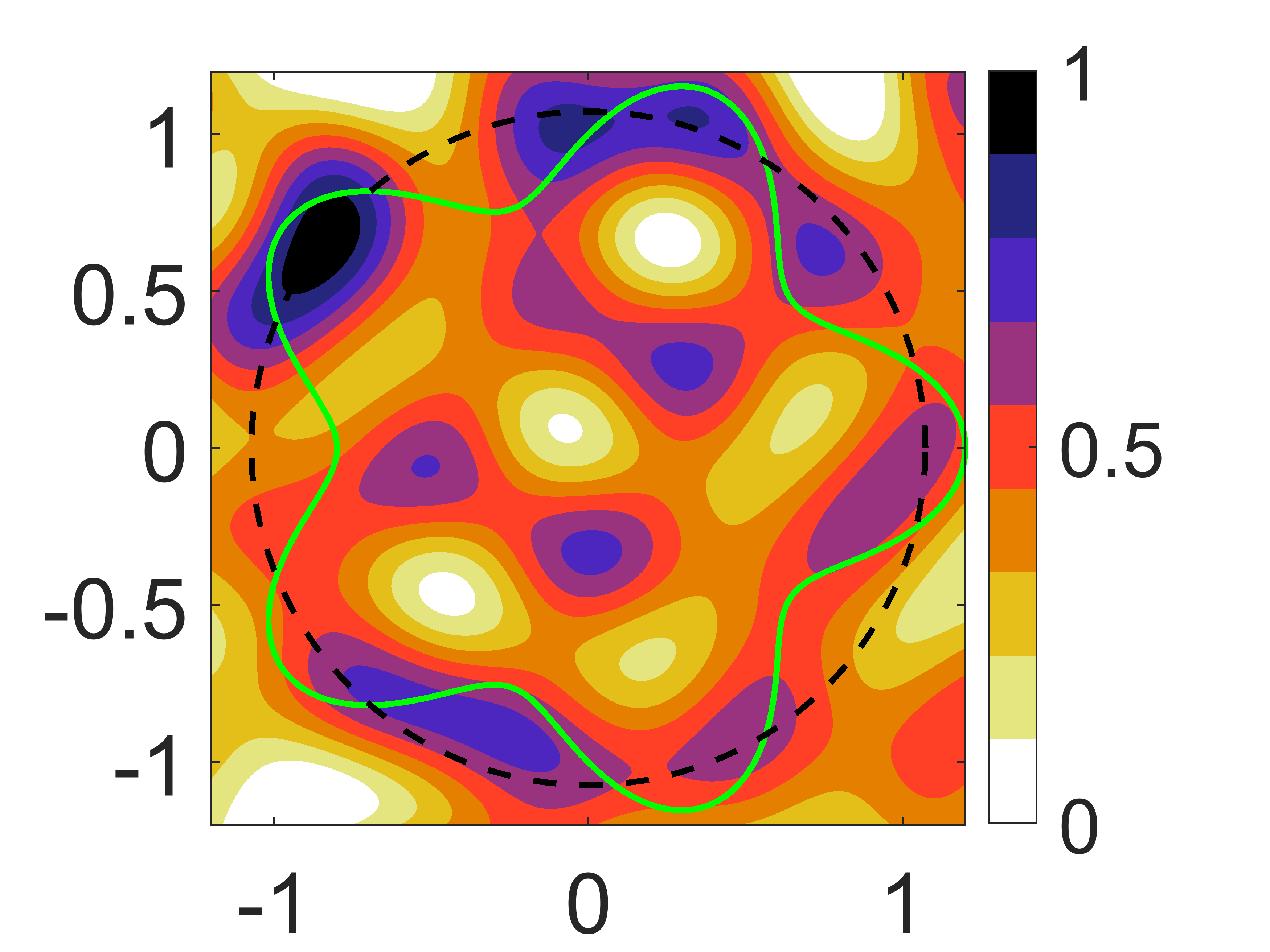}}
		\subfigure[]{\includegraphics[width=0.24\linewidth]{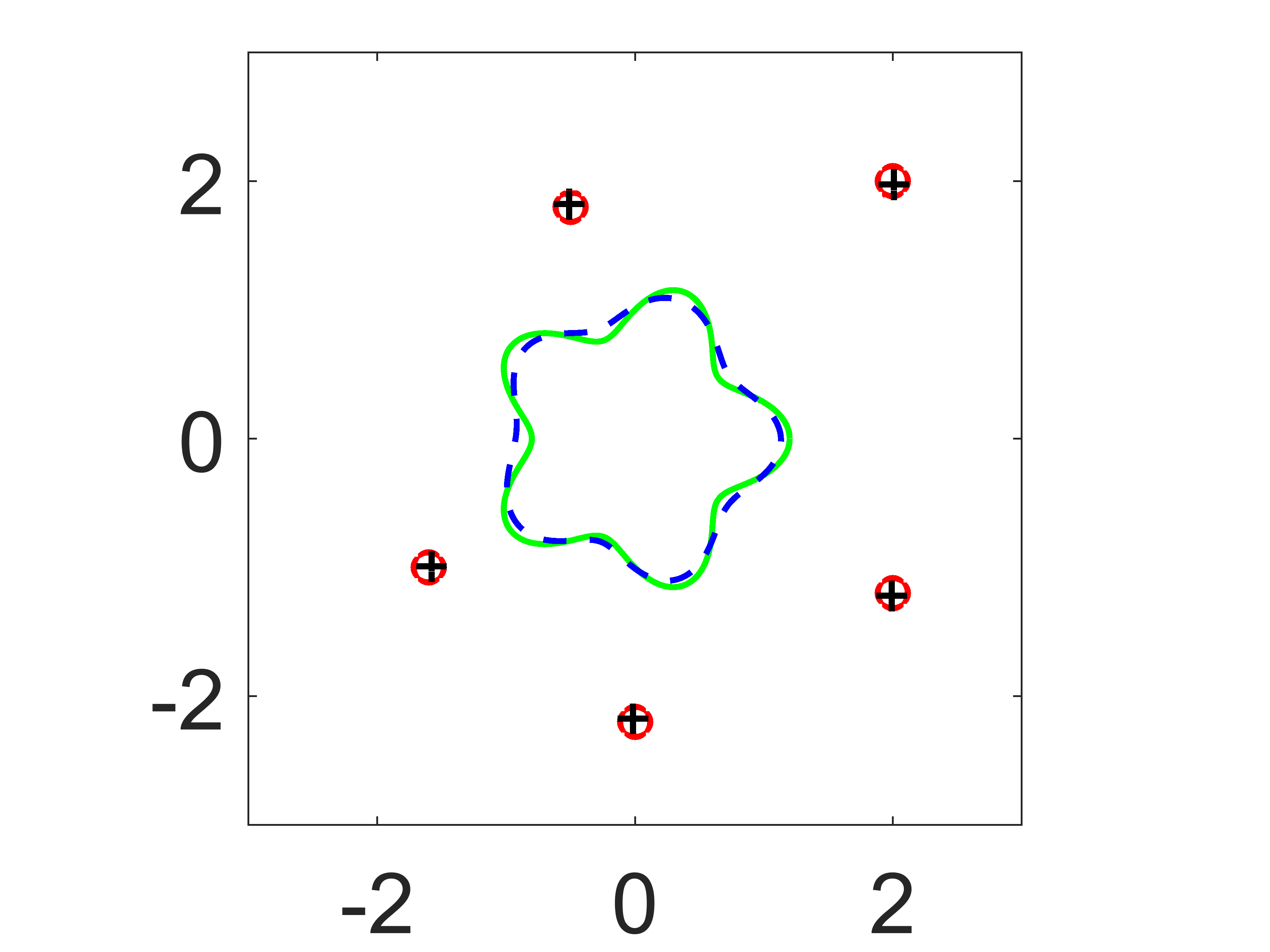}}
		\caption{Geometry setting and the reconstruction of the starfish and $S_1$. (a) Problem geometry; (b) exact and reconstructed locations of the source points by the DSM; (c) image of $I_D(y)$; (d) reconstruction by the optimization method.}
		\label{fig:starfish2}
	\end{figure}
		
	\begin{figure}[htpb]
		\centering
		\subfigure[]{\includegraphics[width=0.3\linewidth]{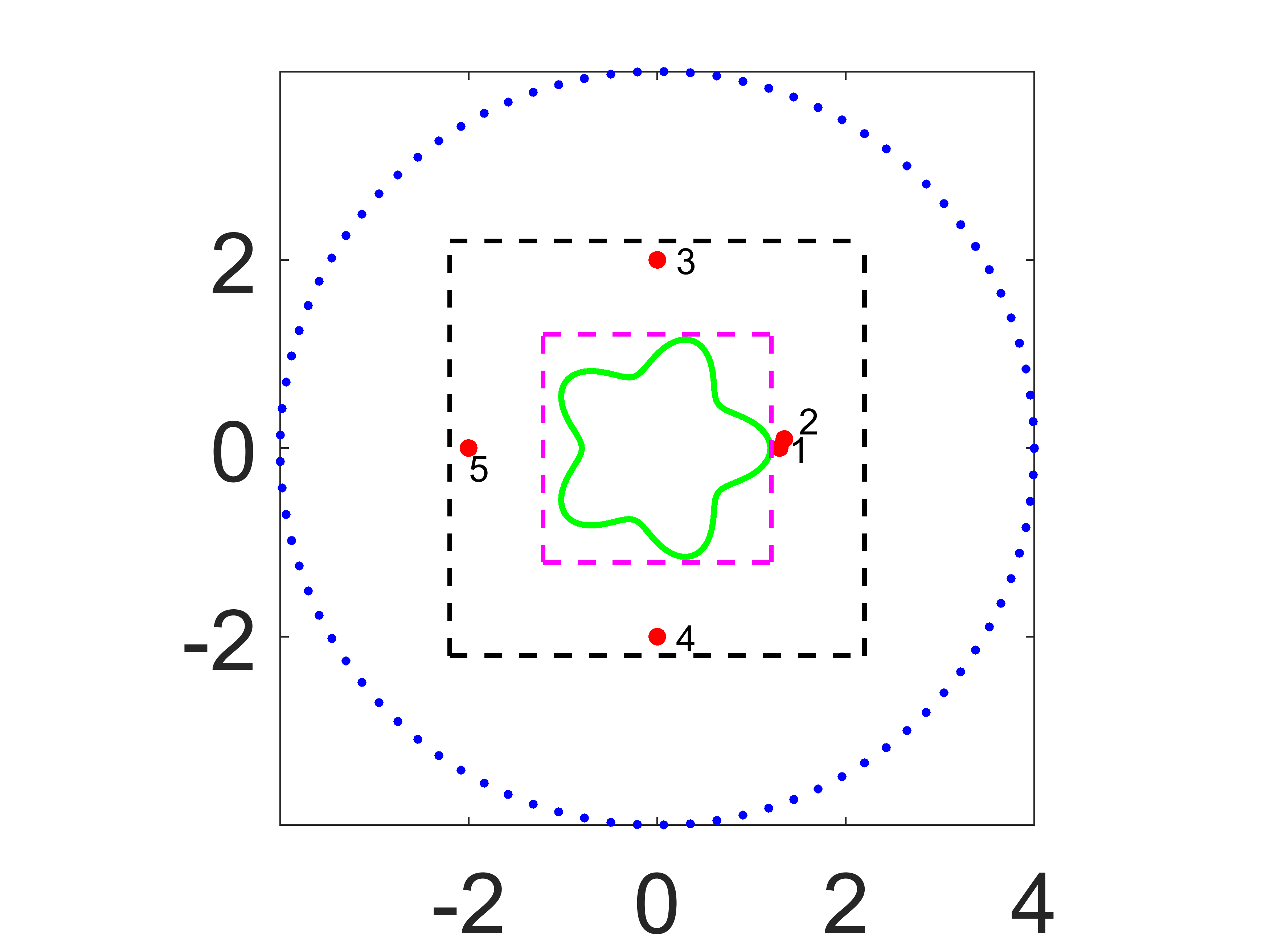}}
		\subfigure[]{\includegraphics[width=0.3\linewidth]{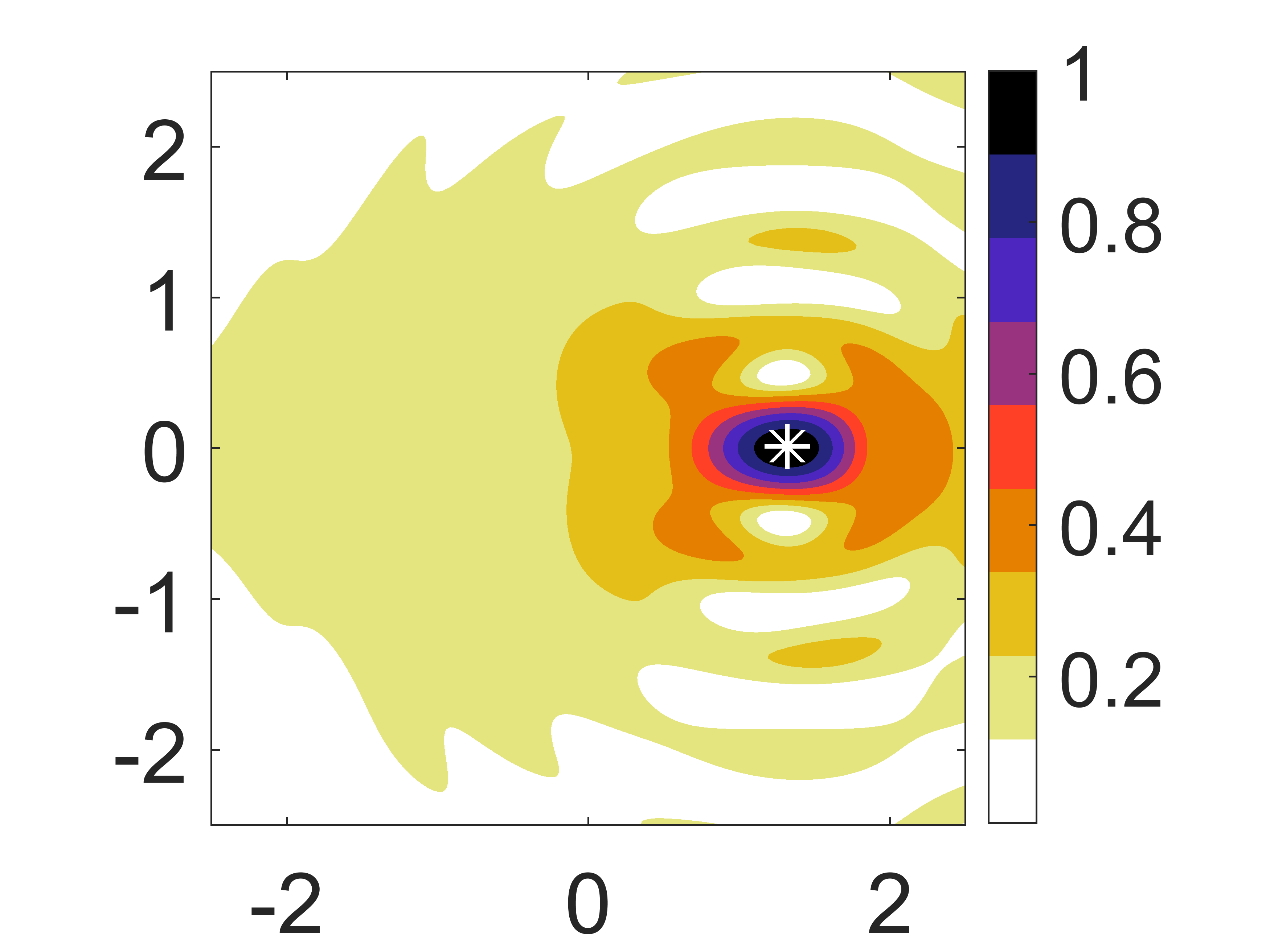}}
		\subfigure[]{\includegraphics[width=0.3\linewidth]{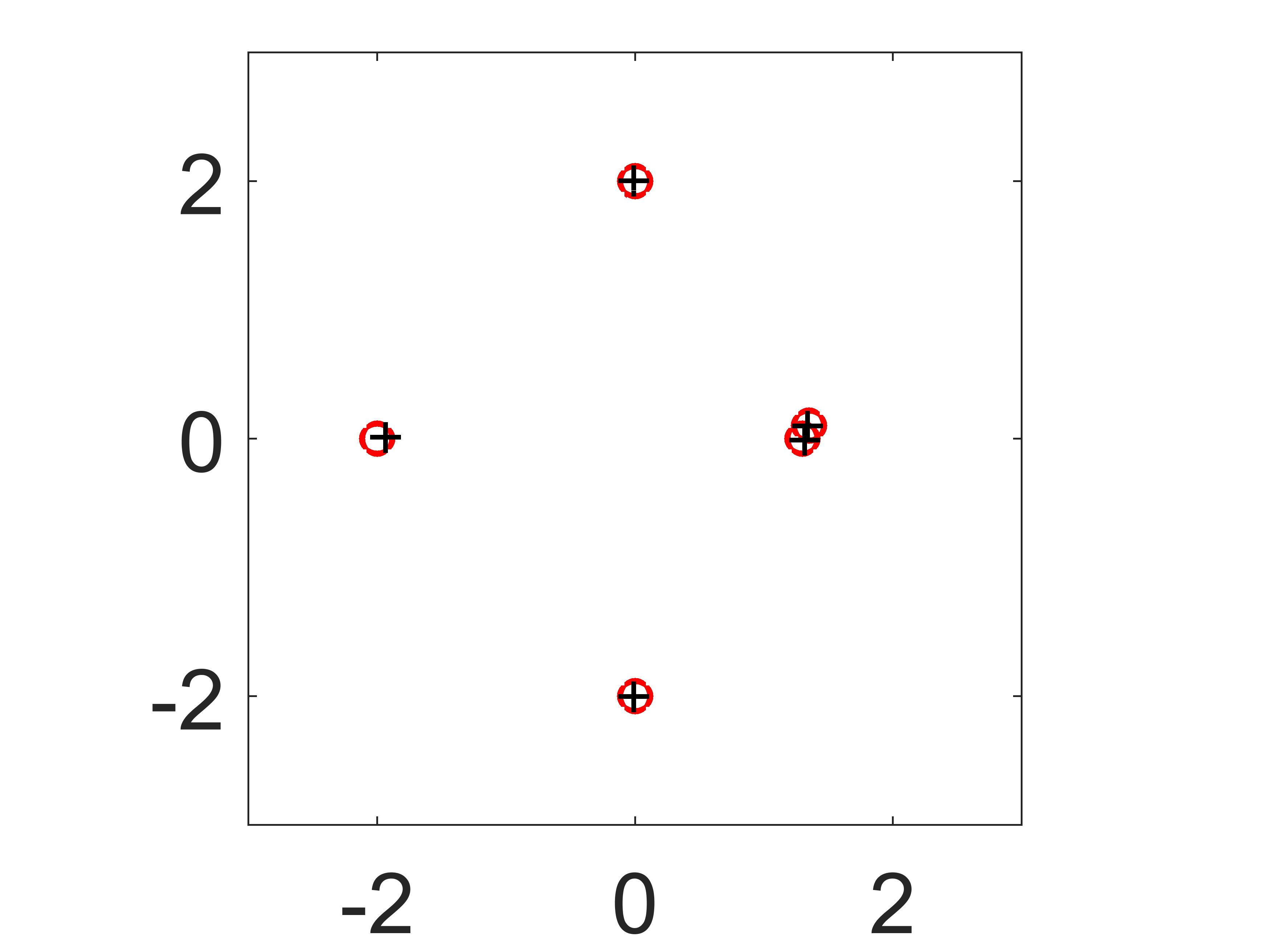}}\\
		\subfigure[]{\includegraphics[width=0.3\linewidth]{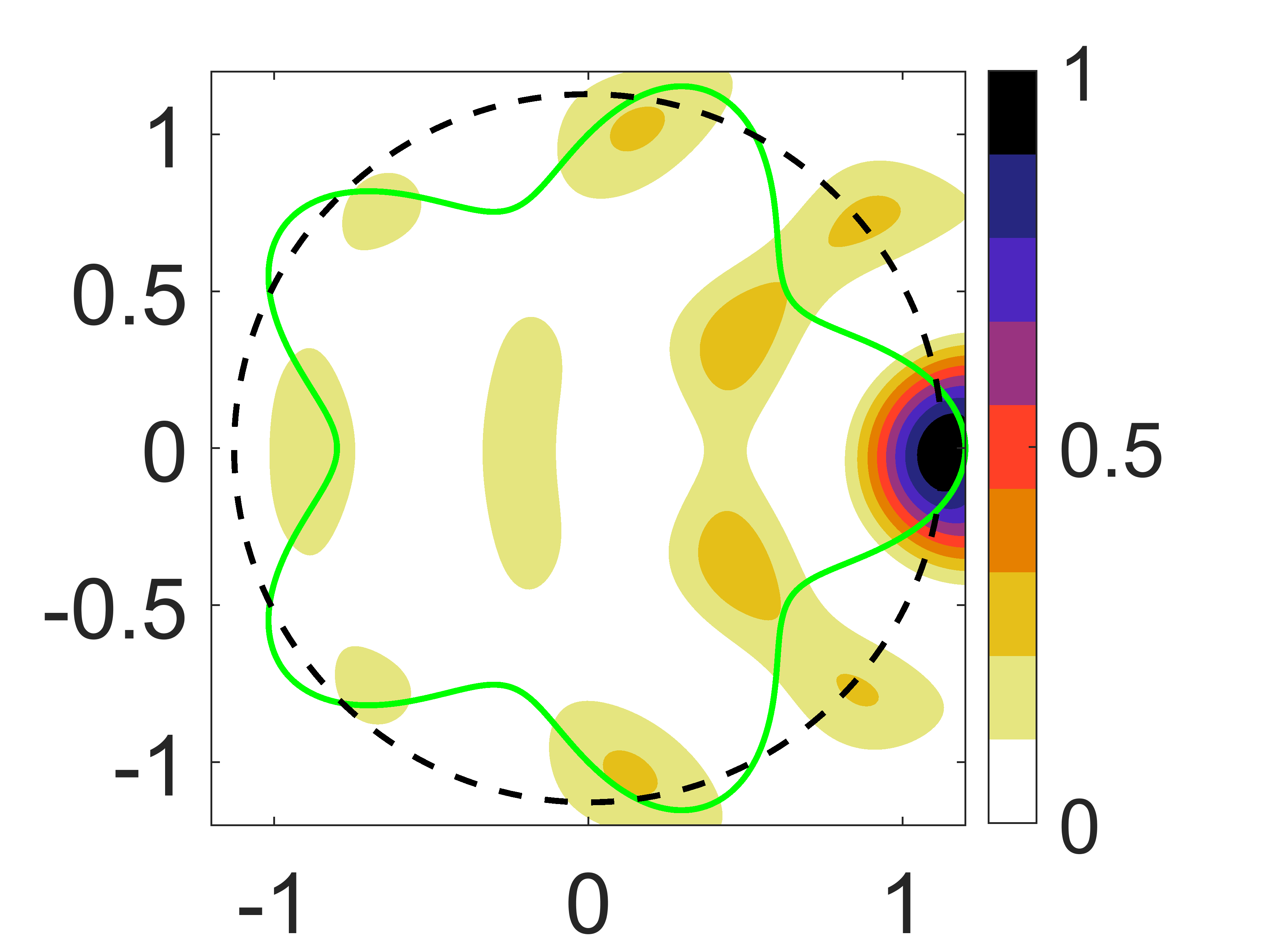}}
		\subfigure[]{\includegraphics[width=0.3\linewidth]{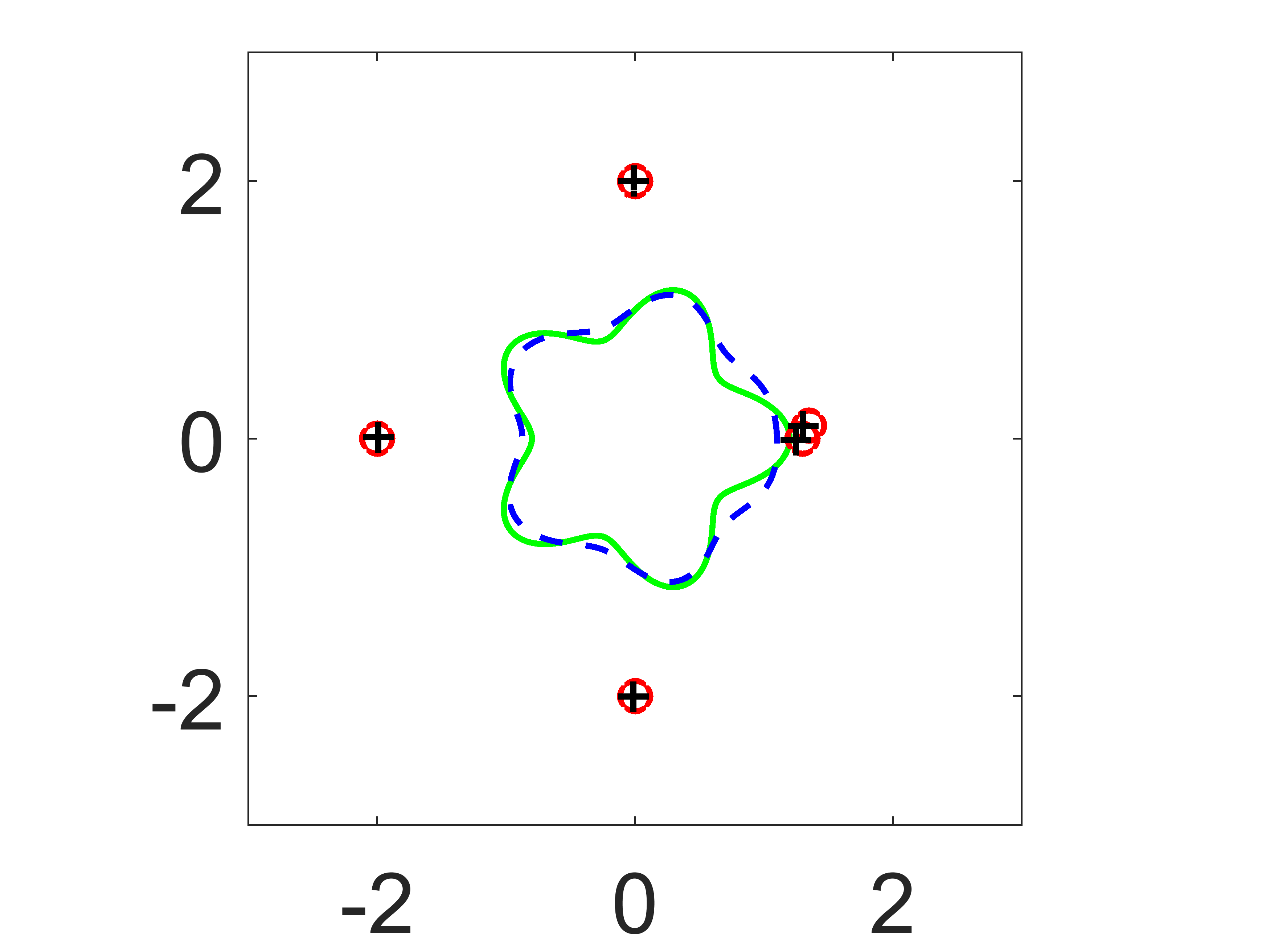}}
		\subfigure[]{\includegraphics[width=0.3\linewidth]{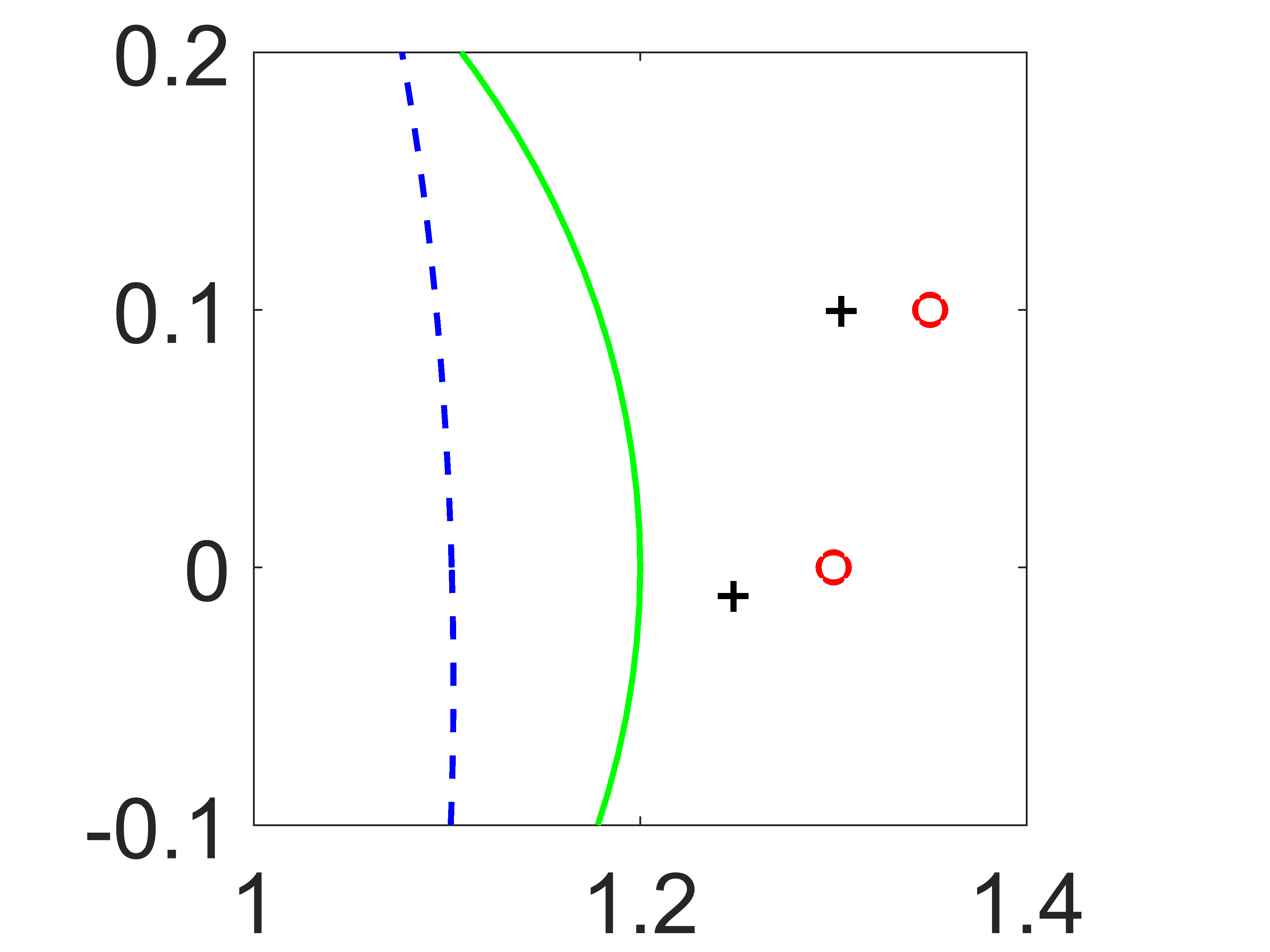}}
		\caption{Geometry setting and the reconstruction of the starfish and $S_2$. (a) Problem geometry; (b) image of $I_1(y)$ (white star: exact location of the first source point); (c) exact and reconstructed locations of the source points with $I_j(y), j=1,\cdots5$; (d) image of $I_D(y)$; (e) reconstruction by the optimization method; (f) optimization result in a local domain.}
		\label{fig:starfish1}
	\end{figure}
	
	Next we consider the reconstruction from a different set of source points, which is given by $
	S_2=\{(1.3,0), (1.35,0.1), (0,2), (0,-2), (-2,0)\}$. In this case, the sampling domains are adapted to be $\Omega_1=[-2.2,2.2]\times[-2.2,2.2],$ $\Omega_2=[-1.2,1.2]\times[-1.2,1.2]$. The other parameters remain unchanged. We refer to Figure \ref{fig:starfish1}(a) for the geometry setting.
			
	Figure \ref{fig:starfish1}(b) shows the indicator $I_1(y)$ over the sampling domain $\Omega_1$. One can observe that the maximizer of $I_1(y)$ coincides well with the exact source location. Figure \ref{fig:starfish1}(c) shows that all the reconstructed source locations match well with the exact positions. 	In Figure \ref{fig:starfish1}(d), the imaging function $I_D(y)$ is plotted. From Figure \ref{fig:starfish1}(d), we find that the maximizer of $I_D(y)$ incurs on the boundary of the obstacle.  
	
	Compared with Figure \ref{fig:starfish2}(c), it seems that the current result deteriorates slightly. A reason accounting for this fact is the specific distribution of the source points where two points are almost adjacent to each other.
	In Figure \ref{fig:starfish1}(e), we show the reconstruction of the starfish by the optimization method. Figure \ref{fig:starfish1}(e) demonstrates that the quality of reconstruction for the left side of the obstacle is better than that of the right side. Physically speaking, this is due to the fact that the right side of the obstacle is affected by the nearby source points 1 and 2. Moreover, the adjacent source points 1 and 2 could be individually identified as long as the sampling grid is sufficiently fine. To see the local reconstruction more clearly, we intercept a part $[1, 1.4]\times[-0.1, 0.2]$ of Figure \ref{fig:starfish1}(e) and it is zoomed in as Figure \ref{fig:starfish1}(f). 
	
	In Table \ref{tab:star_5_location}, we listed the exact and reconstructed source locations. To evaluate the reconstruction of the obstacle quantitatively, we compute the errors between the exact boundary curve and the reconstructed boundary curve as \eqref{eq:error1} with $\tau_i=t_i$. In this case, the error of the boundary curve is $E_D=7.90\%$.
		
	Comparing the above two cases, it can be found that our method could simultaneously  produce a rather accurate reconstruction of the obstacle and sources, no matter whether the source points are relatively sparse or dense, or even quite close to the obstacle.
	
			\begin{table}[htpb]
				\begin{center}
					\caption{Reconstruction of the source points $S_1$.}
					\begin{tabular}{cccc}
						\toprule
						& ~~~~~~~\multirow{2}{2cm}{Exact Location} & \multicolumn{2}{c}{Reconstructed Location} \\
						\cmidrule(r){3-4} &                                          &       DSM       &   Optimization   \\ \midrule
						Point 1      &                 $(2, 2)$                 &  $(1.98,1.98)$  &  $(2.01,1.98)$   \\
						Point 2      &              $(-0.5, 1.8)$               & $(-0.52,1.83)$  &  $(-0.50,1.83)$  \\
						Point 3      &               $(-1.6, -1)$               & $(-1.58,-0.98)$ & $(-1.56,-0.98)$  \\ 
						Point 4      &               $(0, -2.2)$                & $(0.02,-2.19)$  & $(-0.01,-2.19)$  \\ 
						Point 5      &               $(2, -1.2)$                & $(2.03,-1.23)$  &  $(1.99,-1.23)$  \\ \bottomrule
					\end{tabular}
					\label{tab:star_5'_location}
				\end{center}
			\end{table}		
			
	In Figure \ref{fig:starfish_different}, we consider the reconstruction with $N(=2,4,6,8)$ source points. The wave number and the regularization parameter is chosen to be $k=8$ and $\alpha=10^{-6}$, respectively.
	The sampling domains are chosen to be $\Omega_1=[-2.5,2.5]\times[-2.5,2.5],$ $\Omega_2=[-1.5,1.5]\times[-1.5,1.5].$ Further, we take $k=5,8,$ respectively, and list the relative $L^2$ errors for the reconstruction of the starfish-shaped obstacle in Table \ref{tab:error_starfish}. Table \ref{tab:error_starfish} shows that the error of reconstruction decreases as the number of source points increases.
	
	\begin{table}[htpb]
		\begin{center}
			\caption{Reconstruction of the source points $S_2$.} \label{tab:star_5_location}
			\begin{tabular}{ccc}
				\toprule
				& Exact source  & Reconstruction  \\ \midrule
				Point 1 &  $(1.3, 0)$   &  $(1.25,0.02)$  \\
				Point 2 & $(1.35, 0.1)$ &  $(1.30,0.07)$  \\
				Point 3 &   $(0, 2)$    & $(-0.00,1.98)$  \\ 
				Point 4 &   $(0, -2)$   & $(-0.00,-1.98)$ \\ 
				Point 5 &   $(-2, 0)$   & $(-1.99,0.02)$  \\ \bottomrule
			\end{tabular}
		\end{center}
	\end{table}
			
			\begin{table}[htpb]
				\caption{Relative $L^2$ errors for reconstructions of the starfish.}
				\label{tab:error_starfish}
				\centering
				\begin{tabular}{ccccc}
					\toprule
					$N$  &    2     &    4     &    6     &    8     \\ \midrule
					$k=5$ & $8.60\%$ & $7.45\%$ & $8.01\%$ & $5.84\%$ \\
					$k=8$ & $15.46\%$ & $6.63\%$ & $5.93\%$ & $4.68\%$ \\ \bottomrule
				\end{tabular}
			\end{table}

			\begin{figure}[htpb]
				\centering
				\subfigure[]{\includegraphics[width=0.24\linewidth]{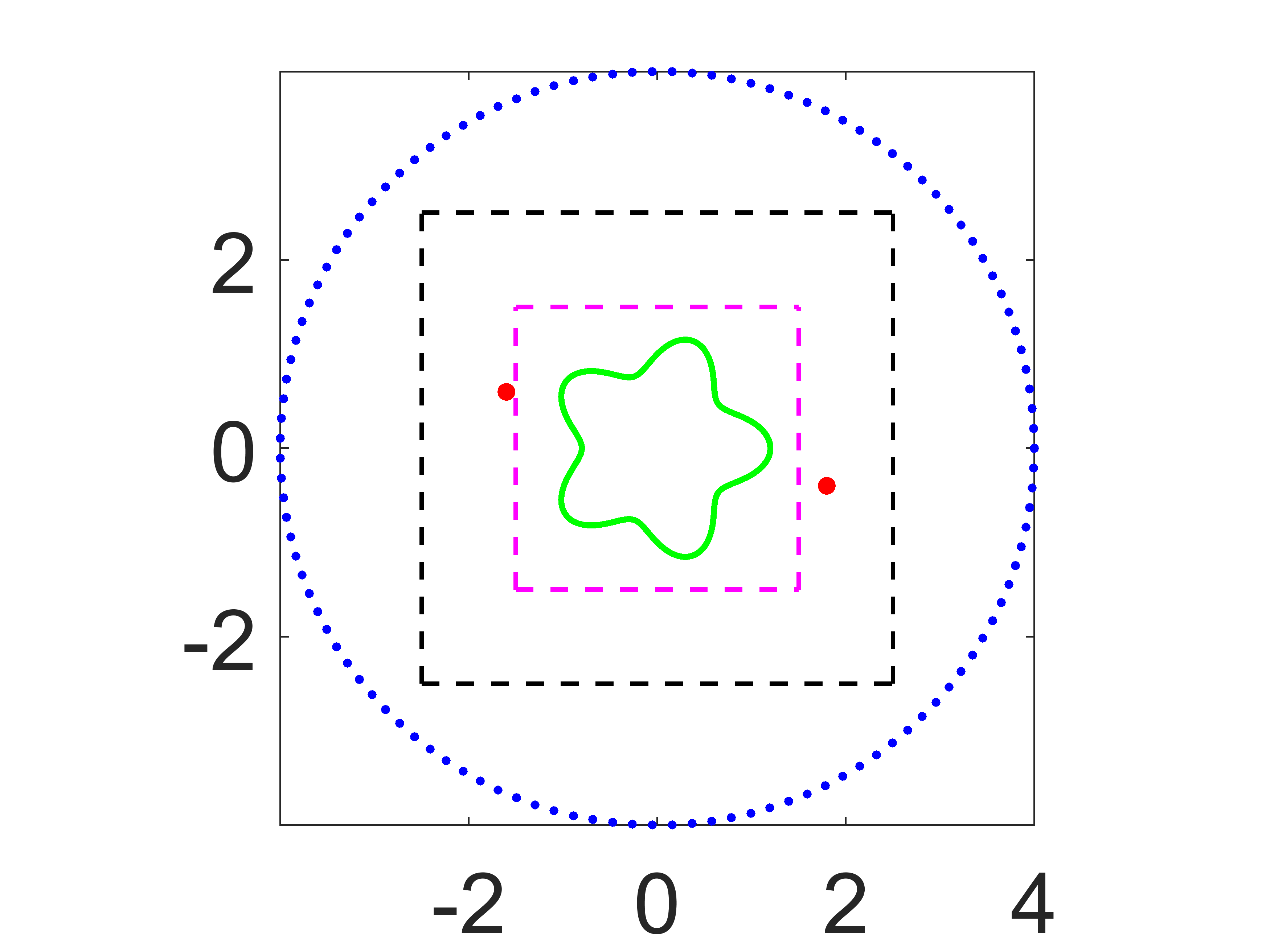}}
				\subfigure[]{\includegraphics[width=0.24\linewidth]{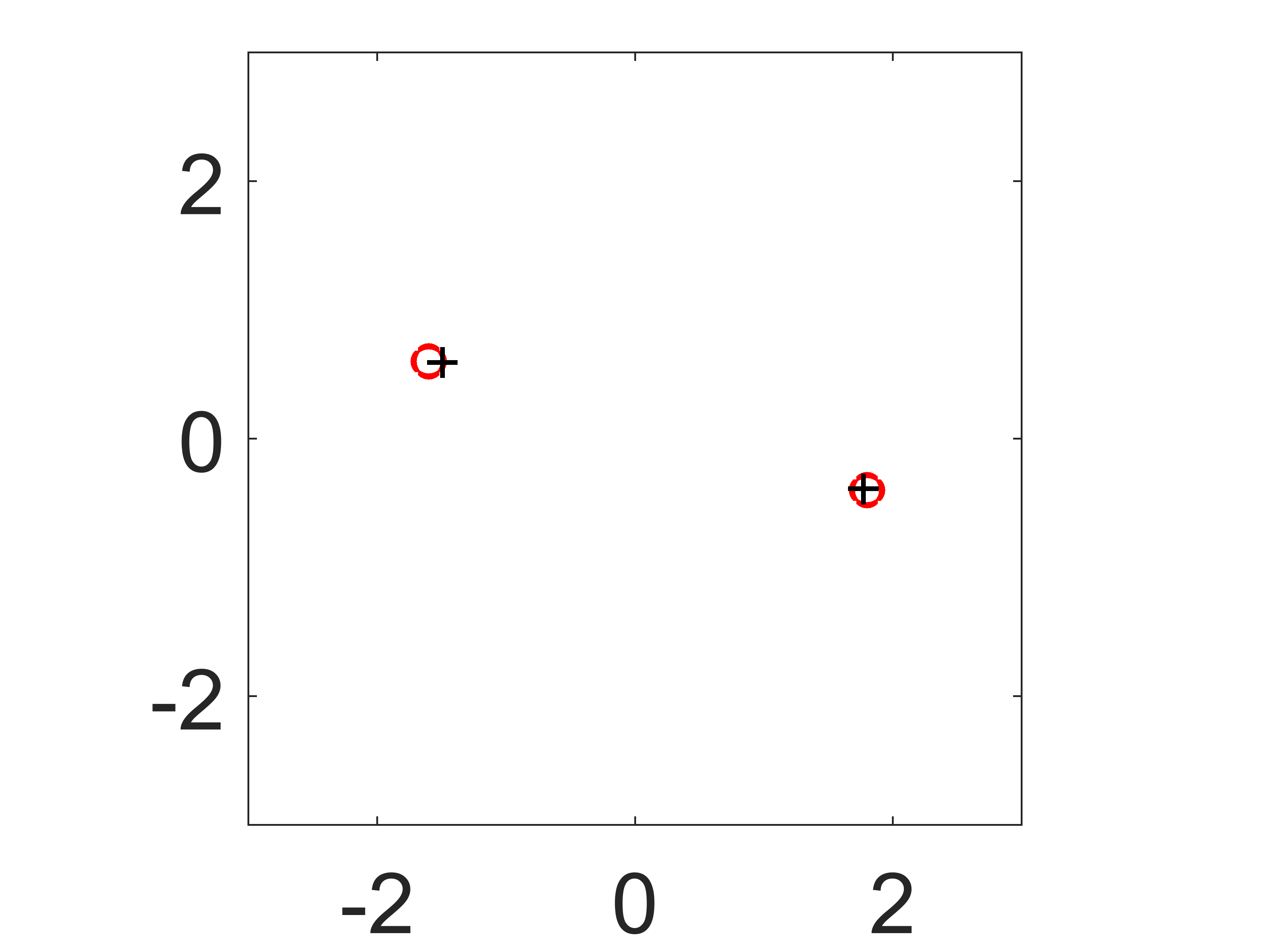}}
				\subfigure[]{\includegraphics[width=0.24\linewidth]{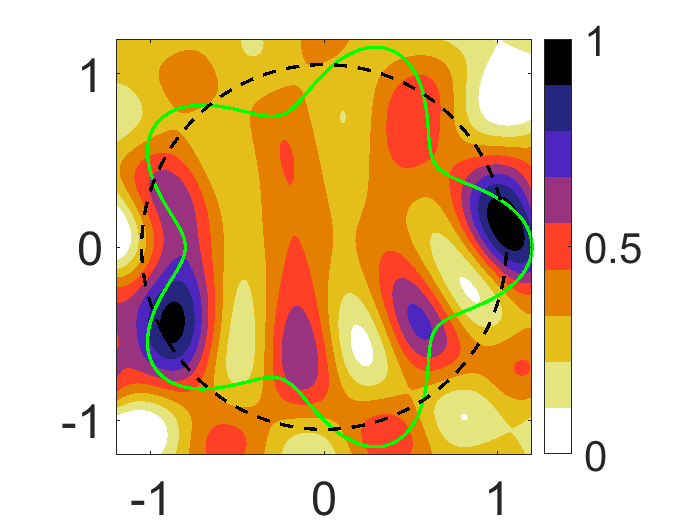}}
				\subfigure[]{\includegraphics[width=0.24\linewidth]{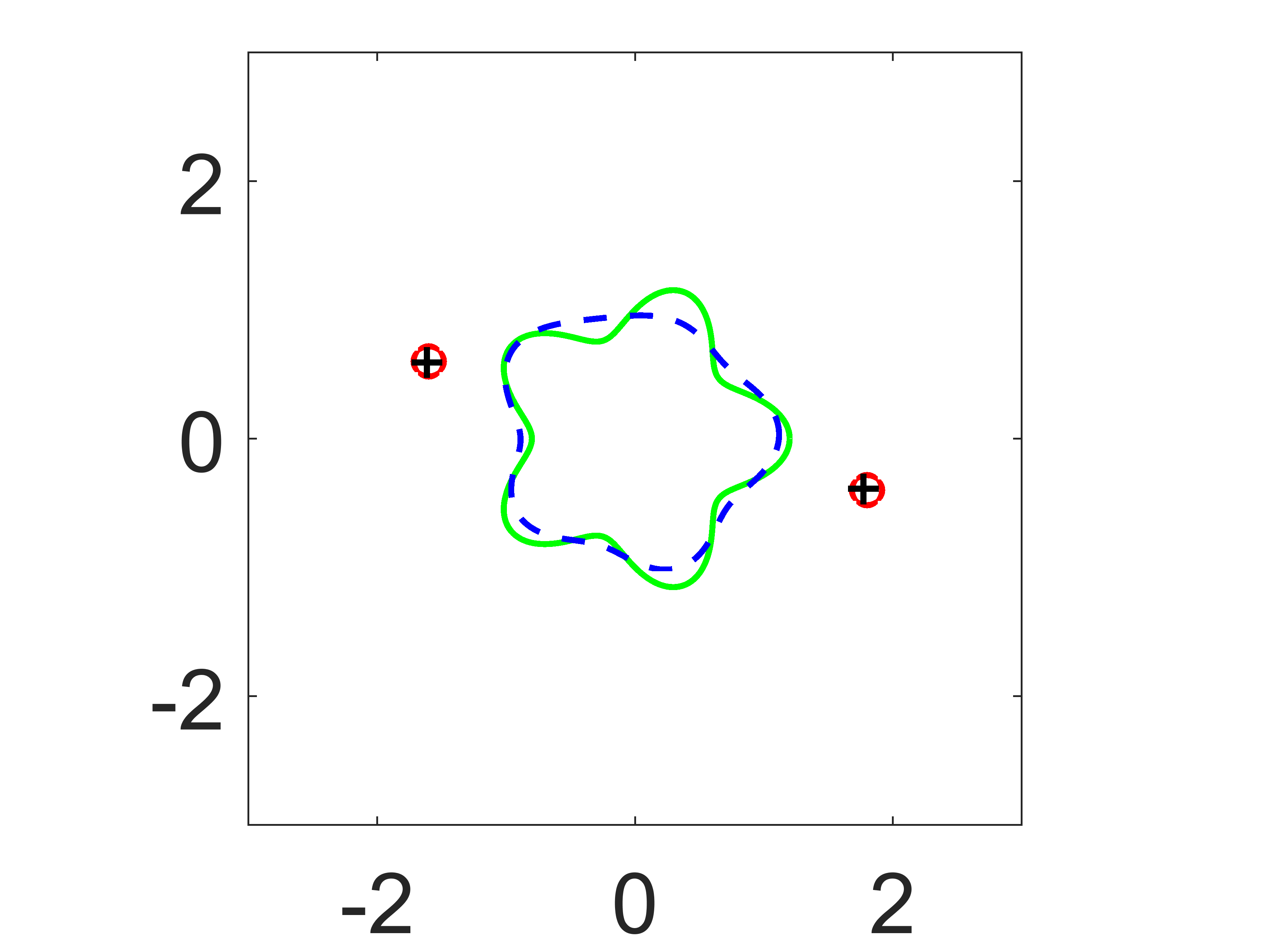}}

				\subfigure[]{\includegraphics[width=0.24\linewidth]{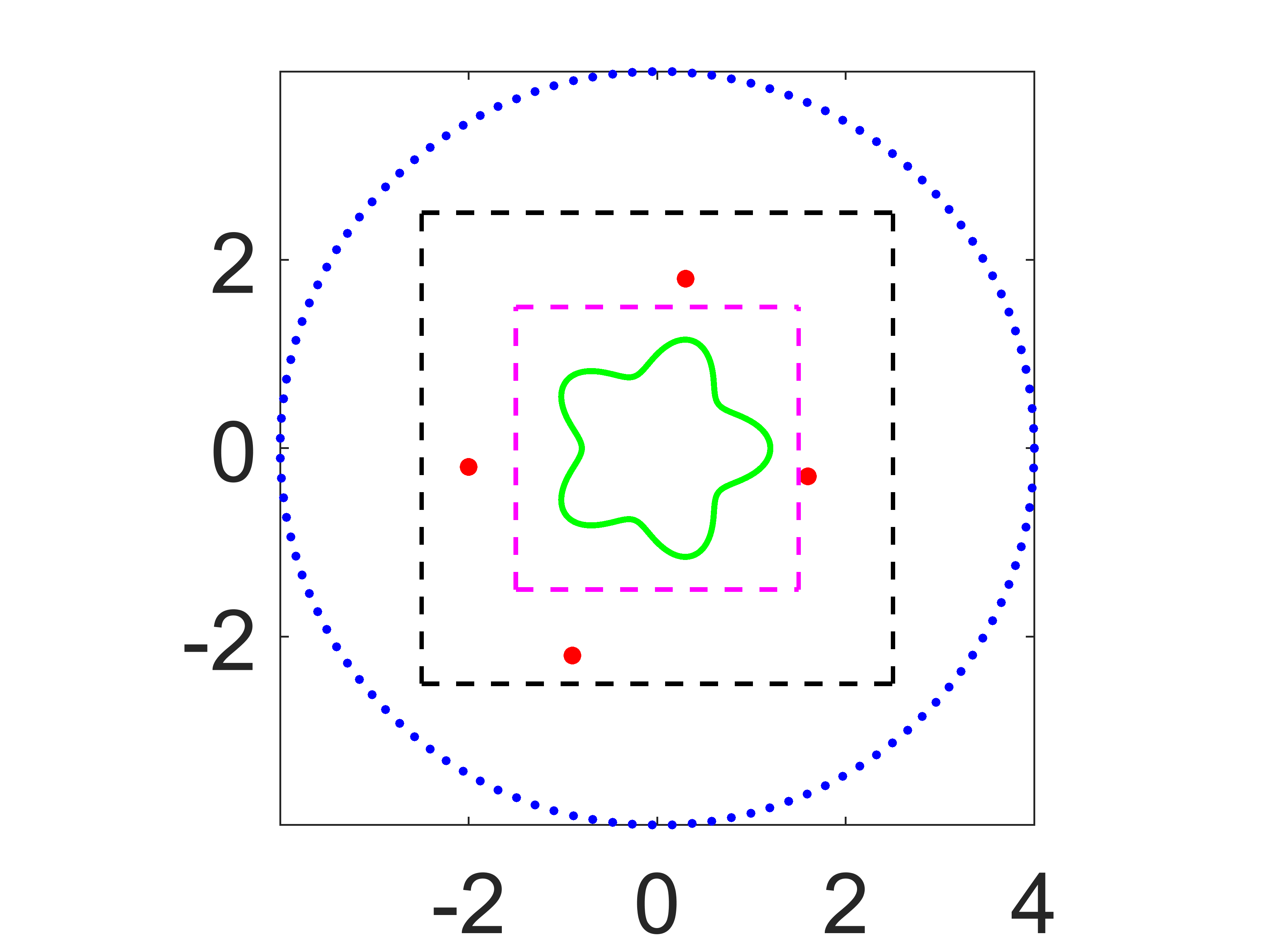}}
				\subfigure[]{\includegraphics[width=0.24\linewidth]{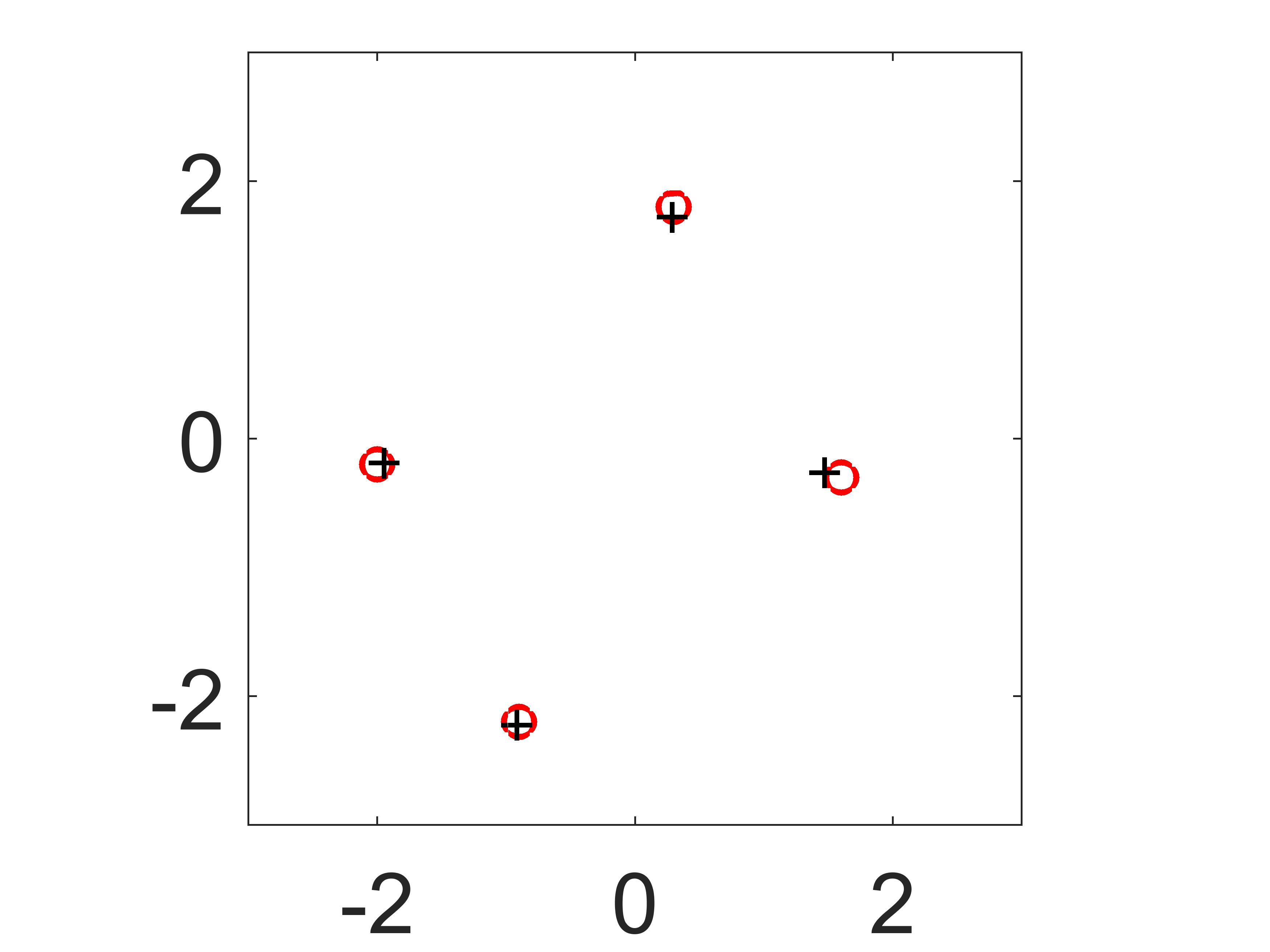}}
				\subfigure[]{\includegraphics[width=0.24\linewidth]{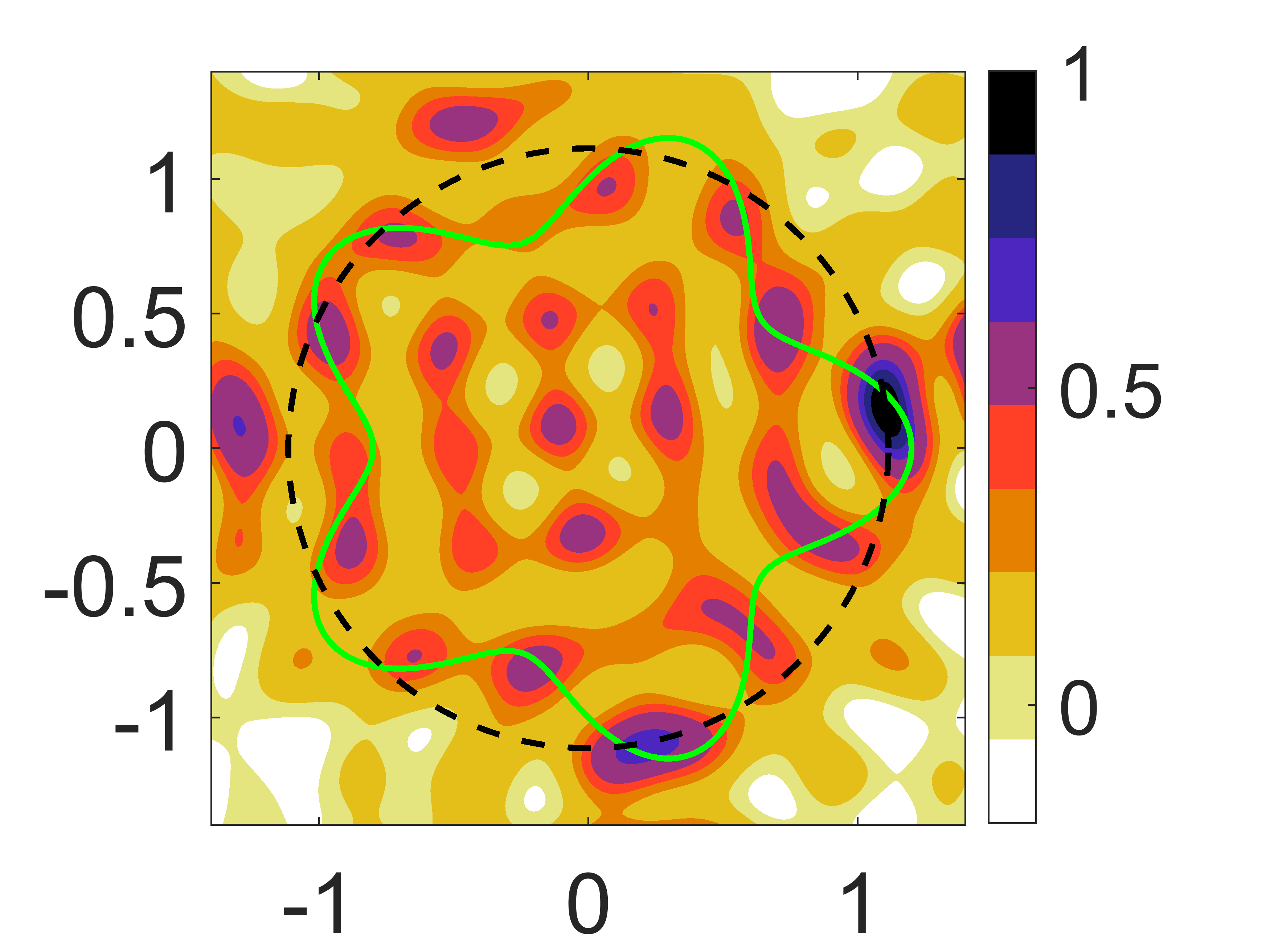}}
				\subfigure[]{\includegraphics[width=0.24\linewidth]{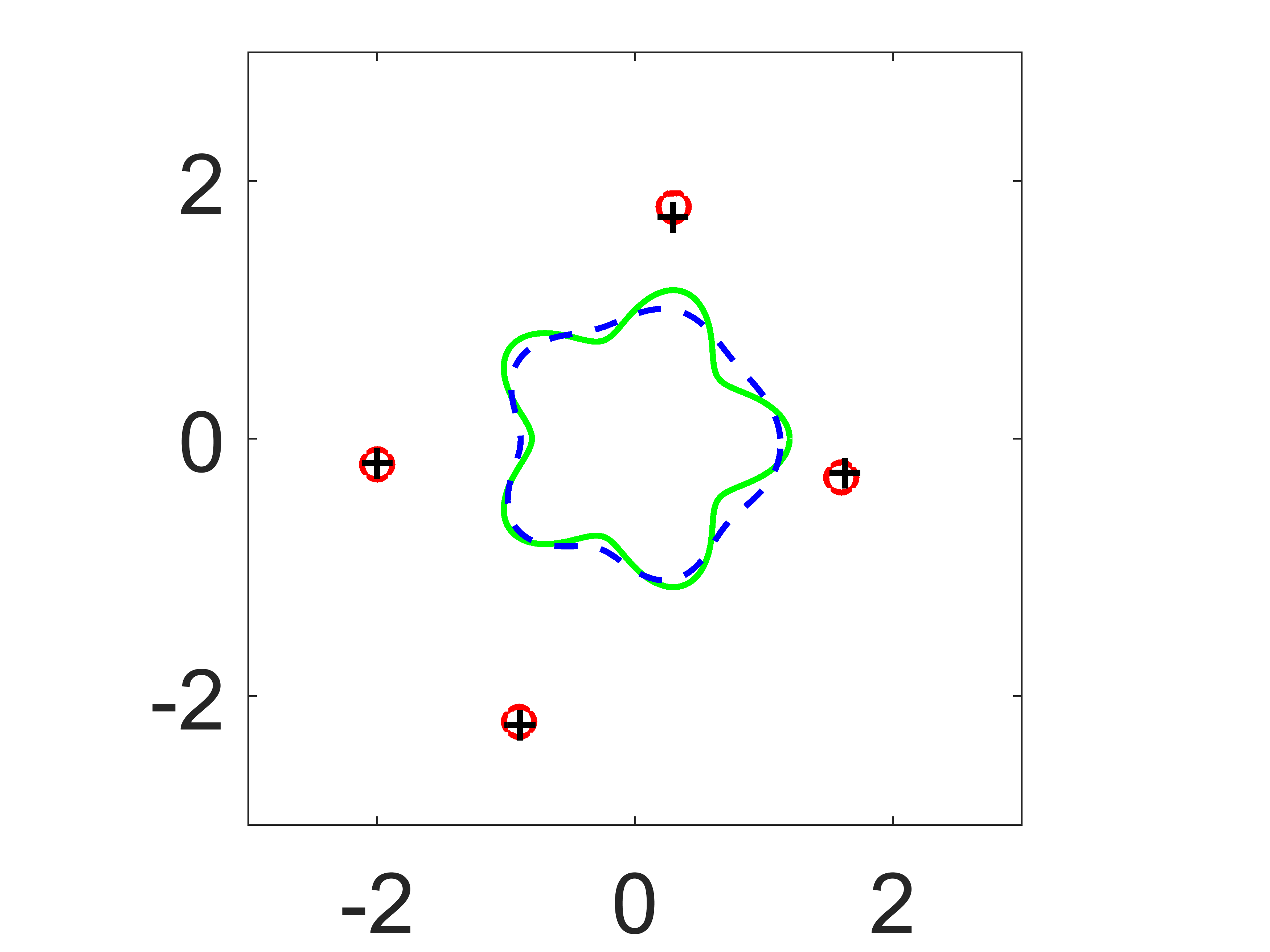}}

				\subfigure[]{\includegraphics[width=0.24\linewidth]{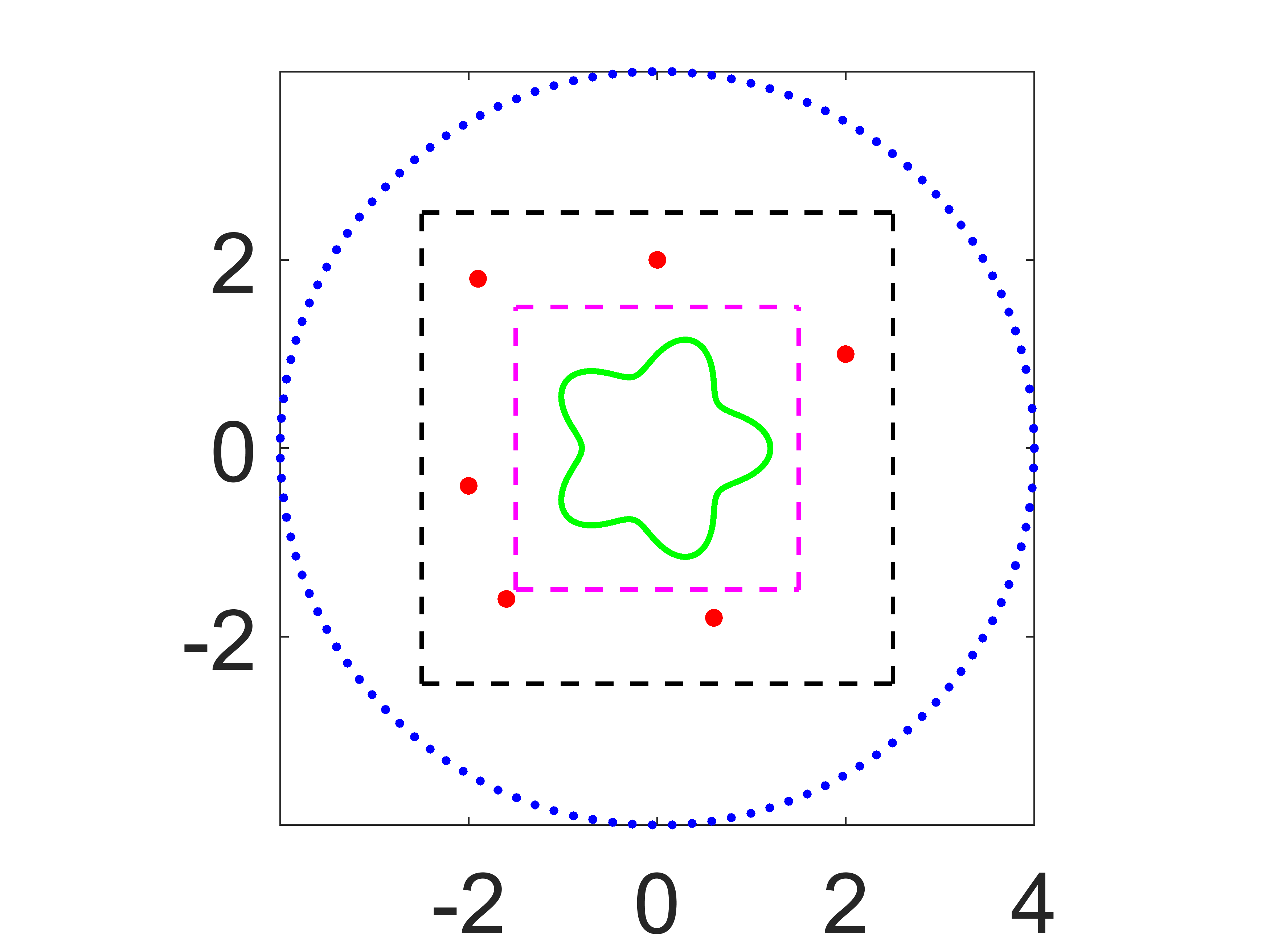}}
				\subfigure[]{\includegraphics[width=0.24\linewidth]{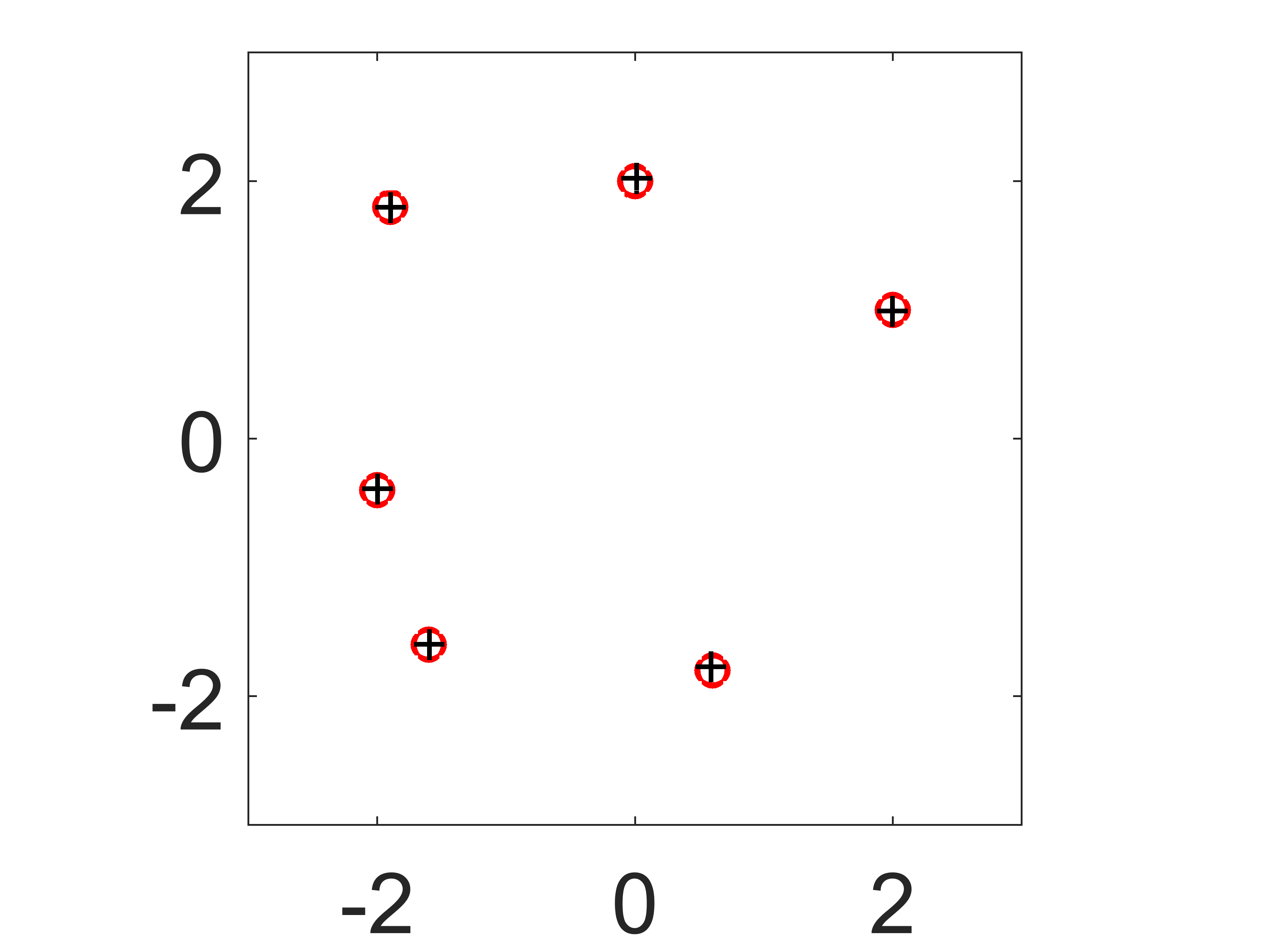}}
				\subfigure[]{\includegraphics[width=0.24\linewidth]{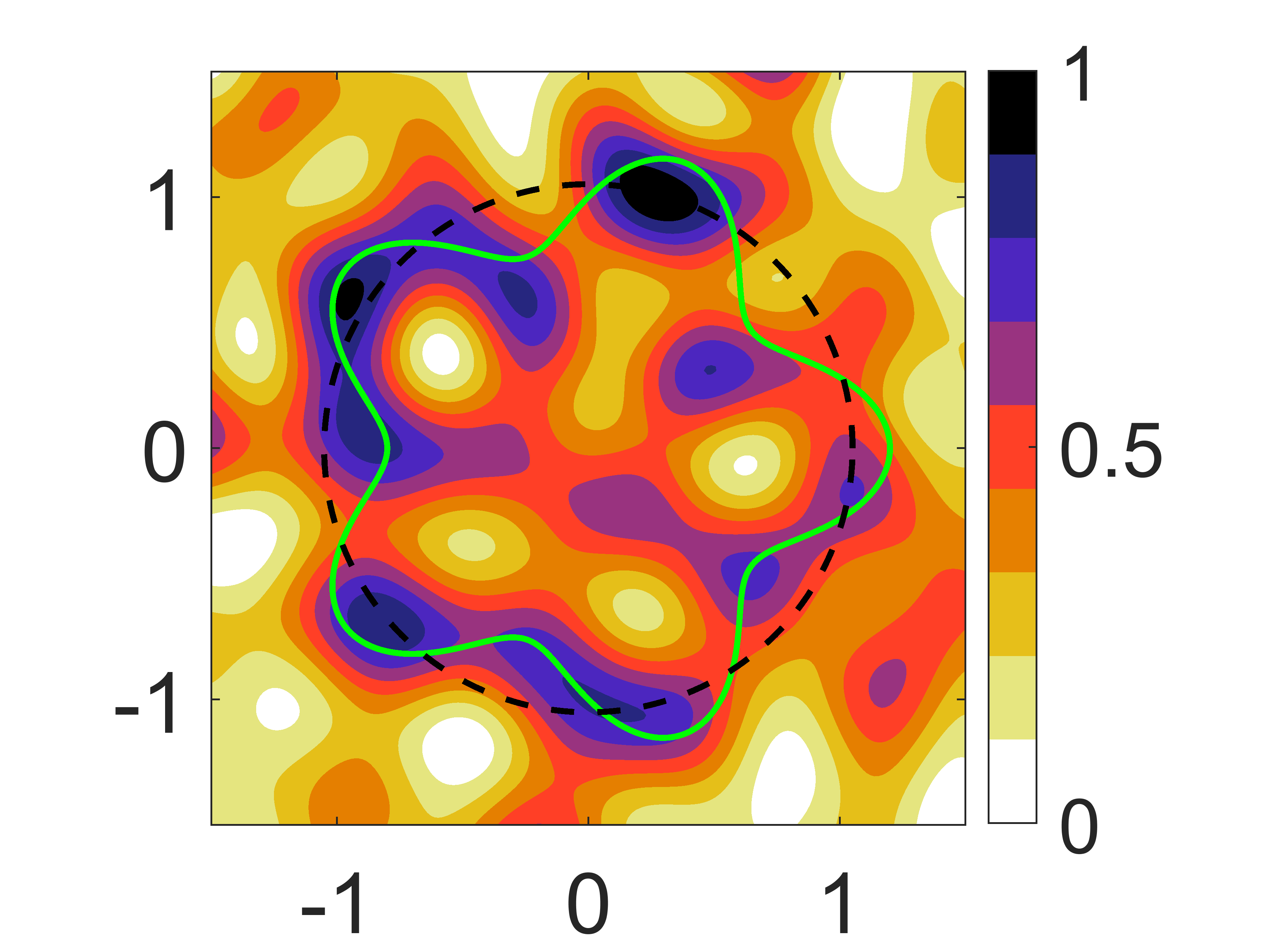}}
				\subfigure[]{\includegraphics[width=0.24\linewidth]{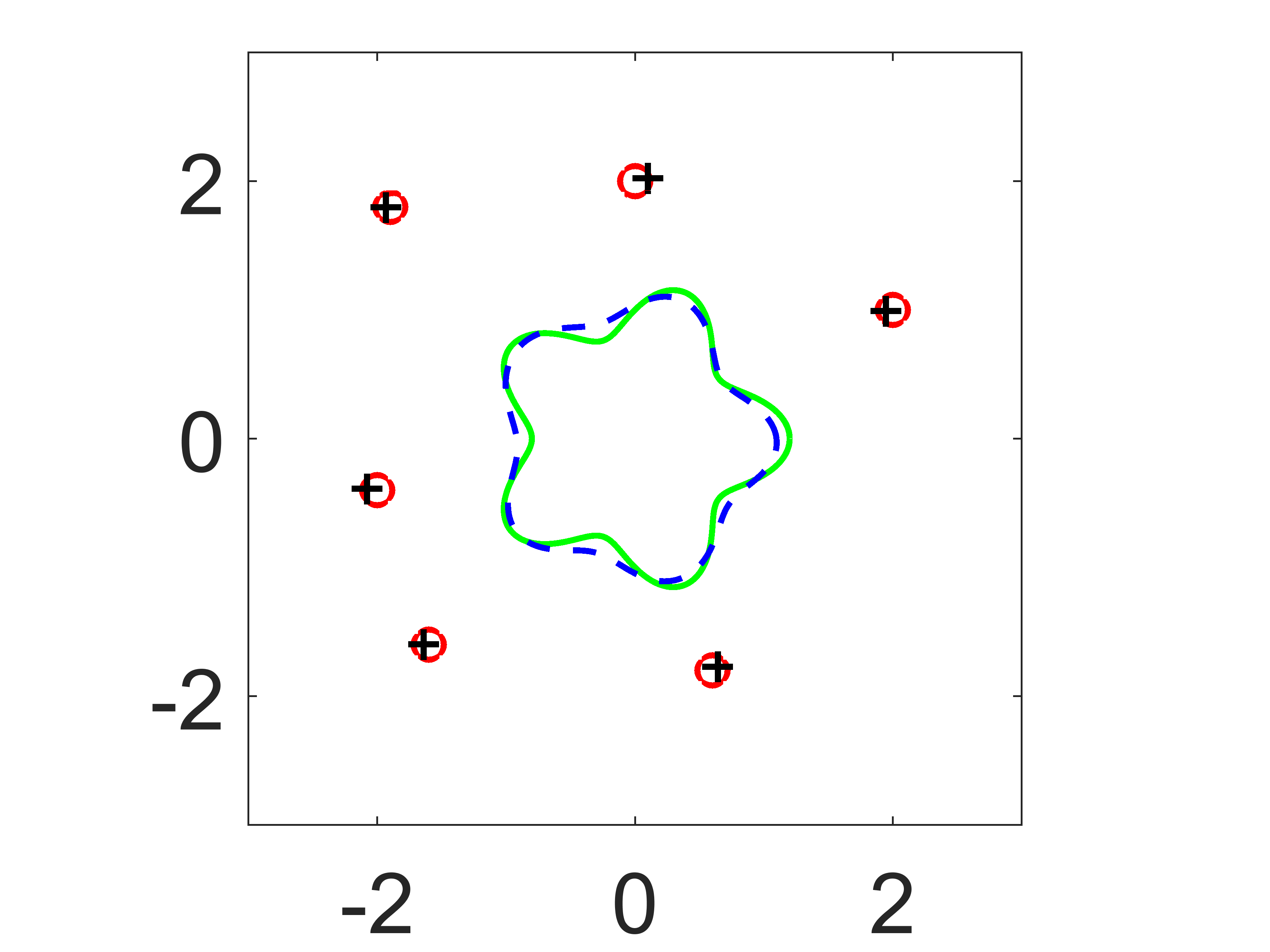}}

				\subfigure[]{\includegraphics[width=0.24\linewidth]{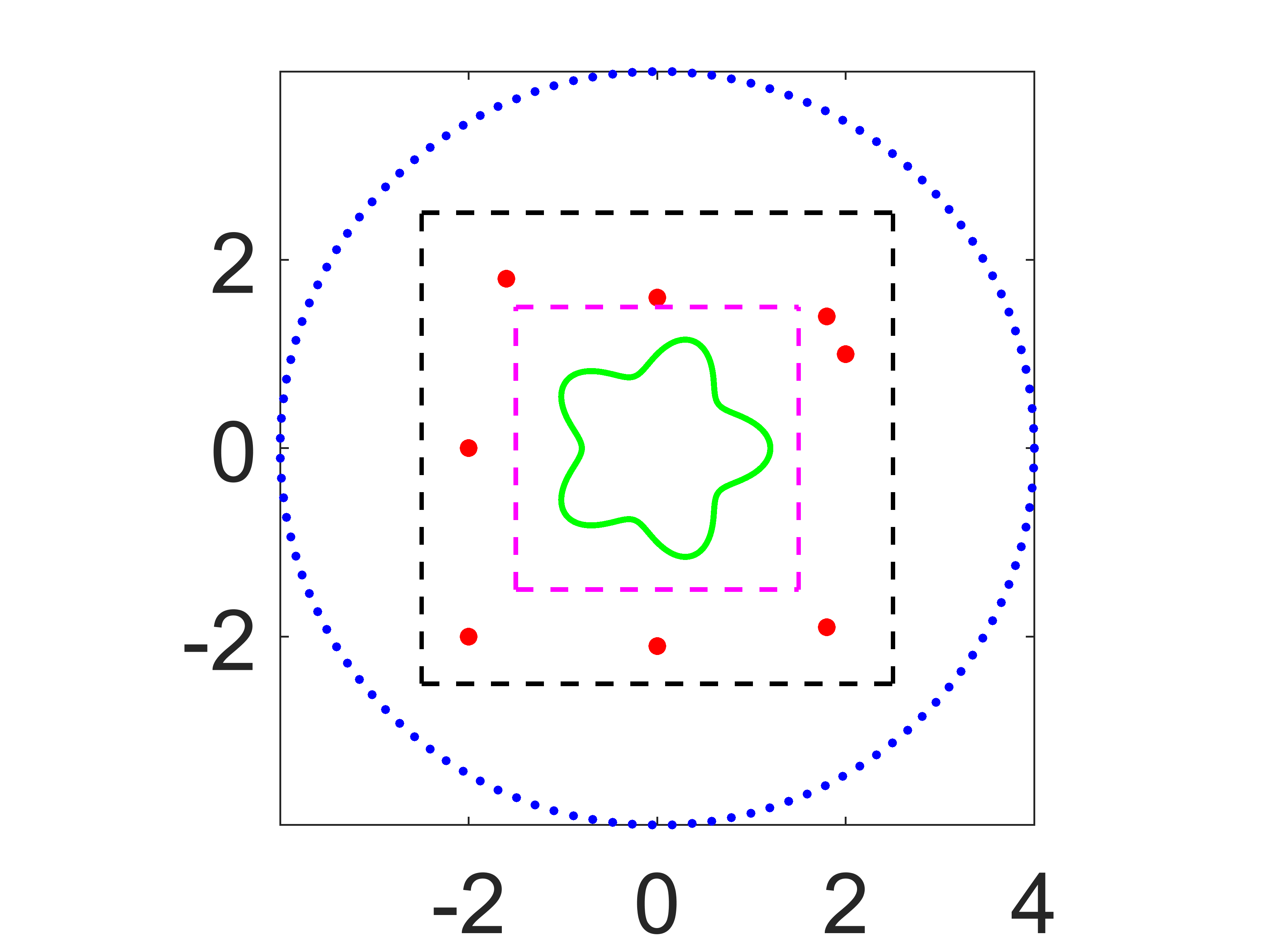}}
				\subfigure[]{\includegraphics[width=0.24\linewidth]{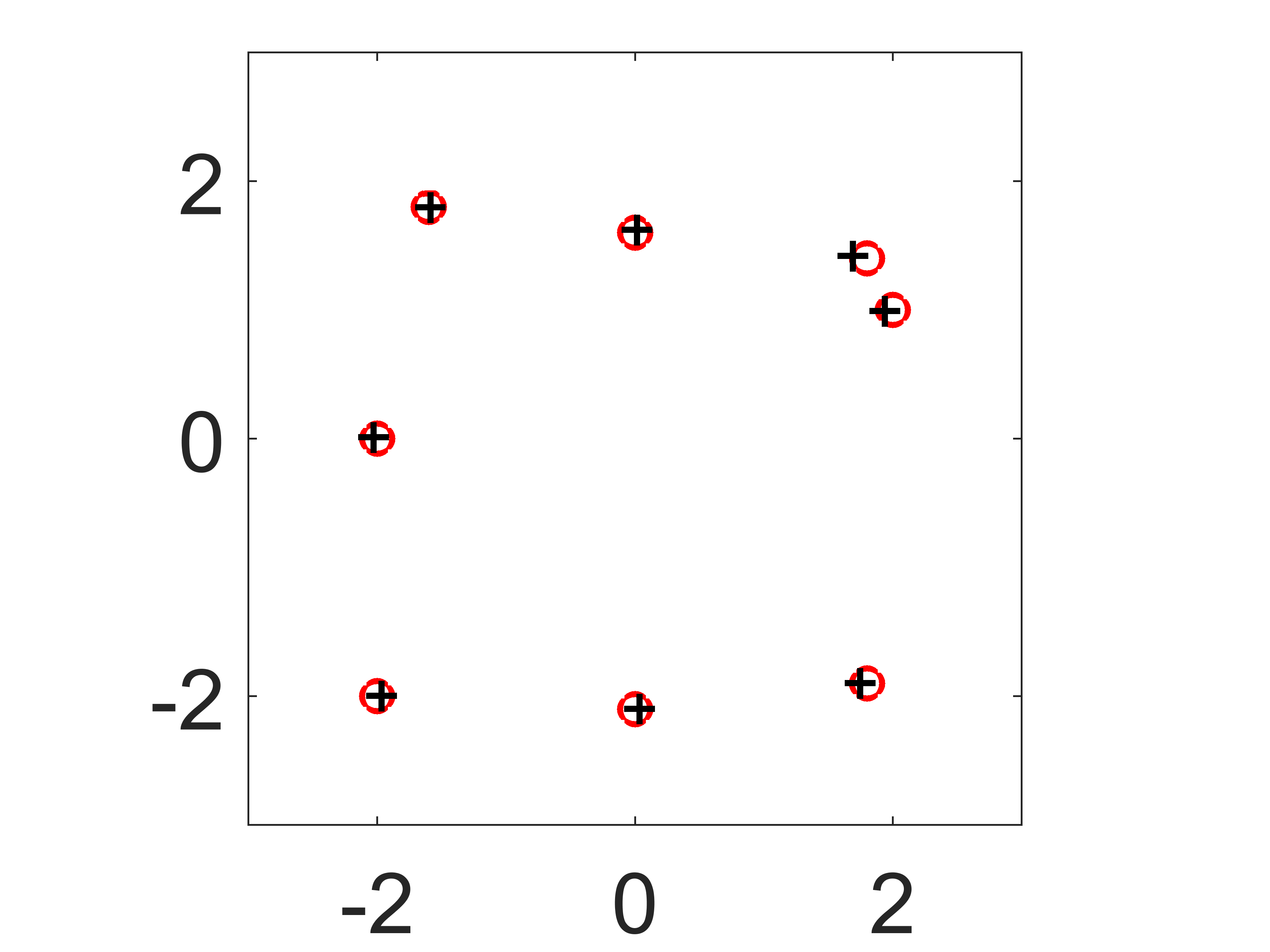}}
				\subfigure[]{\includegraphics[width=0.24\linewidth]{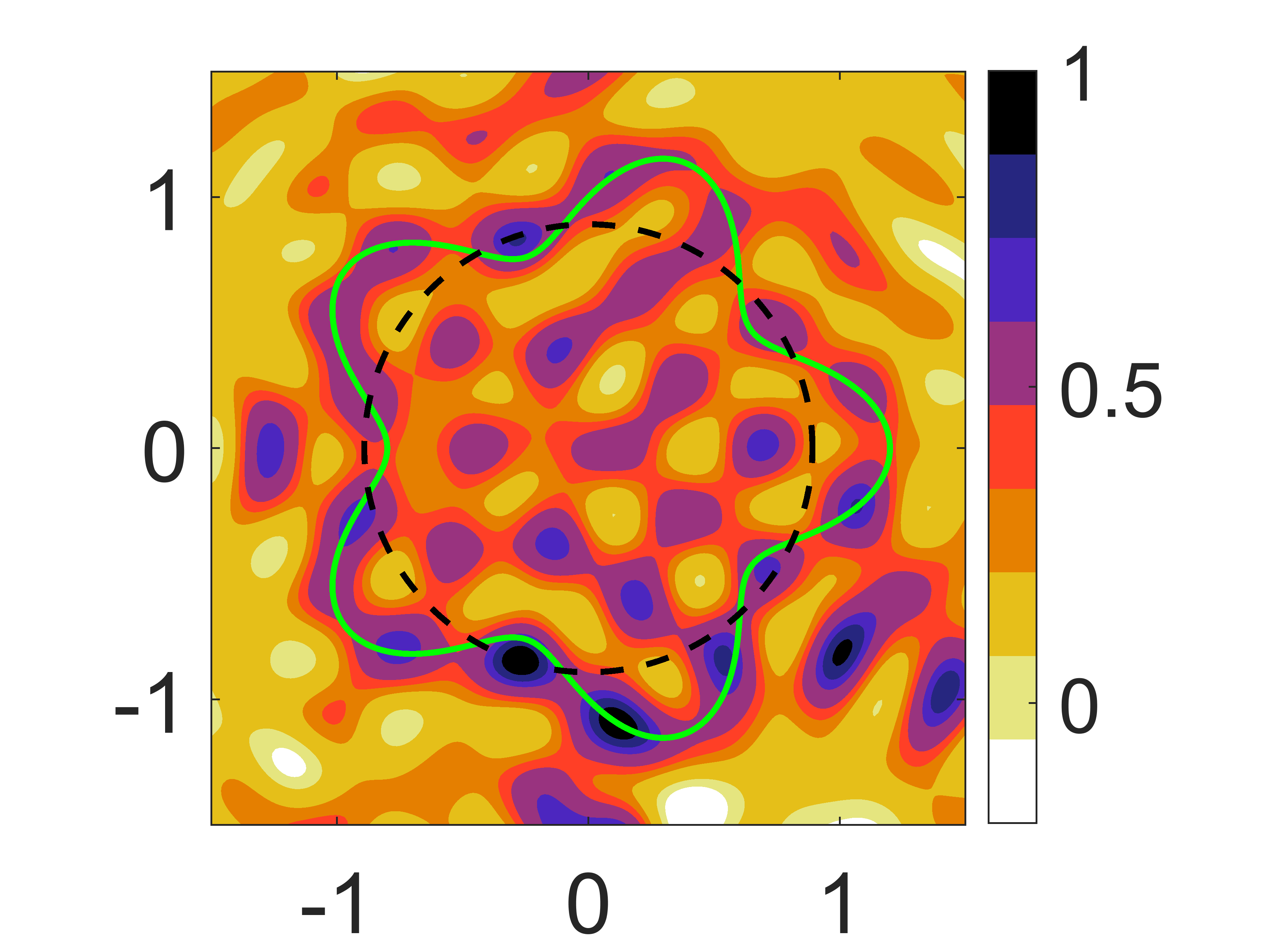}}
				\subfigure[]{\includegraphics[width=0.24\linewidth]{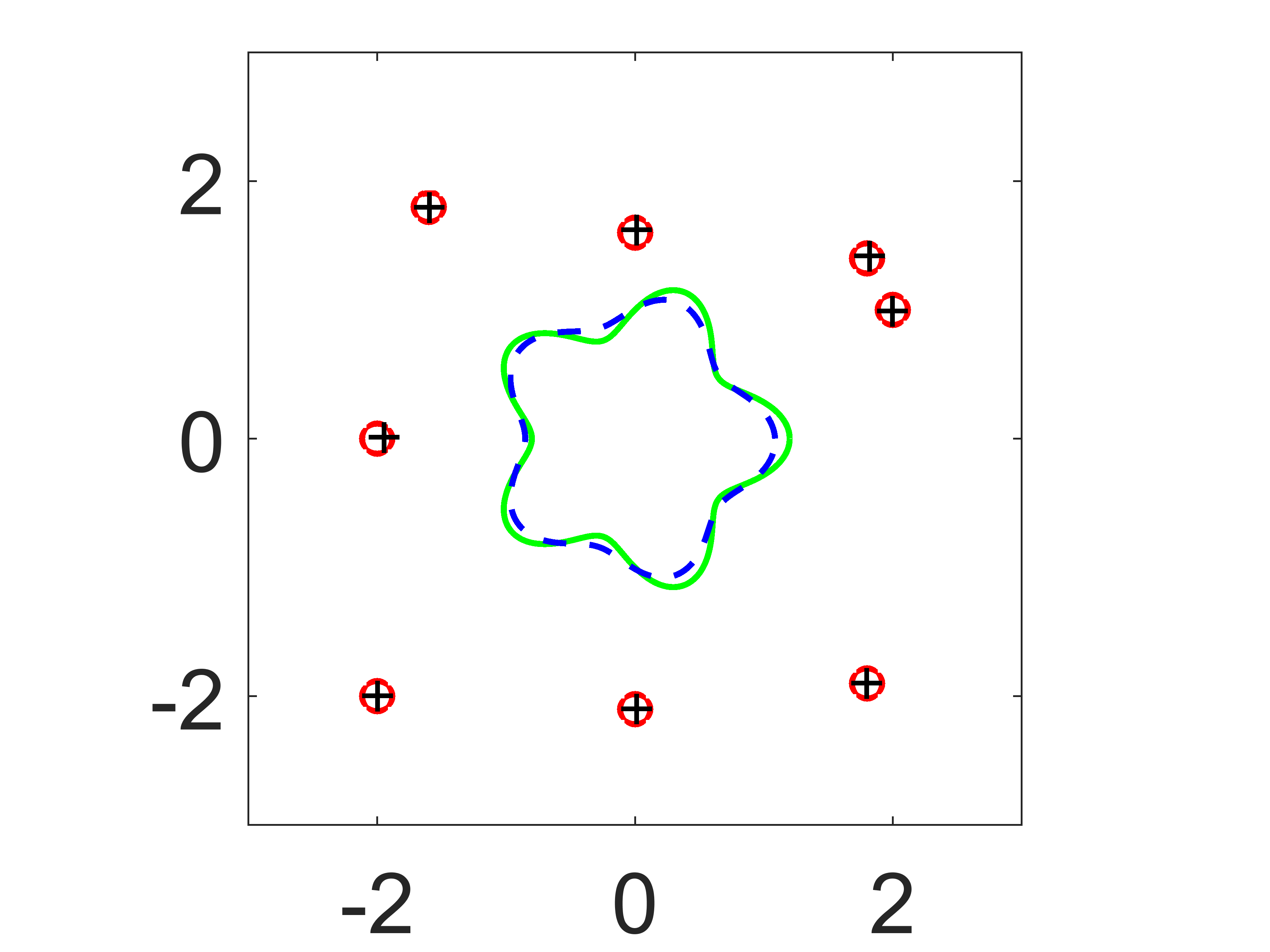}}
				\caption{Reconstruction of the starfish and different number of source points. The 1st column: problem geometry; the 2nd column: exact and reconstructed locations of the source points with DSM; the 3rd column: images of $I_D(y)$; the 4th column: reconstruction by the optimization method.} \label{fig:starfish_different}
			\end{figure}
		\end{example}

	\begin{example}\label{example2}
	In this example, we aim to investigate the accuracy of reconstruction by considering a circular obstacle as well as different number of source points. The parametric boundary of the obstacle is given by $x_C(t)=(\cos t,\sin t),\ 0\leq t< 2\pi,$ and $\alpha=10^{-8},$ $\Omega_1=[-2.5,2.5]\times[-2.5,2.5],$ $\Omega_2=[-1.2,1.2]\times[-1.2,1.2].$ 
			
	To analyze the reconstruction quantitatively, the relative errors of the boundary curves are computed as \eqref{eq:error1} with $\tau_i=t_i$.  Table \ref{tab:circle_error} shows the reconstructions in different cases by varying the wave numbers and the number of source points. In particular, a suitable choice of the parameters (for instance, $k=8$ and $N=6$) would produce a stunningly accurate reconstruction. Figure \ref{fig:circle} exhibits the reconstruction of the circle and $N(=2,4,6,8)$ source points with wave number $k=5$. 
			\begin{table}[htpb]
				\caption{Relative $L^2$ errors for reconstructions of the circle with different numbers of source points.}
				\label{tab:circle_error}
				\centering
				\begin{tabular}{ccccc}
					\toprule
					$N$  &    2     &    4     &    6     &    8     \\ \midrule
					$k=5$ & $7.32\%$ & $2.30\%$ & $2.47\%$ & $1.24\%$ \\
					$k=8$ & $5.44\%$ & $1.38\%$ & $0.72\%$ & $0.76\%$ \\ \bottomrule
				\end{tabular}
			\end{table}		
	\begin{figure}[htpb]
		\centering
		\subfigure[]{\includegraphics[width=0.24\linewidth]{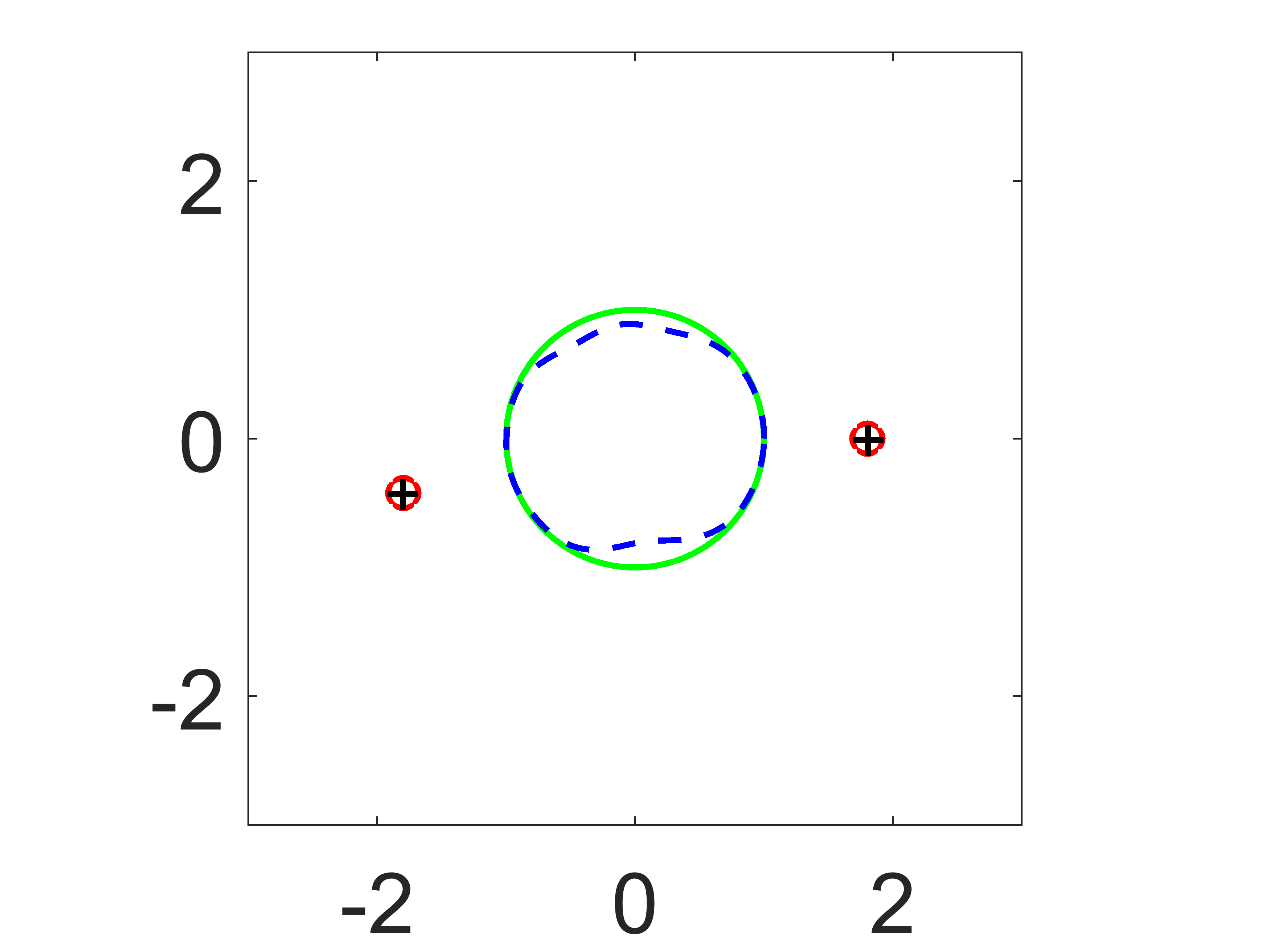}}
		\subfigure[]{\includegraphics[width=0.24\linewidth]{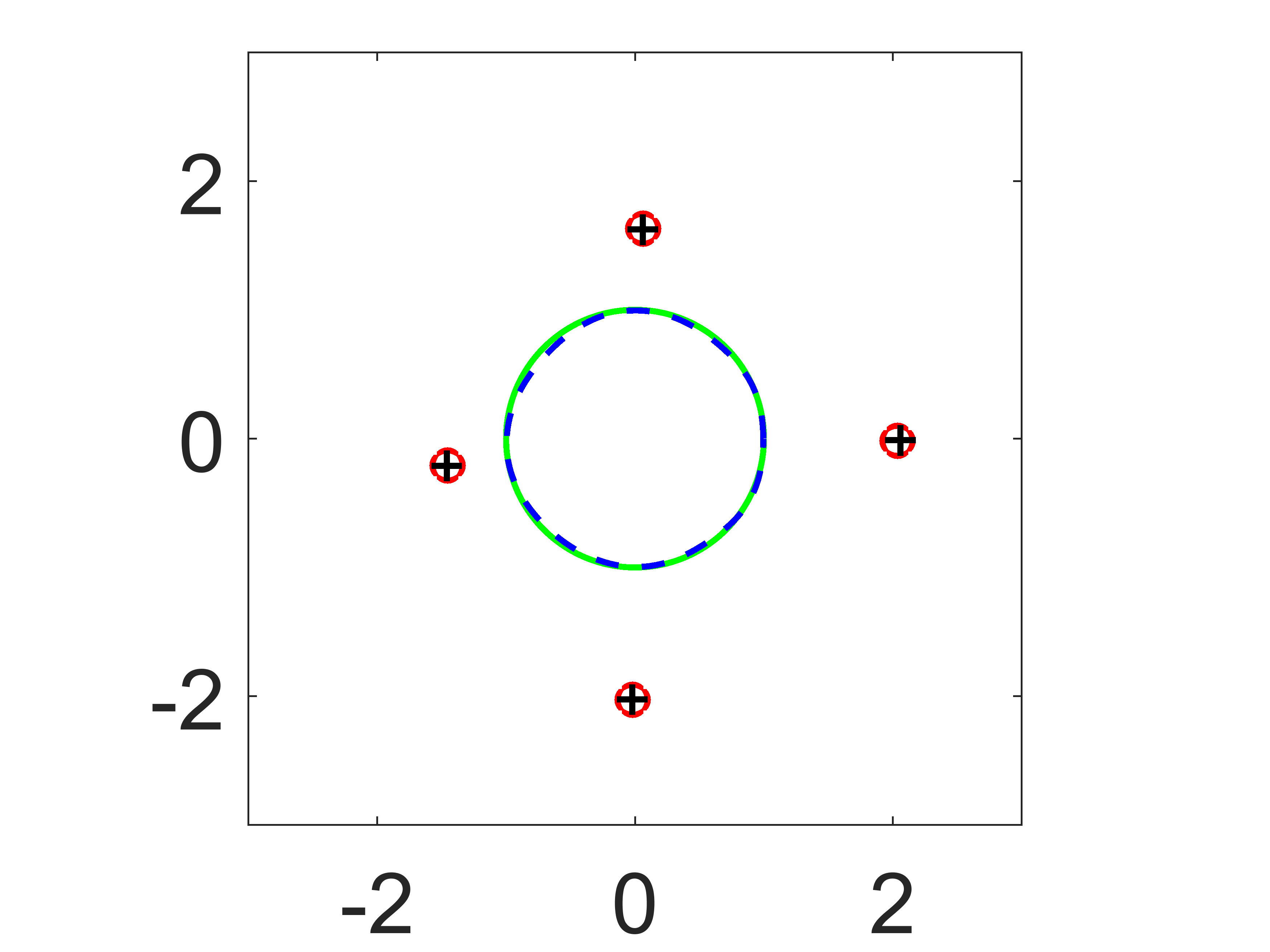}}
		\subfigure[]{\includegraphics[width=0.24\linewidth]{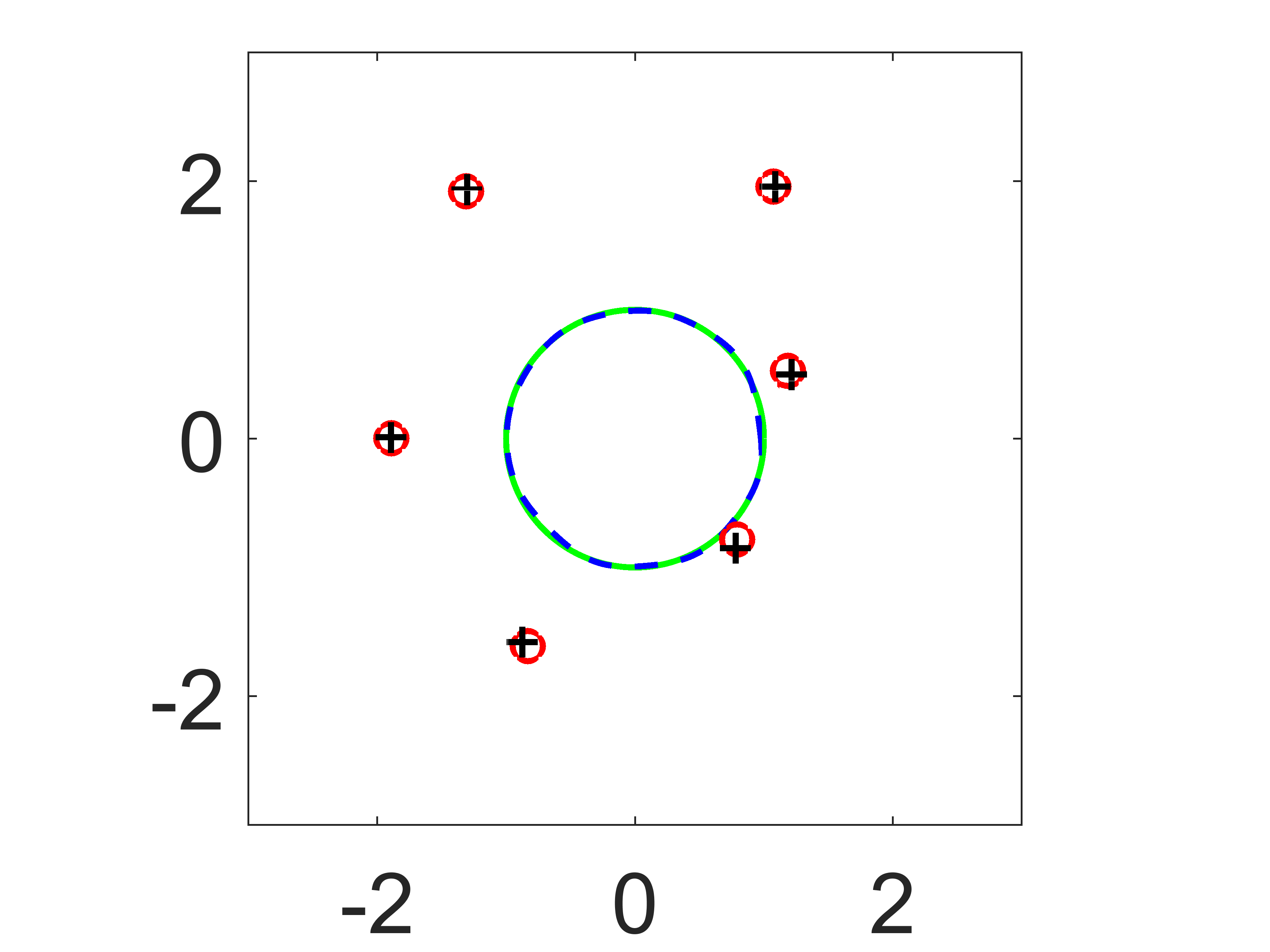}}
		\subfigure[]{\includegraphics[width=0.24\linewidth]{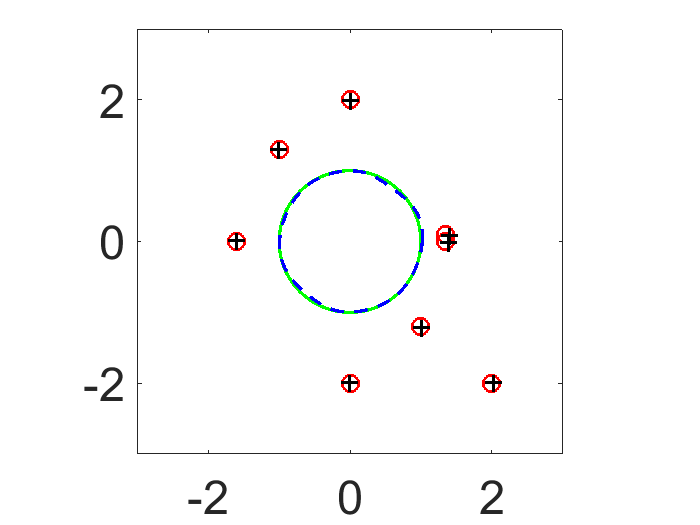}}
		\caption{Reconstructions of the circle and $N$ source points. (a) $N=2$; (b) $N=4$; (c) $N=6$; (d) $N=8$.}\label{fig:circle}
	\end{figure}
		
	Further, we test the effectiveness of the DSM proposed in Section \ref{sec:DSM} for choosing the initial guess. The parameters are chosen to be $\alpha=10^{-9}, k=8, N=8$. The exact source points are given by
	\[
		S_3=\{(2,1),(1.8,1.4),(0,1.6),(-1.6,1.8),(-2,0),(-2,-2),(0,-2.1),(1.8,-1.9)\}.
	\]
	For comparison, we consider the following two sets of artificially chosen initial guess
	\begin{align*}
	S_4&=\{(2.1,0.9),(1.85,1.3),(-0.1,1.75),(-1.7,1.7),\\
	&\quad\ (-1.9,0.05),(-1.94,-1.9),(0,-2.13),(1.8,-1.8)\},
	\end{align*}
	and 
	\begin{align*}
	S_5&=\{(2.3,1.3),(2.1,1.7),(0.3,1.3),(-1.3,1.5),(-2.3,-0.3),\\
	 &\quad\ (-2.3,-2.3),(-0.3,-1.8),(1.5,-1.6)\},
	\end{align*}
	respectively. Taking $S_4$ and $S_5$ as the initial guesses for the source points respectively, we then compare the numerical results with the case where the initial guess is determined by the DSM. The relative $L^2$ errors for $k=5$ are listed in Table \ref{tab:circle_guess}. Furthermore, we present the reconstructions of the circle as well as the source points in Figure \ref{fig:circle2}. From Table \ref{tab:circle_guess} and Figure \ref{fig:circle2}, one can observe that the DSM plays a significant role in the reconstruction by providing a good initial guess for the optimization process, which is a key step for achieving a satisfactory reconstruction. If the initial source points somewhat deviate from the true ones, then the reconstruction would drastically deteriorate, see Figure \ref{fig:circle2}(b). Therefore, results of the DSM indeed exert an immense impact on the effectiveness of the optimization. And a good initial knowledge generated by a high-resolution sampling grid $\mathcal{T}_1$ is the essential prerequisite for achieving the high-quality reconstruction.
		\begin{table}[htpb]
			\caption{Relative $L^2$ errors for reconstructions of the circle subject to different initial guesses.}
			\label{tab:circle_guess}
			\centering
		\begin{tabular}{cccc}\hline
			\toprule
			& $S_4$    &  $S_5$  &  DSM \\ \midrule
			$k=5$&$5.26\%$  &$48.87\%$ & $1.78\%$\\
			$k=8$&$3.35\%$	&$31.41\%$ & $1.22\%$
			\\\bottomrule
		\end{tabular}
		\end{table}
	
	\begin{figure}[htp]
		\centering
		\subfigure[]{\includegraphics[width=0.3\linewidth]{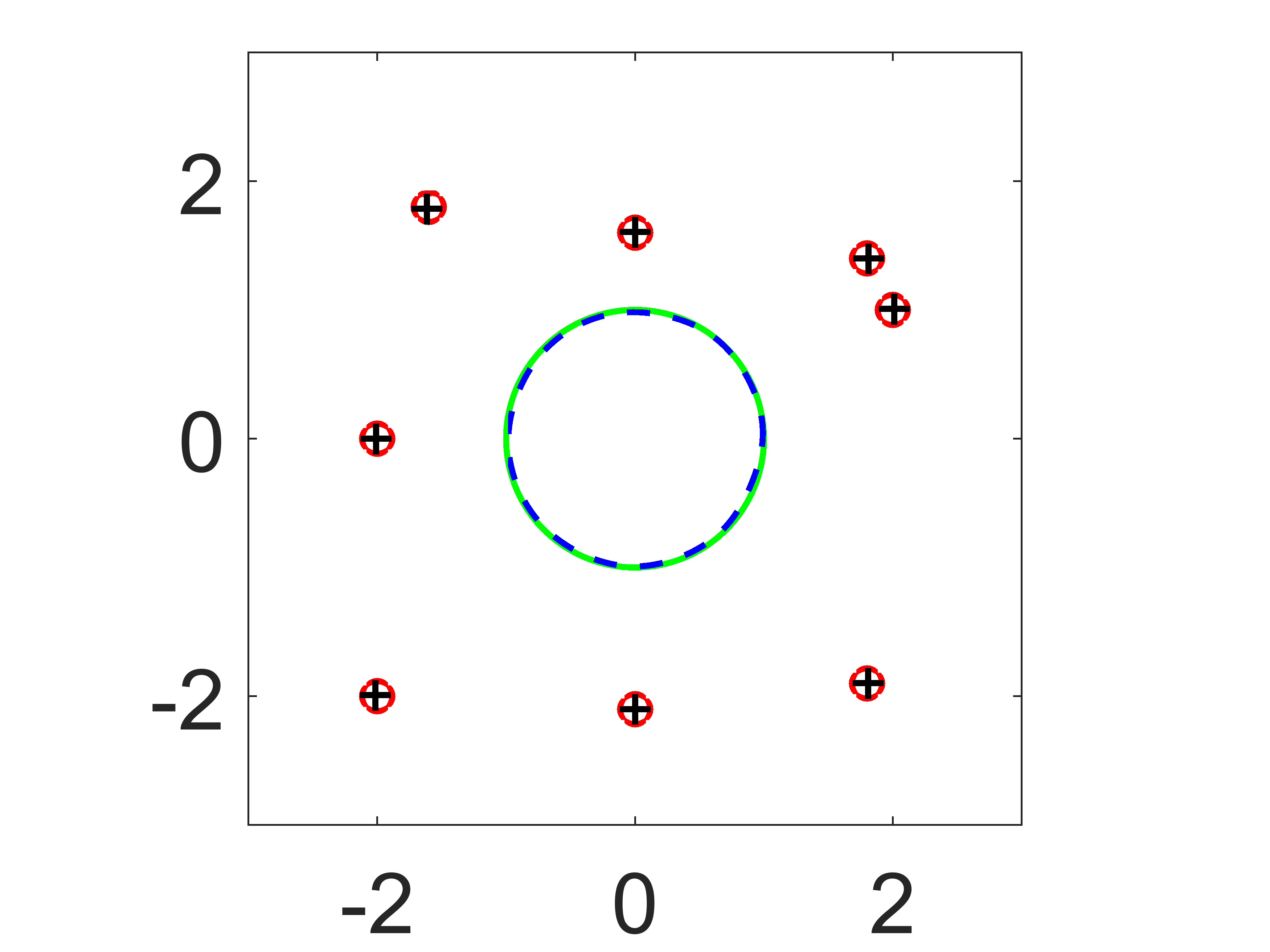}}
		\subfigure[]{\includegraphics[width=0.3\linewidth]{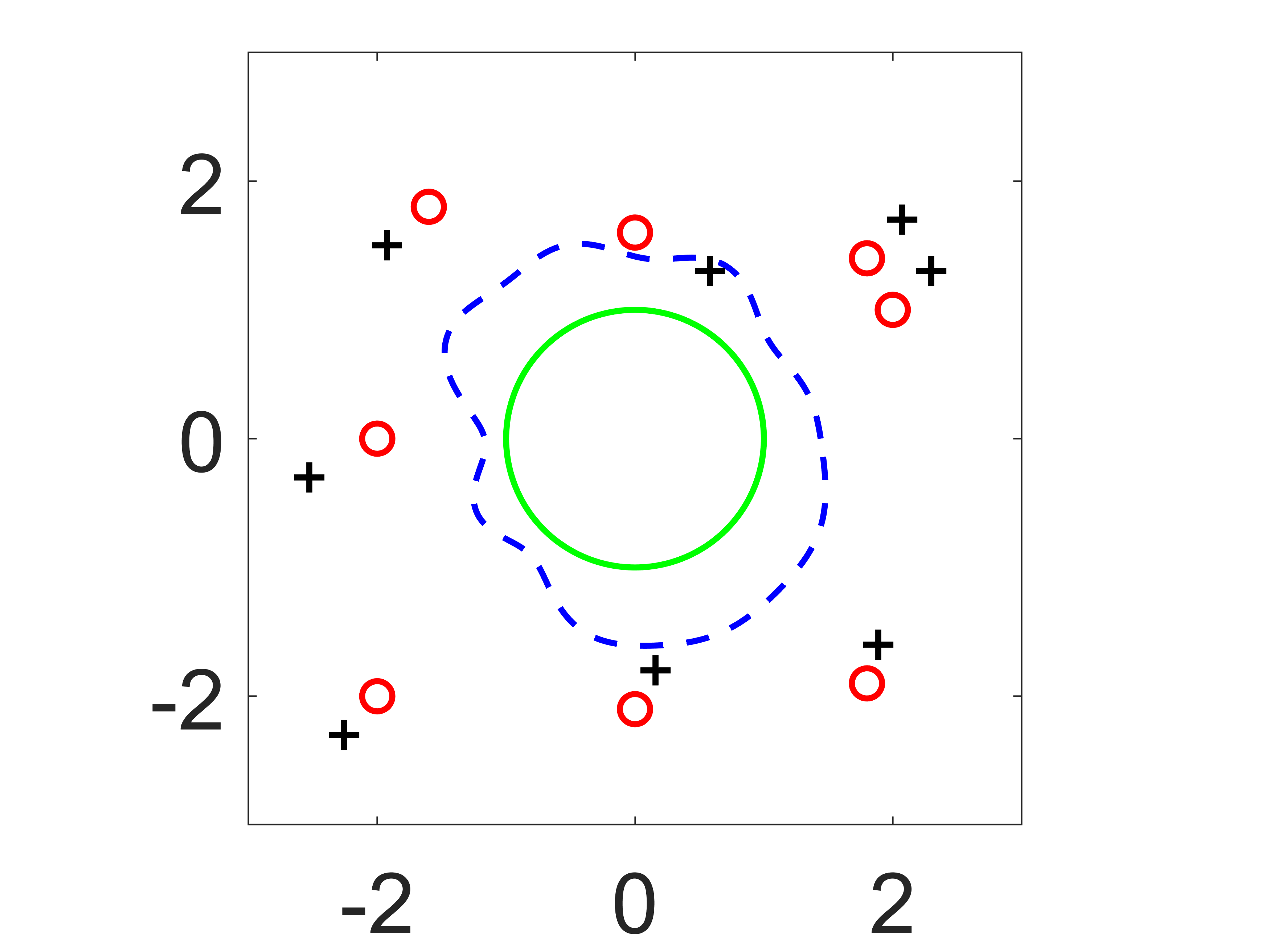}}
		\subfigure[]{\includegraphics[width=0.3\linewidth]{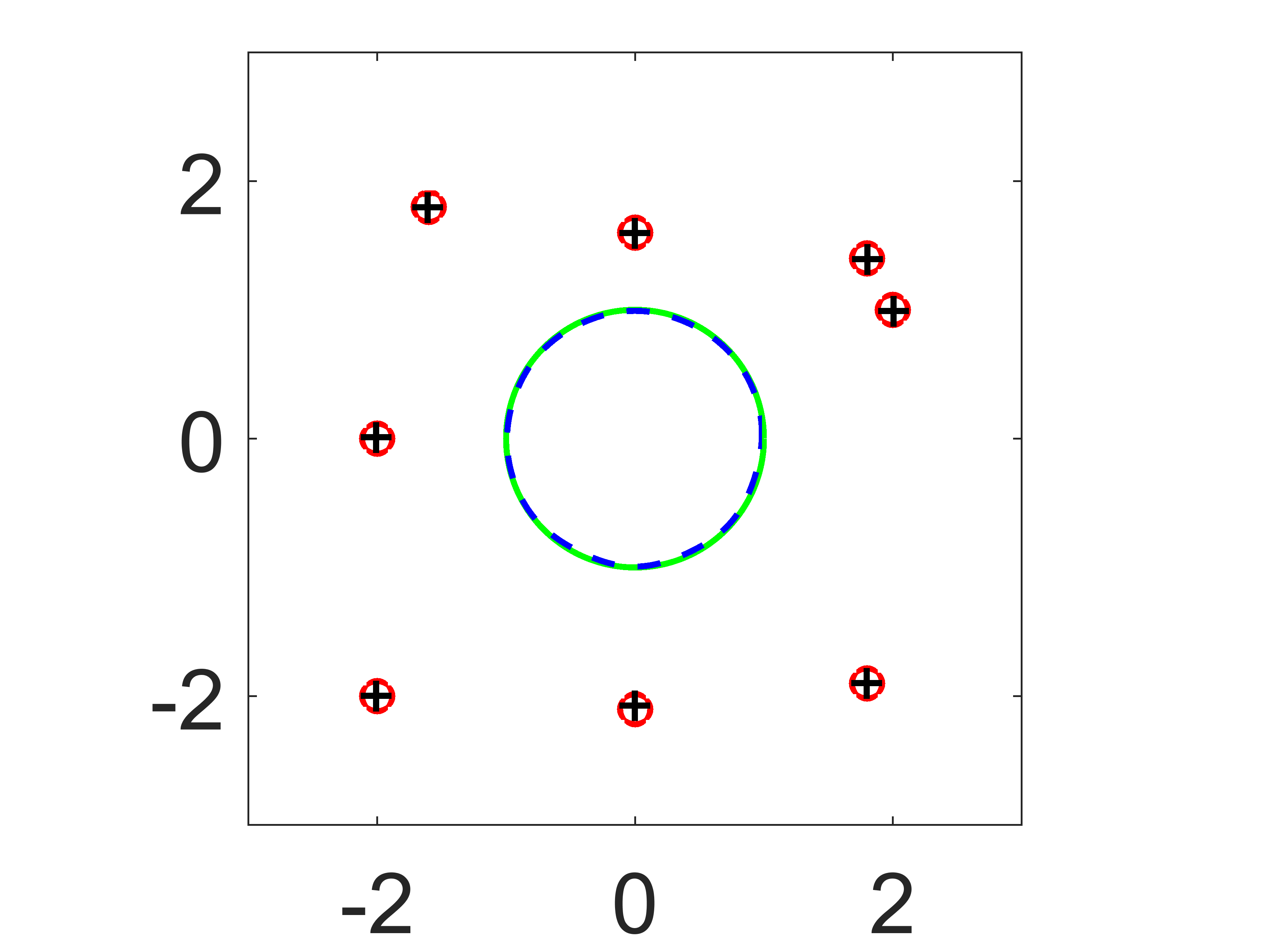}}
		\caption{Reconstructions of the circle and $8$ source points with different initial guesses. (a) $S_4;$ (b) $S_5;$ (c) DSM.}\label{fig:circle2}
	\end{figure}
		\end{example}
		
	\begin{example}\label{example3}
	This example is designed to check the validity of our method when the boundary of obstacle is not star-like. Now we consider the reconstruction of two kite-shaped obstacles with boundaries $\partial D$ described by the parametric representations
	\begin{align*}
		 & \text{kite-I:}\ x_K^{(1)}:=(x_1^{(1)},x_2^{(1)})=(\cos t+0.15\sin t+0.35\cos 2t-0.35, 1.2\sin t), \\
		 & \text{kite-II:}\ x_K^{(2)}:=(x_1^{(2)},x_2^{(2)})=(\cos t+0.65\cos 2t-0.65, 1.5\sin t),
	\end{align*}
	respectively. In this example, the sampling domains are taken to be $\Omega_1=[-2.5,2.5]\times[-2.5,2.5],$ $\Omega_2=[-1.5,1.5]\times[-1.5,1.5]$. To compute the accuracy of reconstruction, the relative errors of the boundary curves are computed by \eqref{eq:error1} with
	\[
	\tau_i^{(\ell)}=\text{arc}\tan\frac{x_2^{(\ell)}(t_i)}{x_1^{(\ell)}(t_i)},\quad \ell=1,2.
	\]
	
	In Figure \ref{fig:kite1_NF}, we take the ansatz function space of different dimensions and present the reconstruction of the kite-I with $k=5$ and  $\alpha=10^{-6}.$ By taking different $M(=2,4,6,8,12,20)$ in \eqref{eq:rM}, we compute the reconstruction errors defined by \eqref{eq:error1}. From the errors in Table \ref{table:error_kite}, we observe that the accuracy of inversion is insensitive to moderately large $M$. Except for the cases $M=2$ and $M=4$, the reconstruction is almost negligibly influenced by the dimension of $U_M$.
			
	Next, we fix $M=8$ and reconstruct the two kites. In Table \ref{table_kite}, we provide the errors for reconstructing the kites with different number of source points. These results illustrate that the reconstructions of kite-I are superior to those of kite-II, which is probably due to the fact that the boundary fluctuation of  kite-I is relatively mild in comparison to that of kite-II.
	
	In Figure \ref{fig:kite1_different_sourcepoints},  we plot the reconstruction of kite-II with $N(N=2,6,8,10)$ source points. It can be found that both the source points and the obstacle are reconstructed well. So the proposed method also works for the recovery of non-starlike scatterers.
			\begin{figure}[htpb]
				\centering
				\subfigure[]{\includegraphics[width=0.24\linewidth]{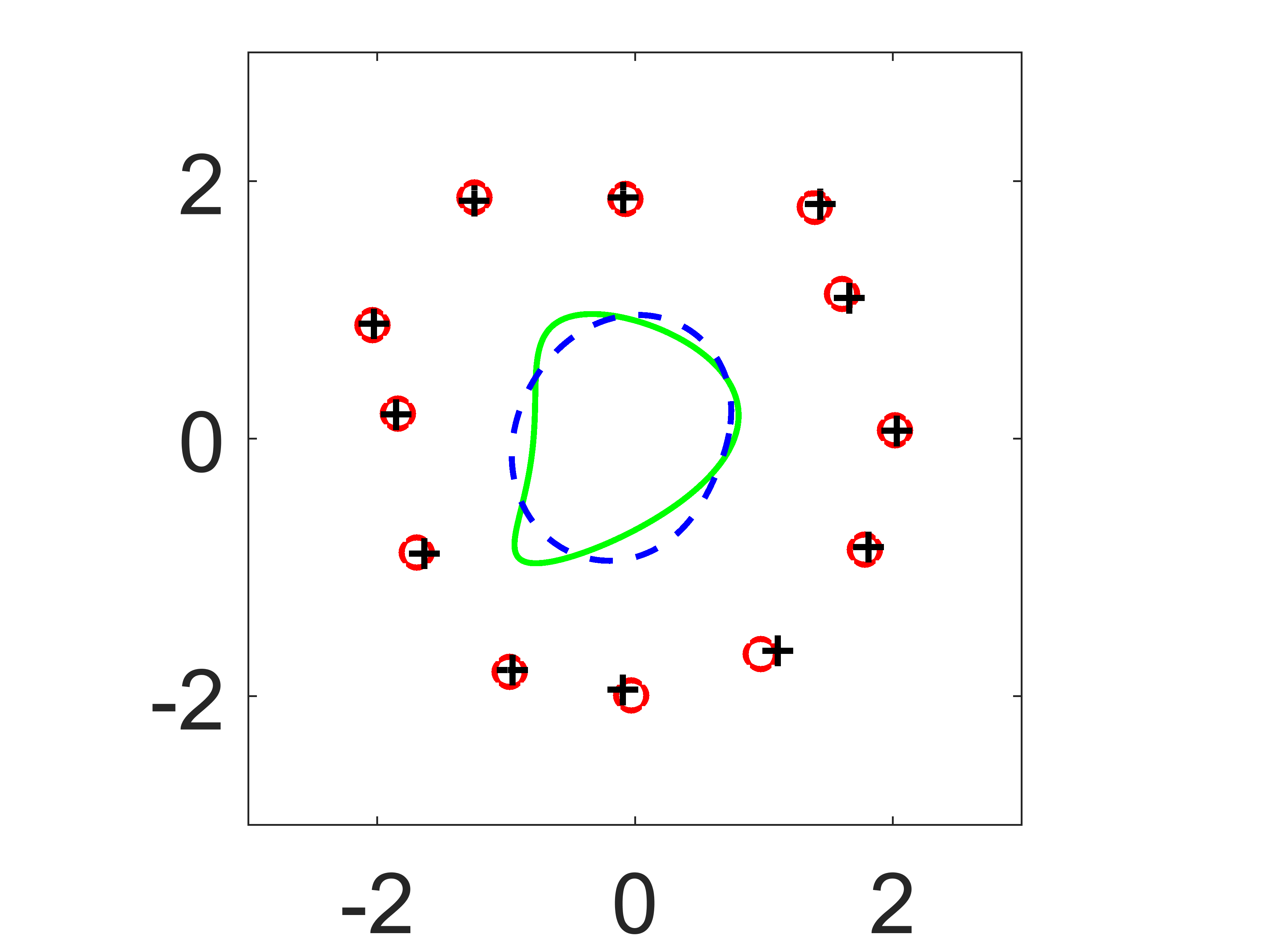}}
				\subfigure[]{\includegraphics[width=0.24\linewidth]{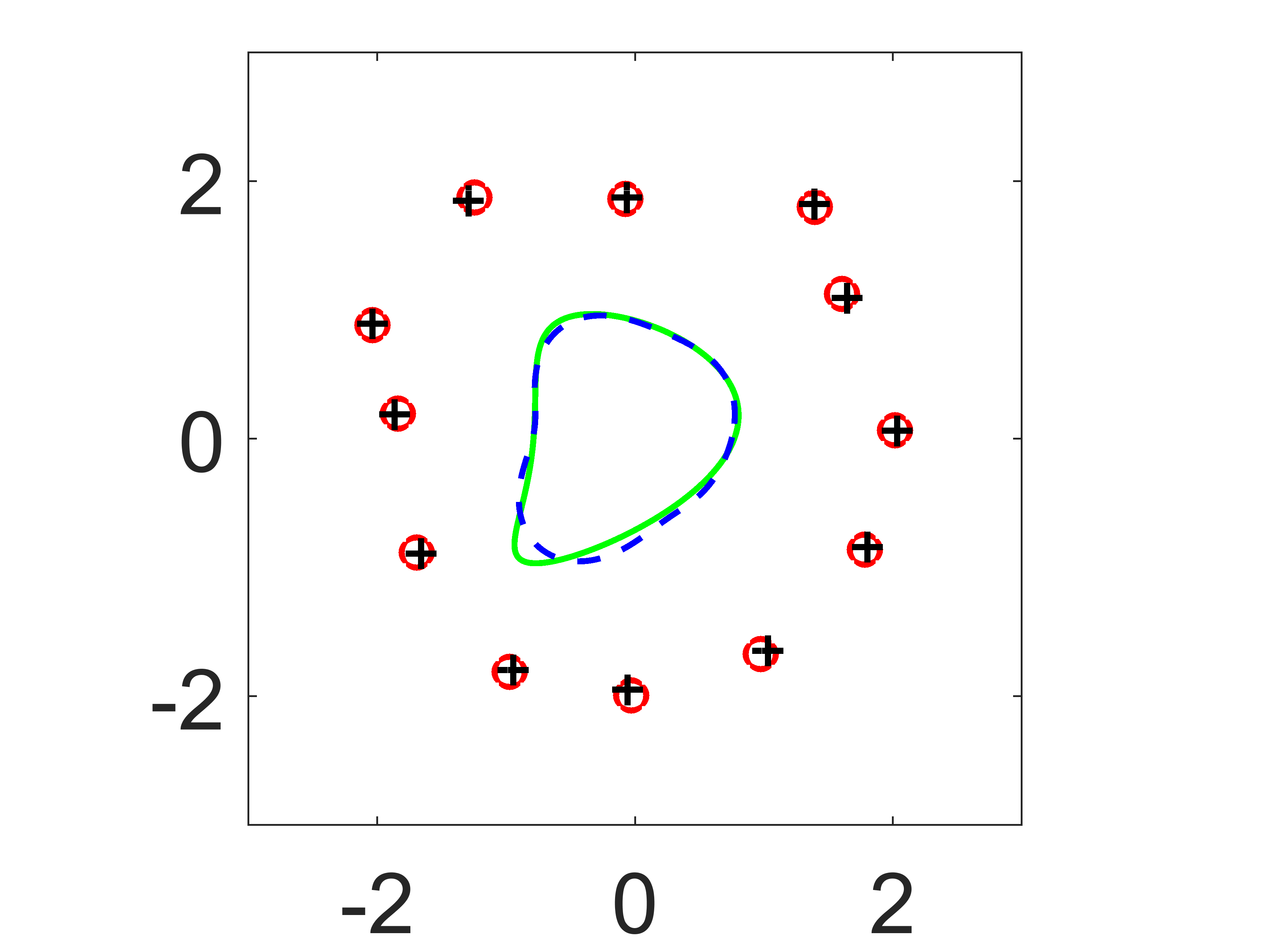}}
				\subfigure[]{\includegraphics[width=0.24\linewidth]{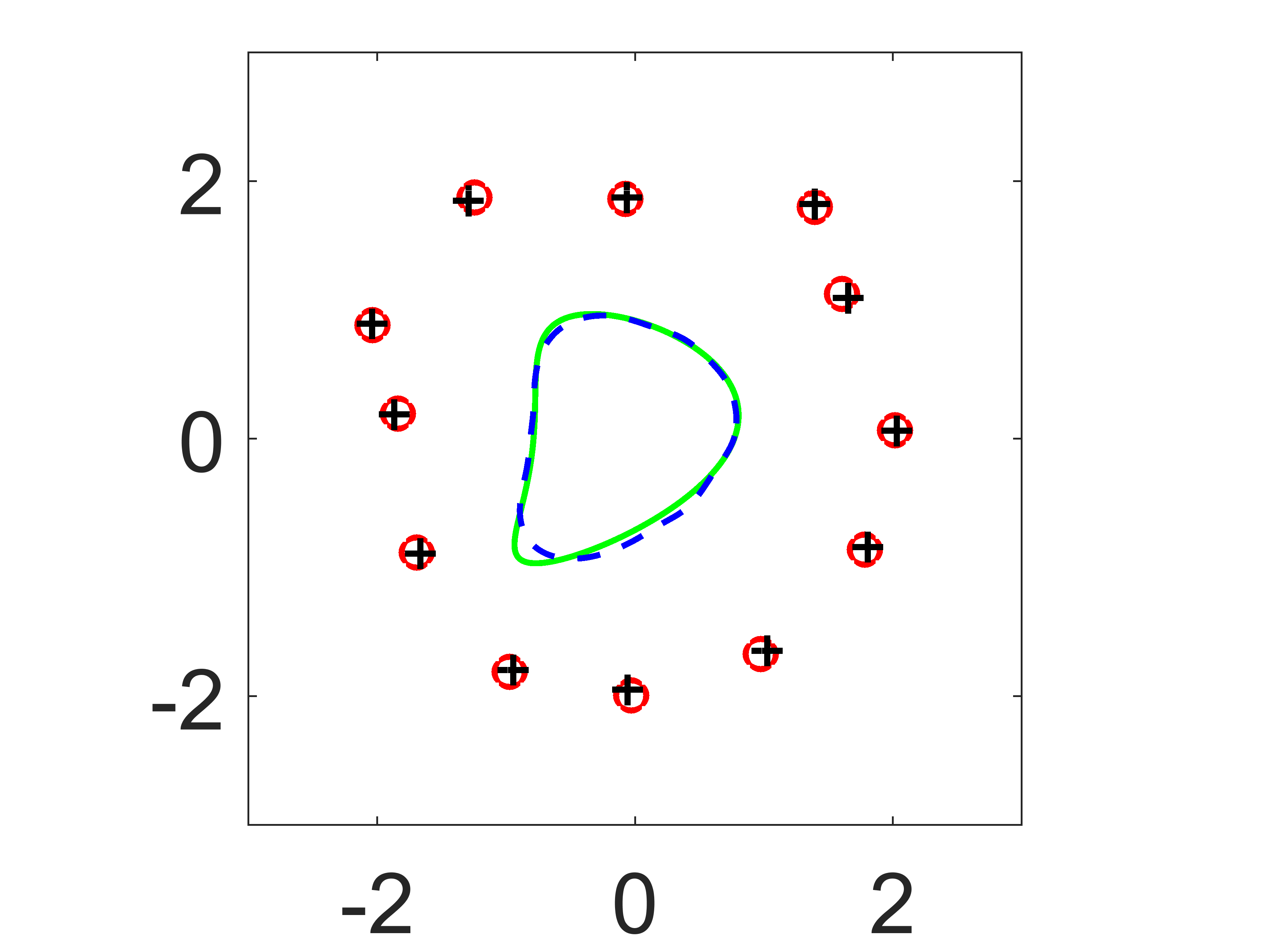}}
				\subfigure[]{\includegraphics[width=0.24\linewidth]{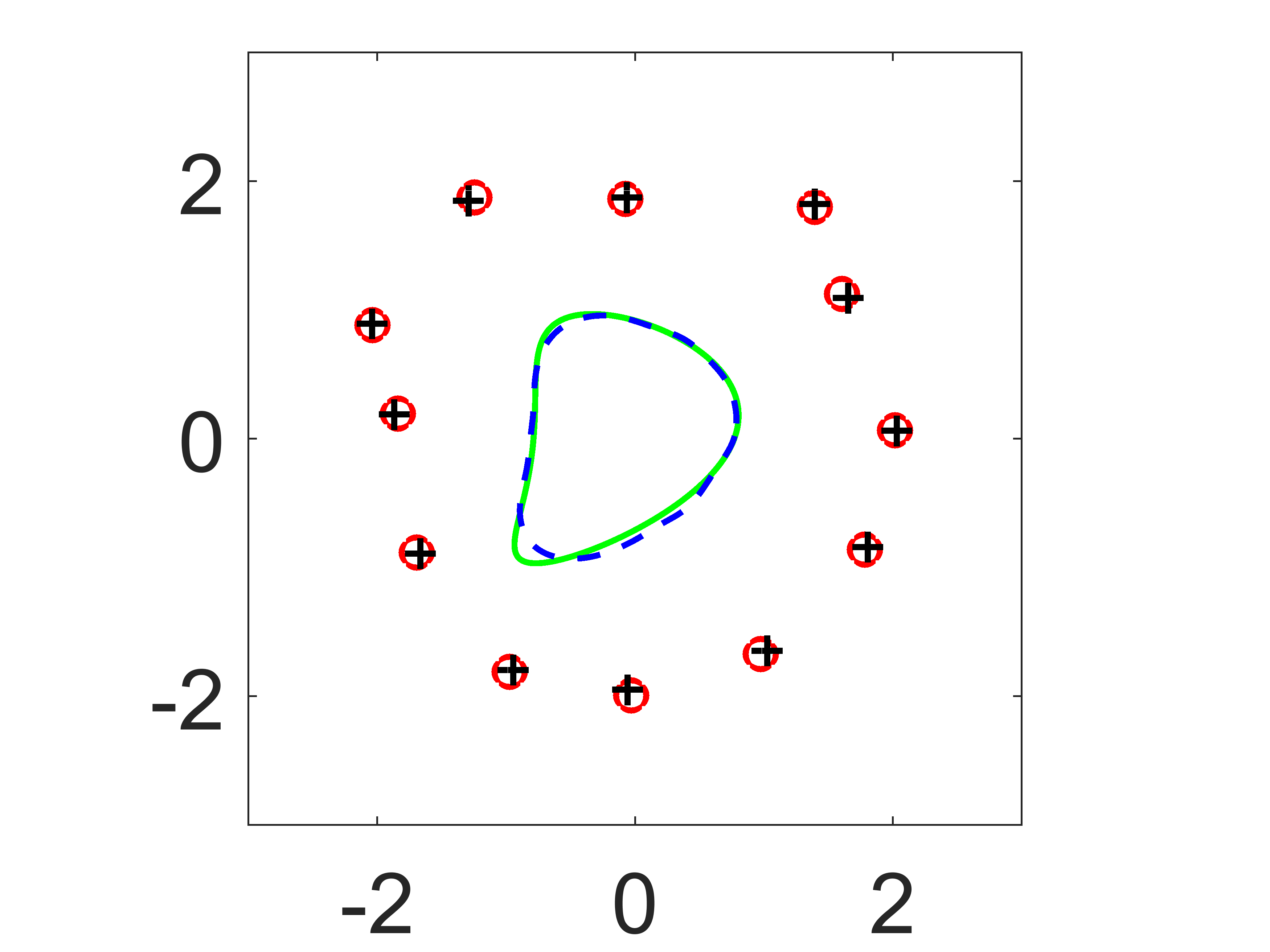}}
				\caption{Reconstructions of the kite-I and $12$ source points with subspace $U_M$ of different dimensions: (a) $M=2;$ (b) $M=4;$ (c) $M=12;$ (d) $M=20.$}\label{fig:kite1_NF}
			\end{figure}
			
			\begin{table}[htpb]
				\caption{Relative $L^2$ errors for reconstructions of the kite-I with 12 source points.}
				\label{table:error_kite}\centering
				\begin{tabular}{ccccccc}
					\toprule
					$M$  &     2     &    4     &    6     &    8     &    12    &    20    \\ \midrule
					$E_D$ & $14.39\%$ & $6.68\%$ & $5.97\%$ & $5.77\%$ & $5.75\%$ & $5.75\%$ \\ \bottomrule
				\end{tabular}
			\end{table}
			
			\begin{table}[htpb]
				\caption{Relative $L^2$ errors for reconstructions of the kite-shaped obstacles.}\label{table_kite}
				\centering
				\begin{tabular}{ccccc}
					\toprule
					&    \multicolumn{2}{c}{Error of kite-I}     &        \multicolumn{2}{c}{Error of kite-II}        \\
					\cmidrule(r){2-3} \cmidrule(r){4-5}
					$N$ & $k=3,\alpha=10^{-10}$ & $k=5,\alpha=10^{-9}$ & $k=3,\alpha=10^{-9}$ & $k=5,\alpha=1\times10^{-8}$ \\ \midrule
					$2$                     &       $10.29\%$       &       $7.00\%$       &      $10.57\%$       &           $8.27\%$           \\
					$4$                     &       $7.72\%$        &       $4.81\%$       &      $11.30\%$       &           $9.96\%$           \\
					$6$                     &       $7.33\%$        &       $5.07\%$       &      $11.46\%$       &          $12.98\%$           \\
					$8$                     &       $7.33\%$        &       $4.63\%$       &      $11.87\%$       &          $10.91\%$           \\
					$10$                    &       $6.93\%$        &       $4.84\%$       &      $12.04\%$       &          $10.56\%$           \\ \bottomrule
				\end{tabular}
			\end{table}
			
			\begin{figure}[htpb]
				\centering
				\subfigure[]{\includegraphics[width=0.24\linewidth]{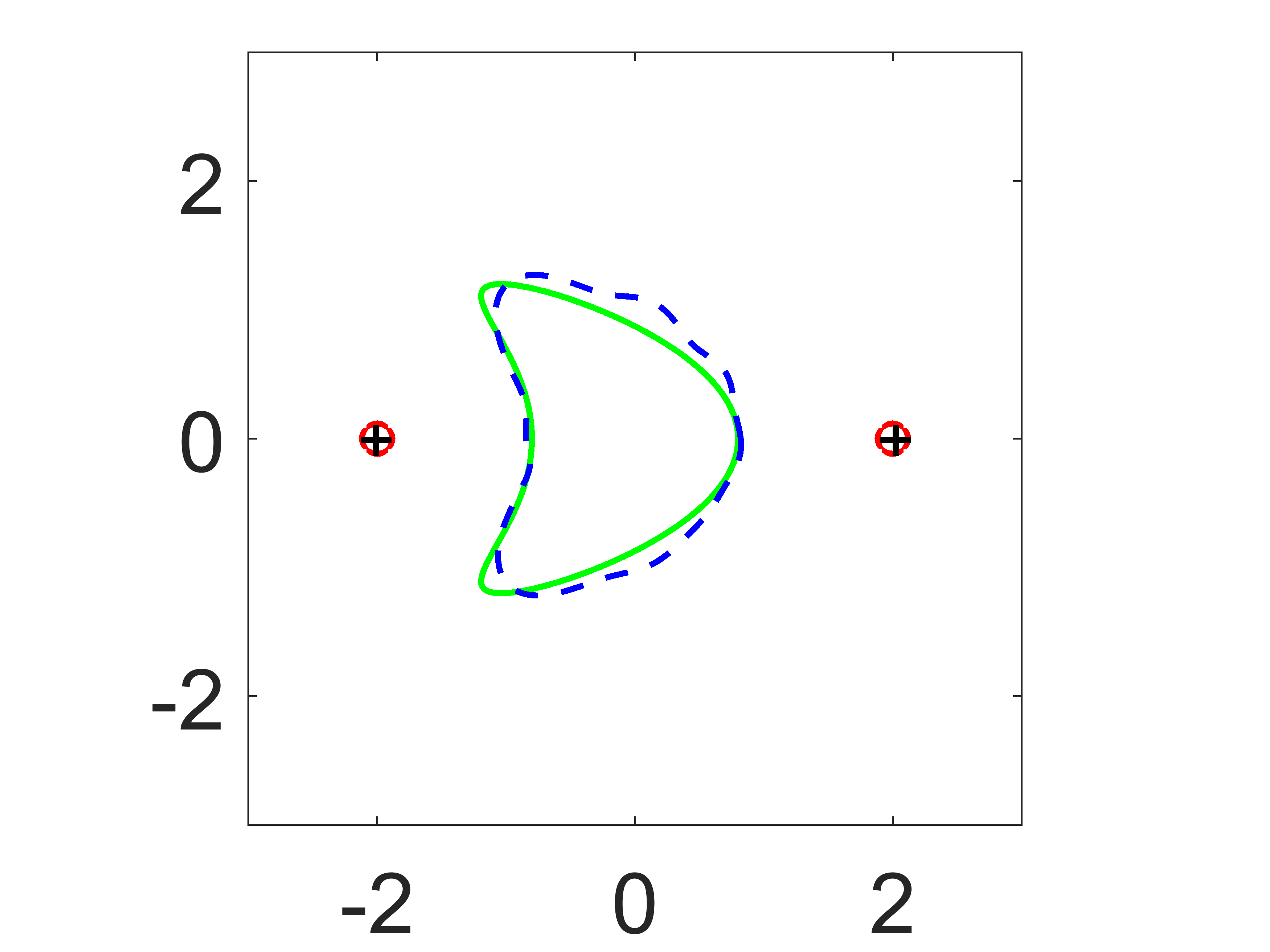}}
				\subfigure[]{\includegraphics[width=0.24\linewidth]{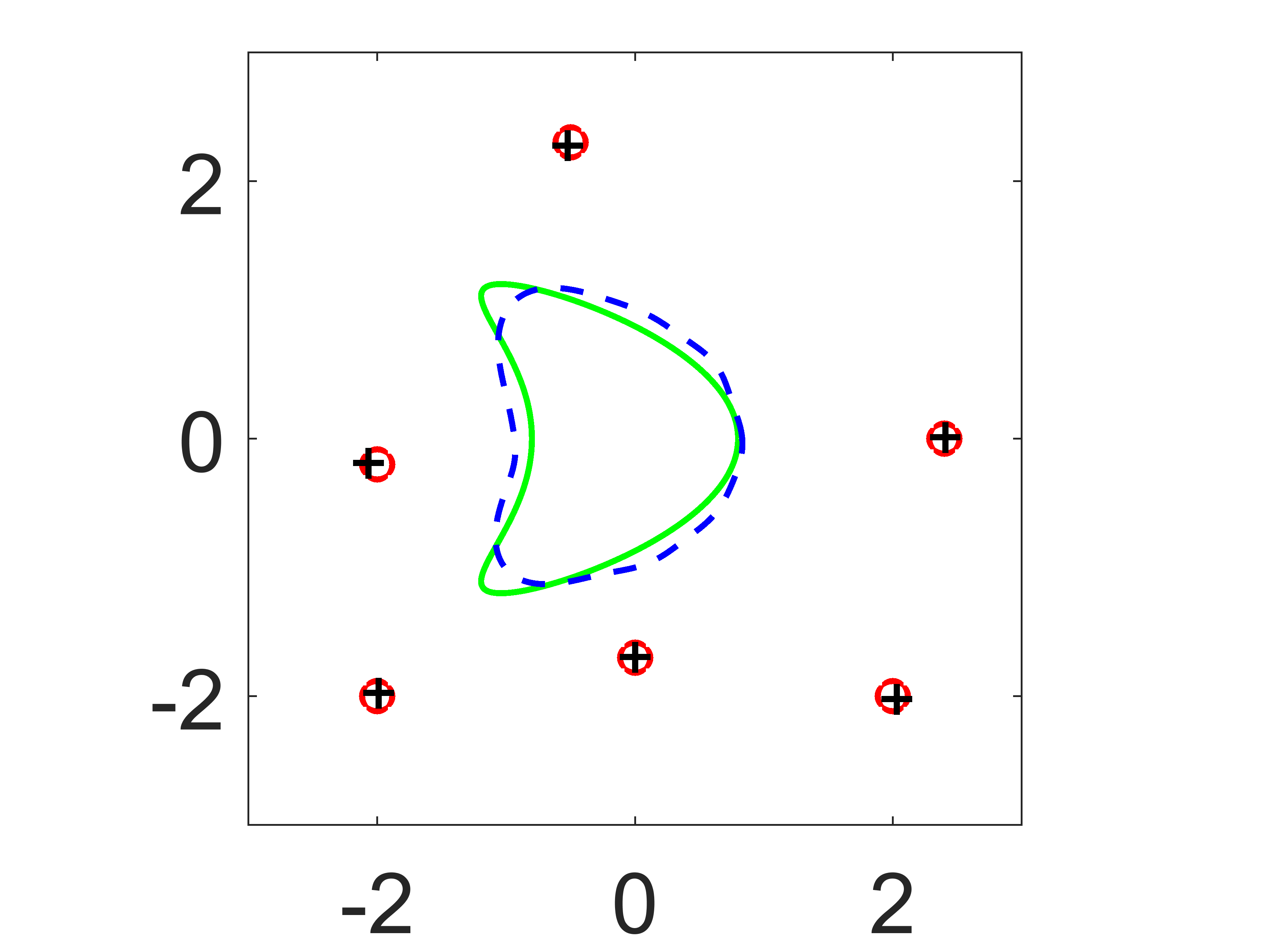}}
				\subfigure[]{\includegraphics[width=0.24\linewidth]{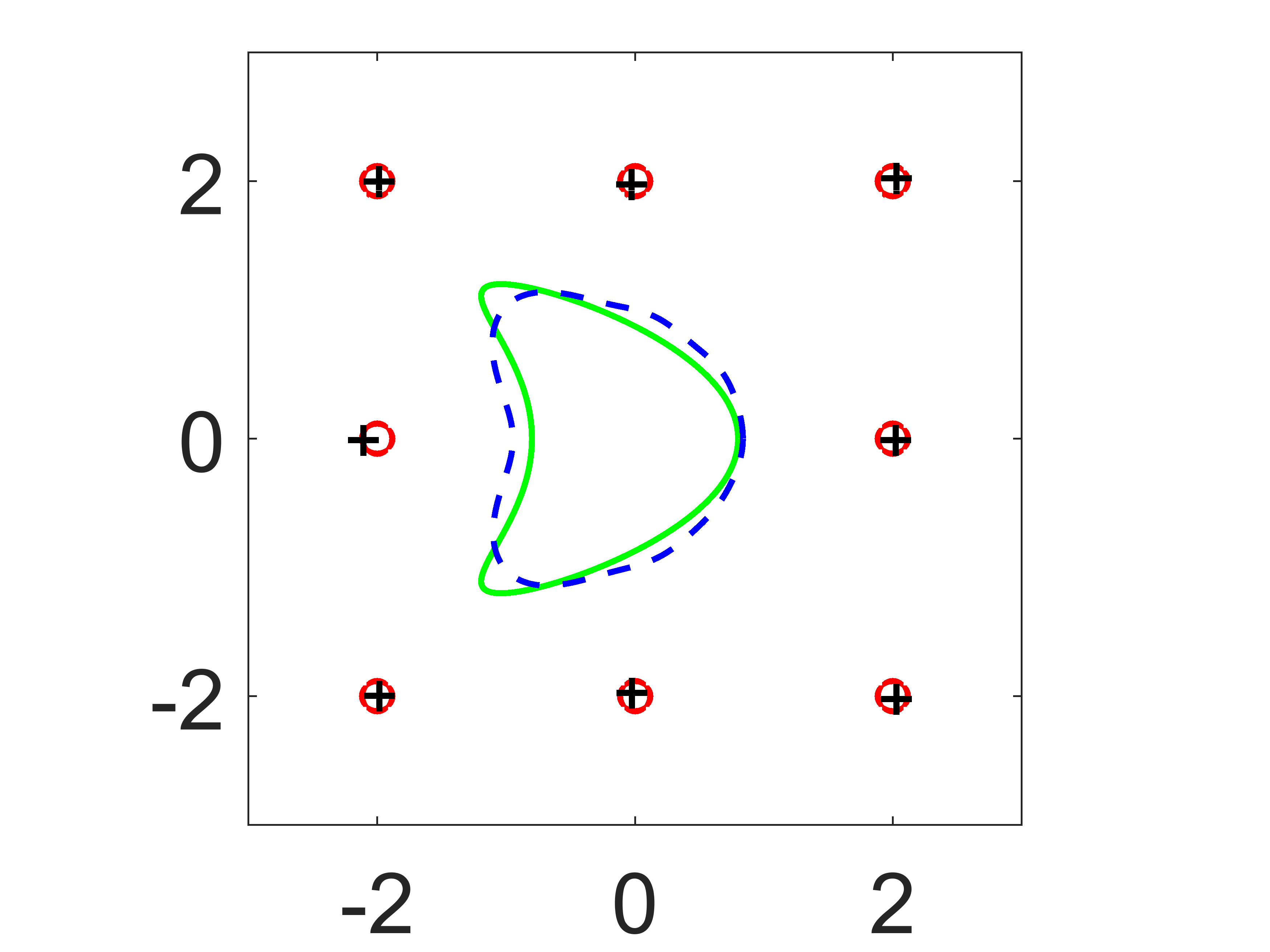}}
				\subfigure[]{\includegraphics[width=0.24\linewidth]{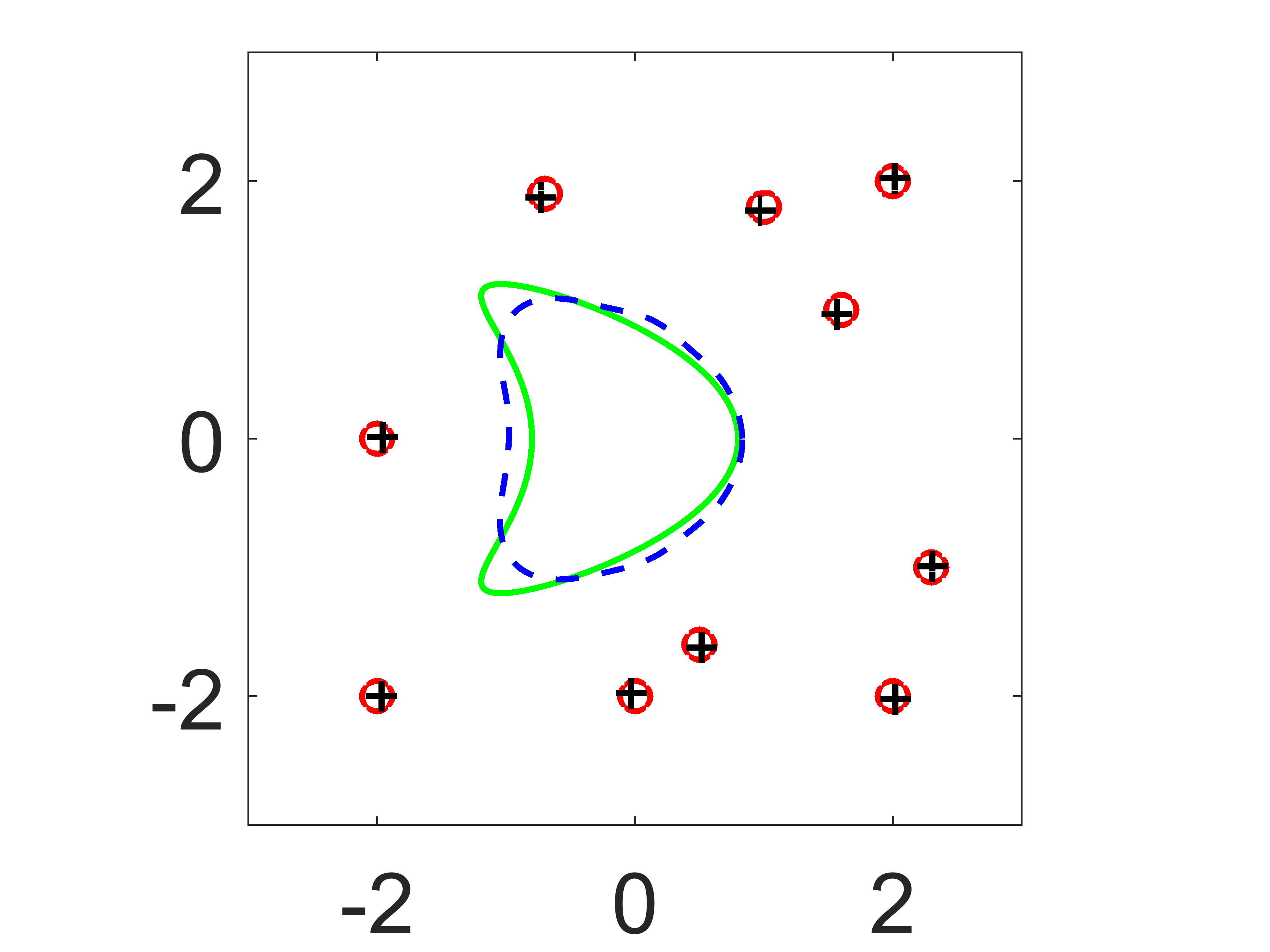}}
				\caption{Reconstructions of kite-II and source points of different number: (a) $N=2;$ (b) $N=6;$ (c) $N=8;$ (d) $N=10.$}\label{fig:kite1_different_sourcepoints}
			\end{figure}
		\end{example}
		
	\begin{example}\label{example4}
	In the last example, we plan to reconstruct the obstacle-source pair with limited-aperture measurements. Motivated by realistic applications where only partial data is available, we test the proposed method when the total field is only accessed on a portion of the measurement curve. In this example, we collect the total field recorded by the receivers located on a part of $\Gamma_R$ with the aperture chosen to be $\theta=3\pi/2$ and $\pi,$ respectively. The regularization parameter is chosen to be $\alpha=10^{-6}$. The sampling domains are $\Omega_1 = [-2.5,2.5]\times[-2.5,2.5]$ and $\Omega_2 = [-1.5,1.5]\times[-1.5,1.5]$. 
					
	In Table \ref{table_limited}, we list the relative $L^2$ errors of the boundary curves with respect to different parameters $k$, $\theta$ and $N$. Figure \ref{fig:limited} shows the reconstruction for the starfish together with $N(=4,8)$ source points in the limited-aperture case with $k=8$. Figure \ref{fig:limited} illustrate that the part of boundary curve encompassed by the receivers can always be well reconstructed. Typically, due to the lack of information, recovery of the opposite side oriented towards the open aperture is less accurate. Anyhow, all the numerical results show that, even with limited-aperture data, the reconstruction method still reasonably works.
			
			\begin{table}[htpb]
				\caption{Relative $L^2$ errors for the reconstructions of the starfish with limited angles. }\label{table_limited}
				\centering
				\begin{tabular}{ccccc}
					\toprule
					& \multicolumn{2}{c}{$k=8$} & \multicolumn{2}{c}{$k=5$} \\
					\cmidrule(r){2-3} \cmidrule(r){4-5}
					$\theta$ &   $N=8$   &     $N=4$     &  $N=8$   &     $N=4$      \\ \midrule
					$\pi$       & $14.30\%$  &   $16.17\%$   & $8.39\%$ &   $12.04\%$    \\
					$3\pi/2$ & $7.06\%$ &   $7.33\%$    & $7.99\%$ &    $7.92\%$    \\ \bottomrule
				\end{tabular}
			\end{table}
			\begin{figure}[htpb]
				\centering
				\subfigure[]{\includegraphics[width=0.24\linewidth]{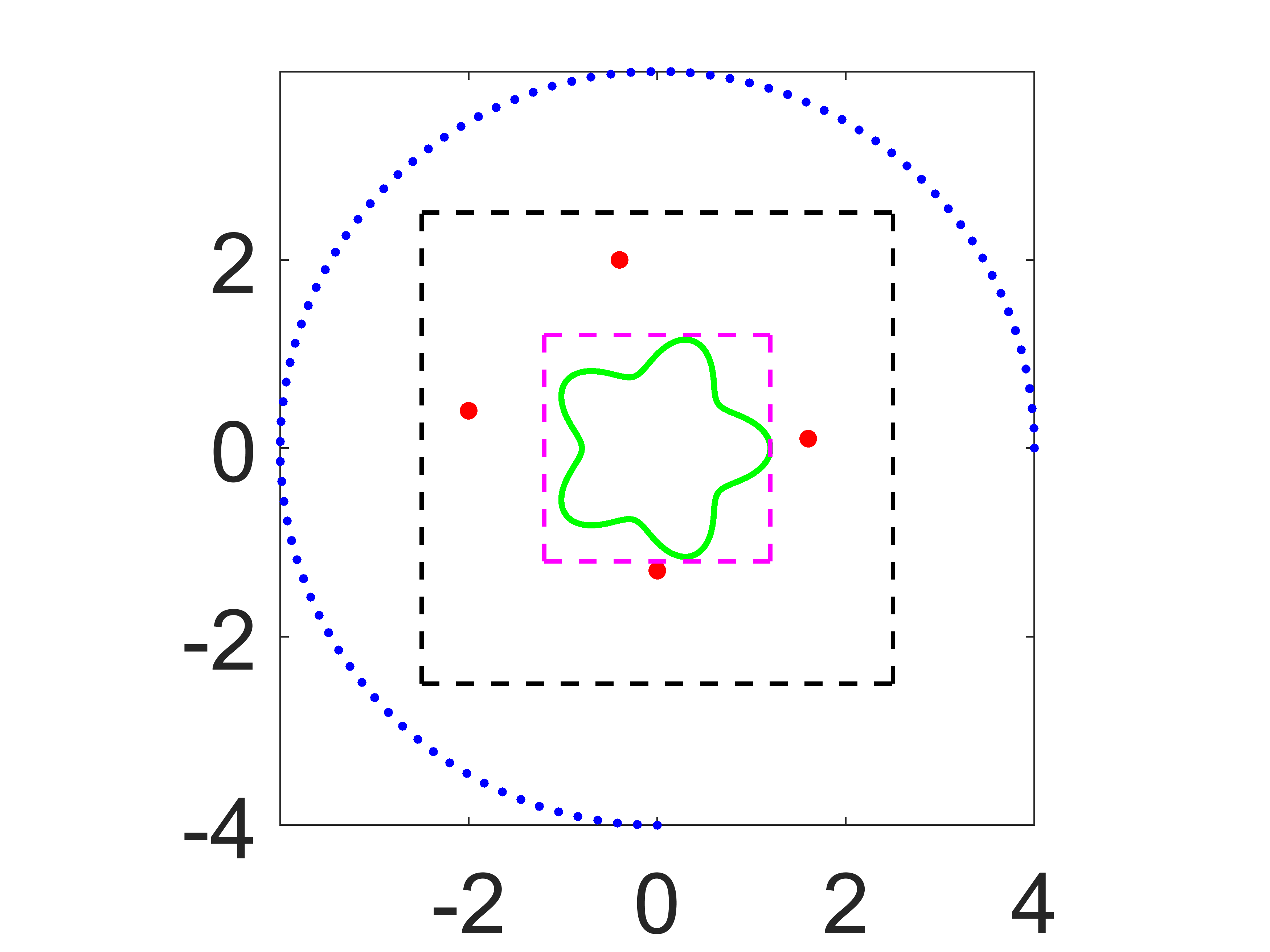}}
				\subfigure[]{\includegraphics[width=0.24\linewidth]{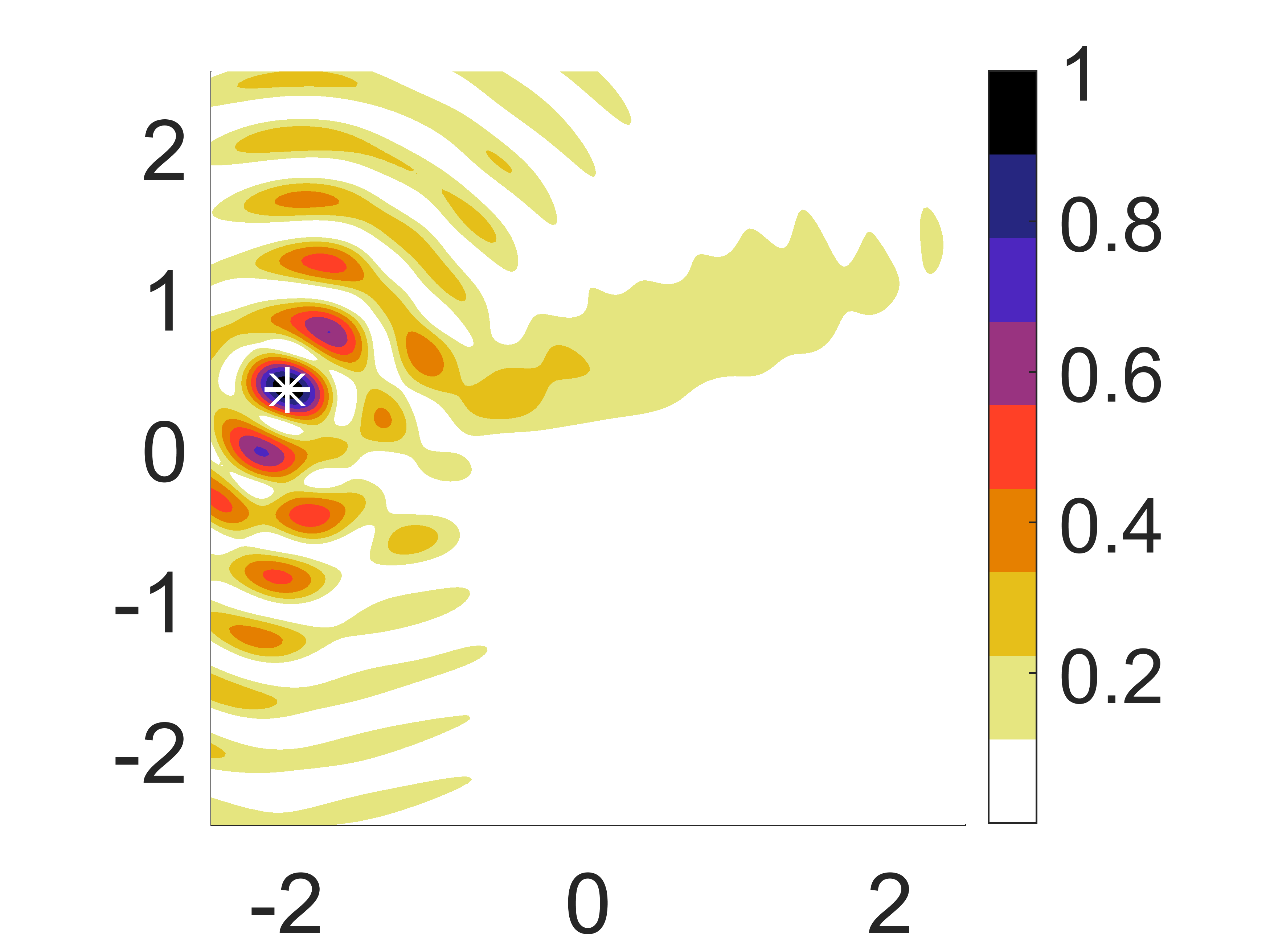}}
				\subfigure[]{\includegraphics[width=0.24\linewidth]{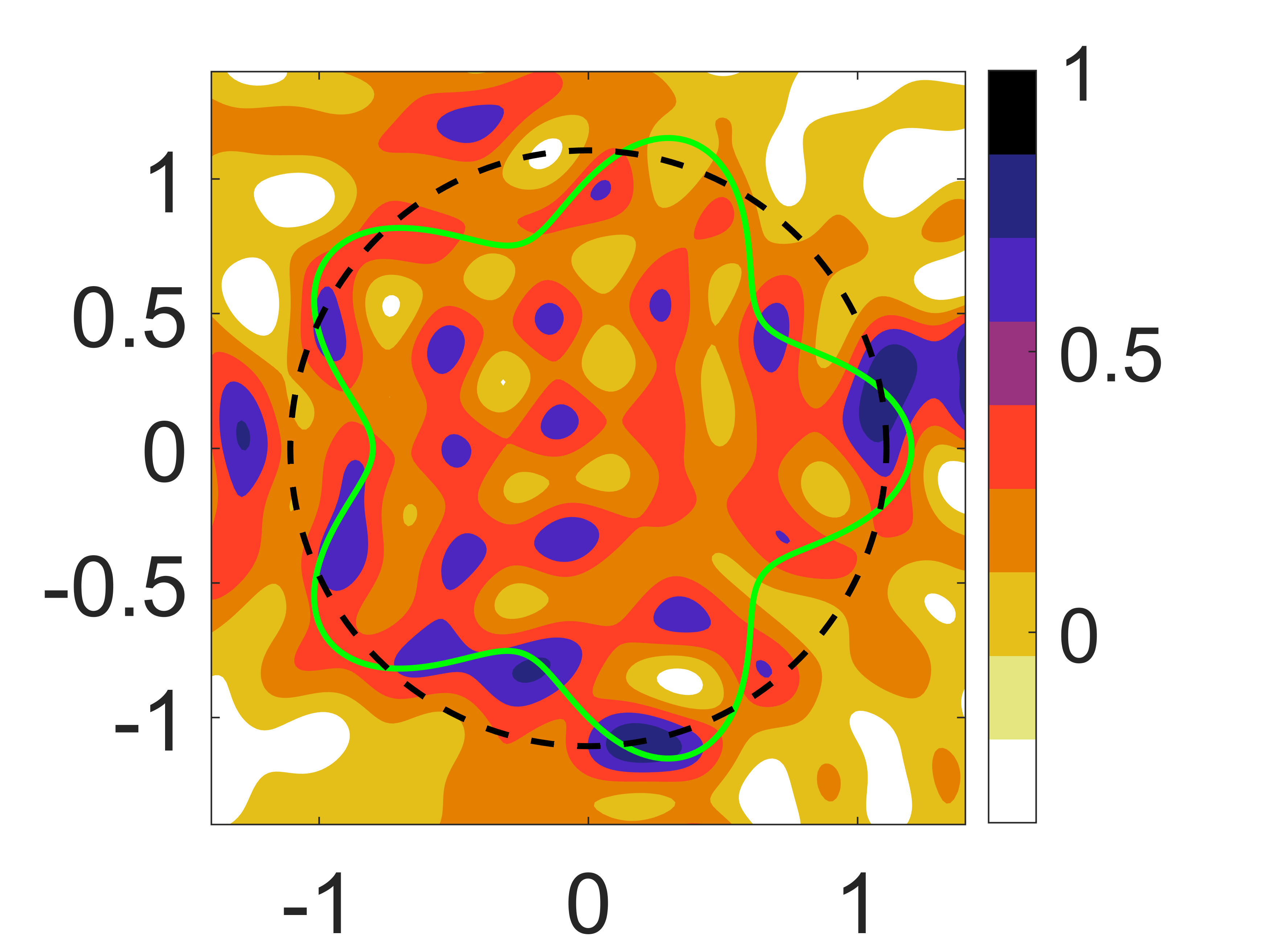}}
				\subfigure[]{\includegraphics[width=0.24\linewidth]{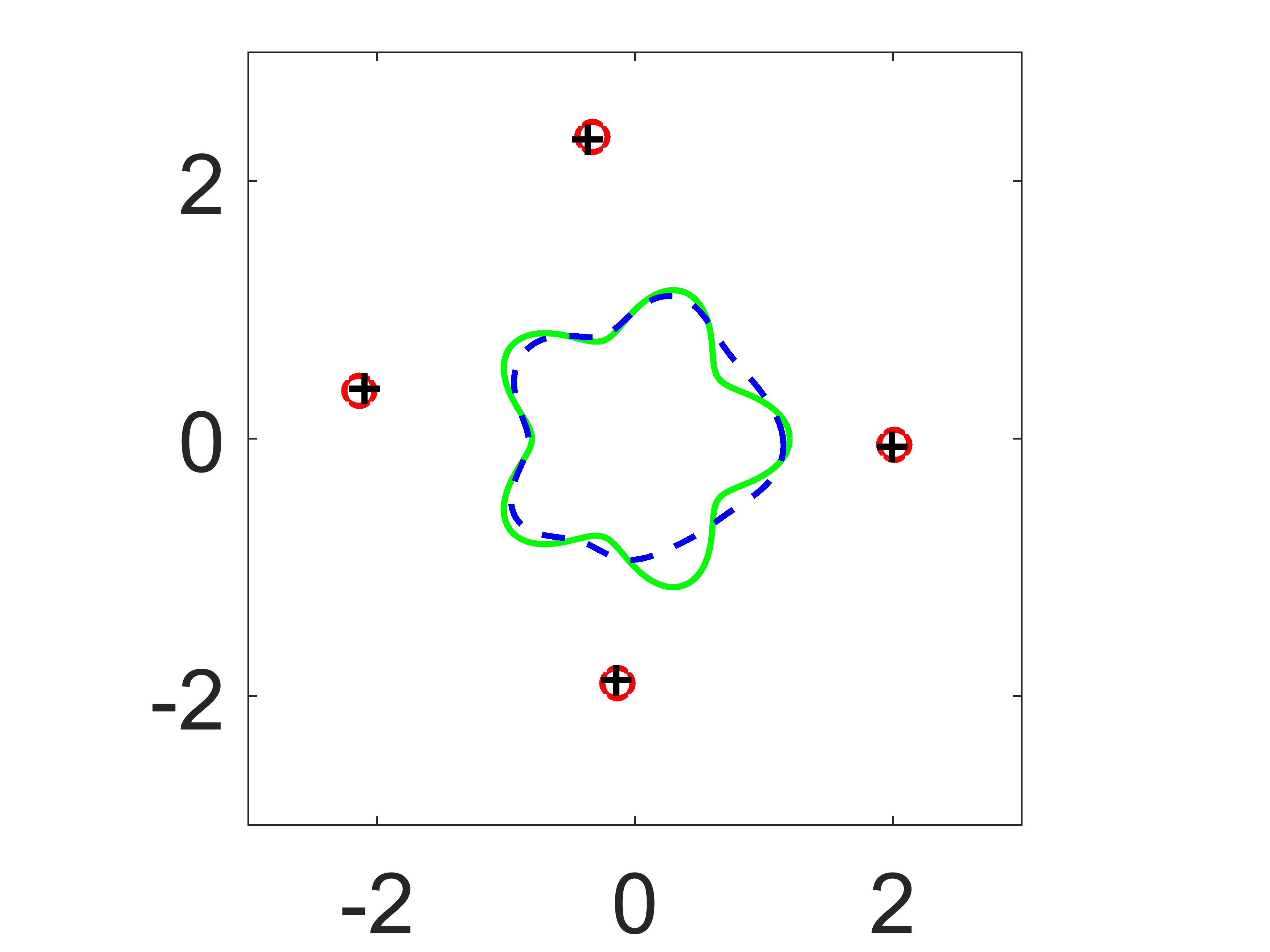}}
				
				\subfigure[]{\includegraphics[width=0.24\linewidth]{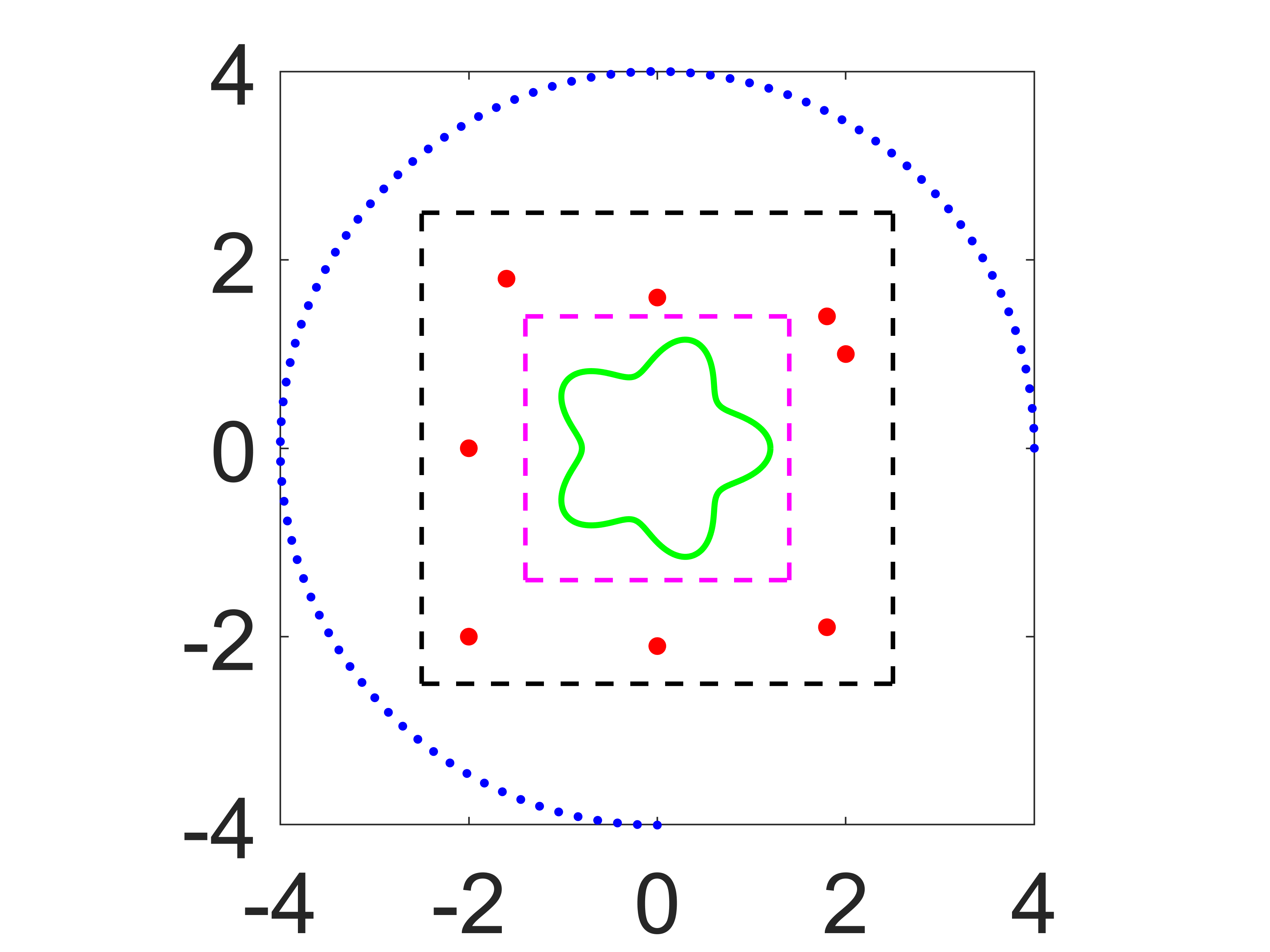}}
				\subfigure[]{\includegraphics[width=0.24\linewidth]{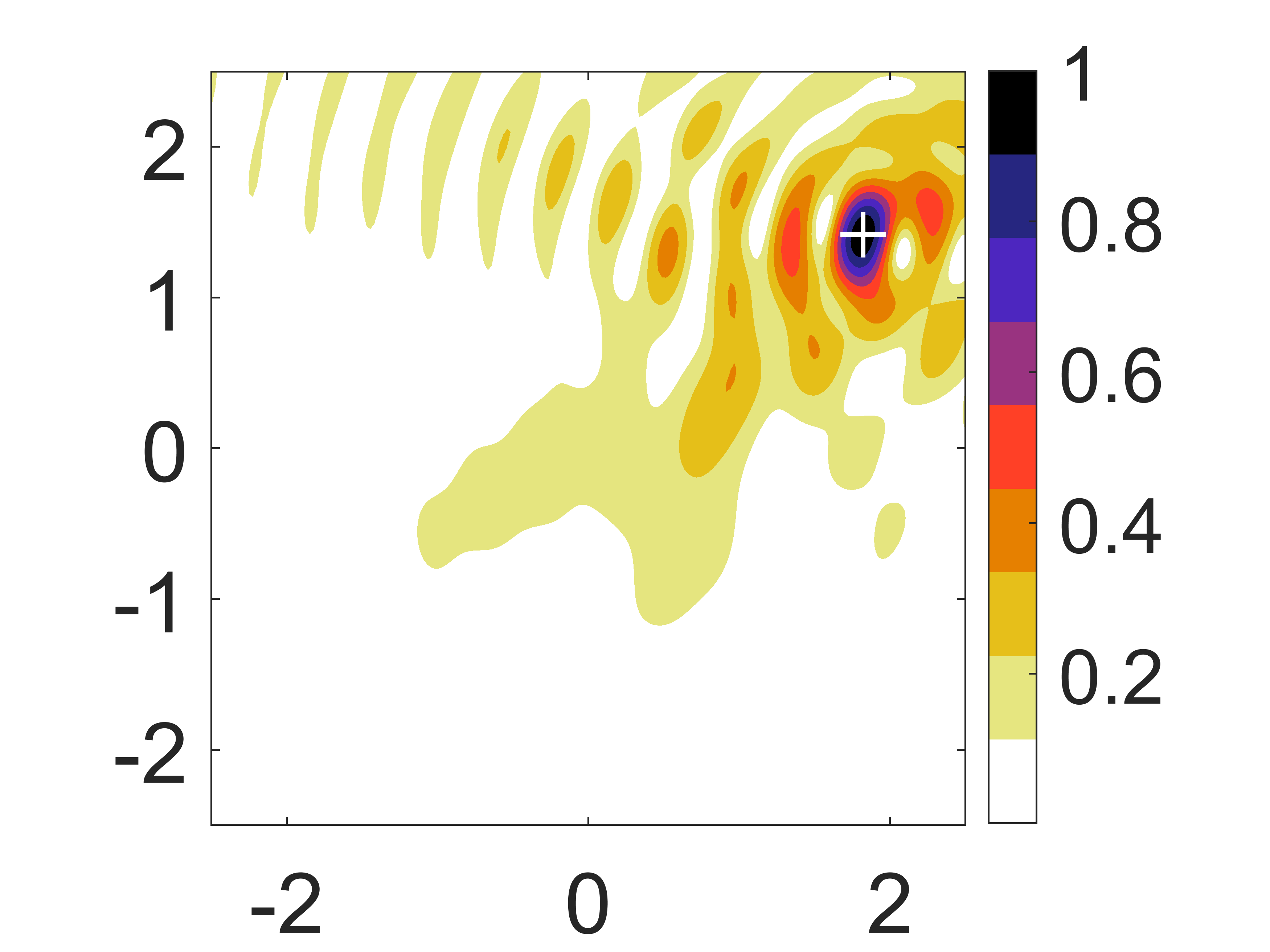}}
				\subfigure[]{\includegraphics[width=0.24\linewidth]{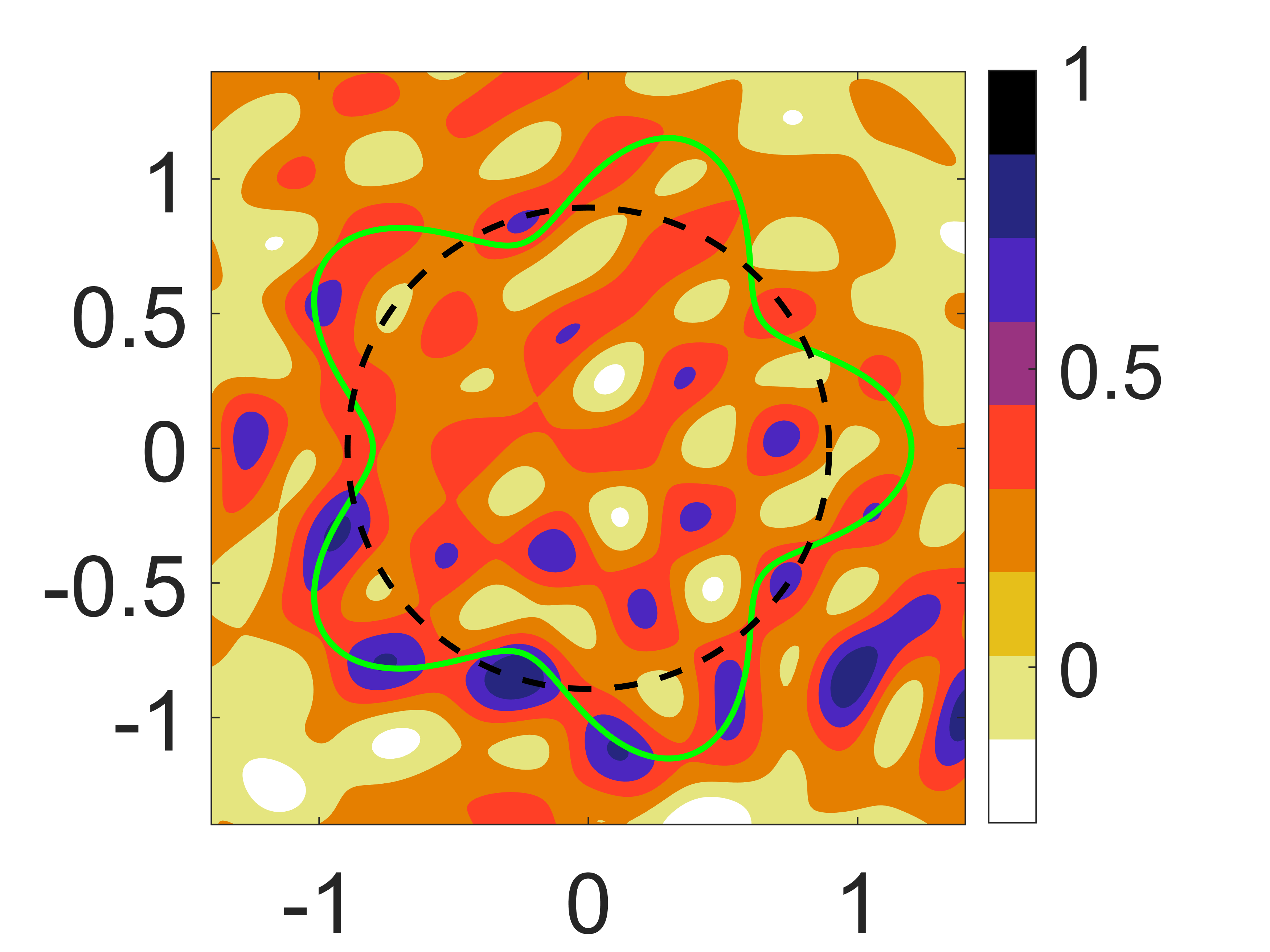}}
				\subfigure[]{\includegraphics[width=0.24\linewidth]{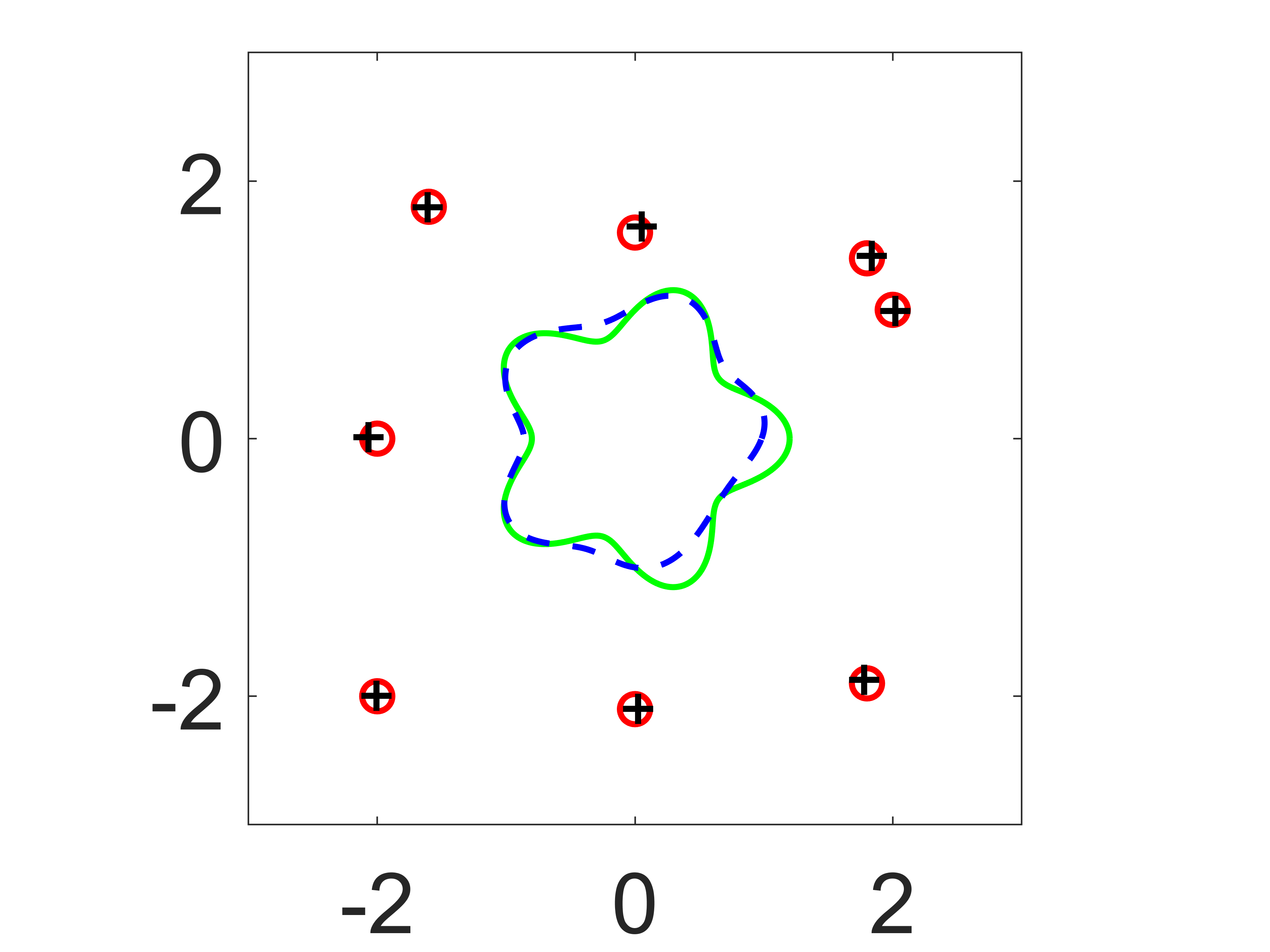}}
				
				\subfigure[]{\includegraphics[width=0.24\linewidth]{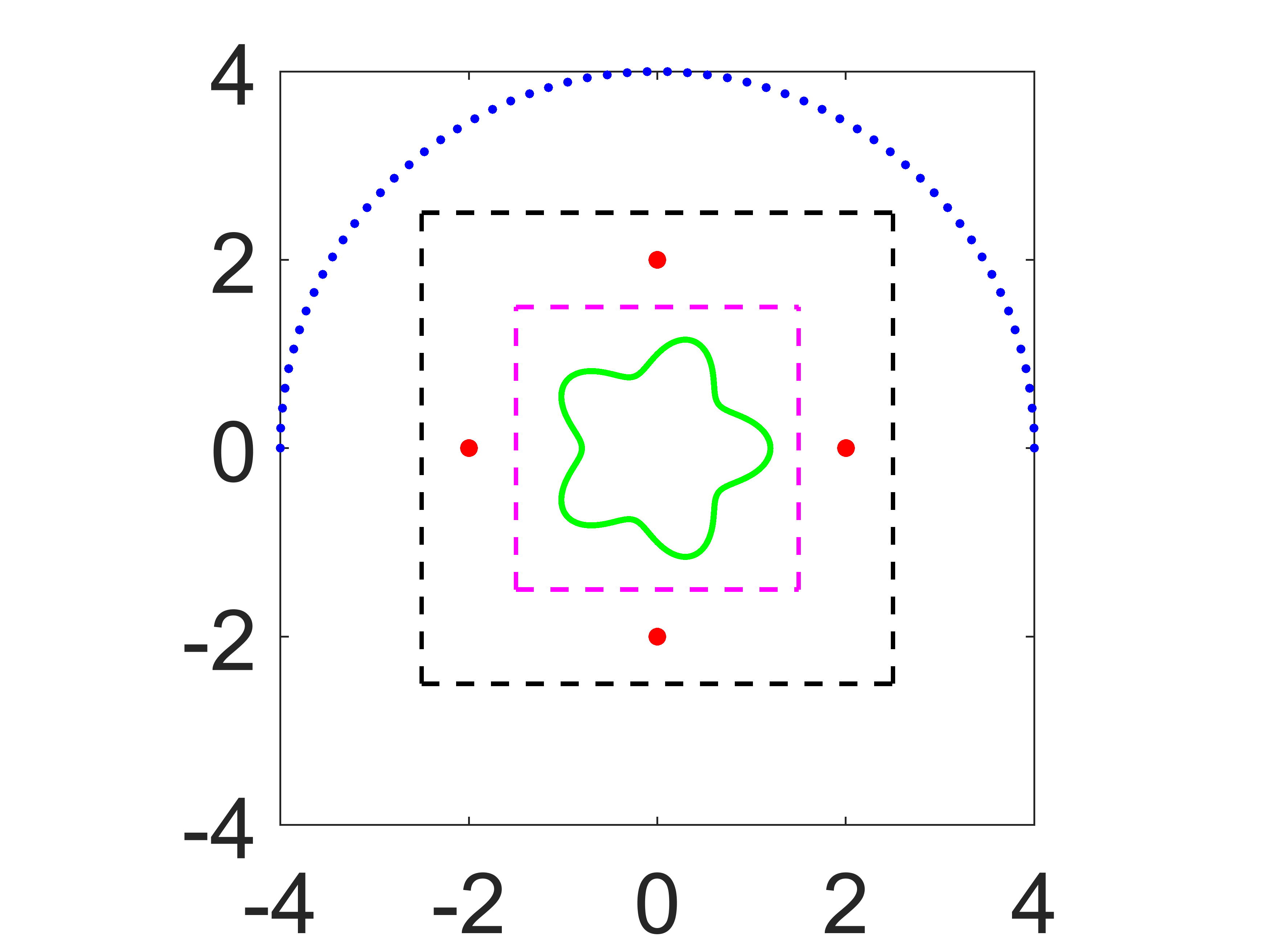}}
				\subfigure[]{\includegraphics[width=0.24\linewidth]{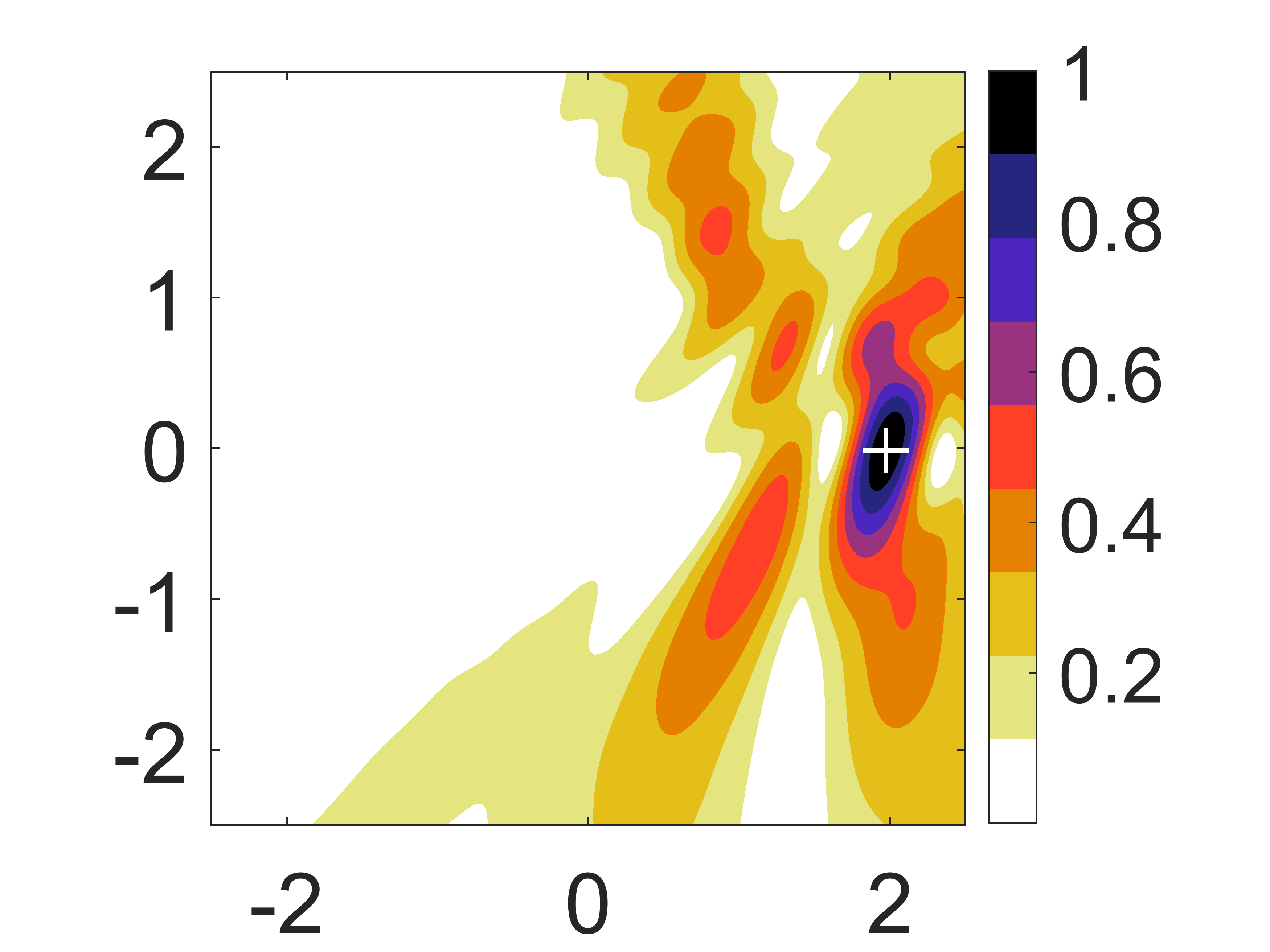}}
				\subfigure[]{\includegraphics[width=0.24\linewidth]{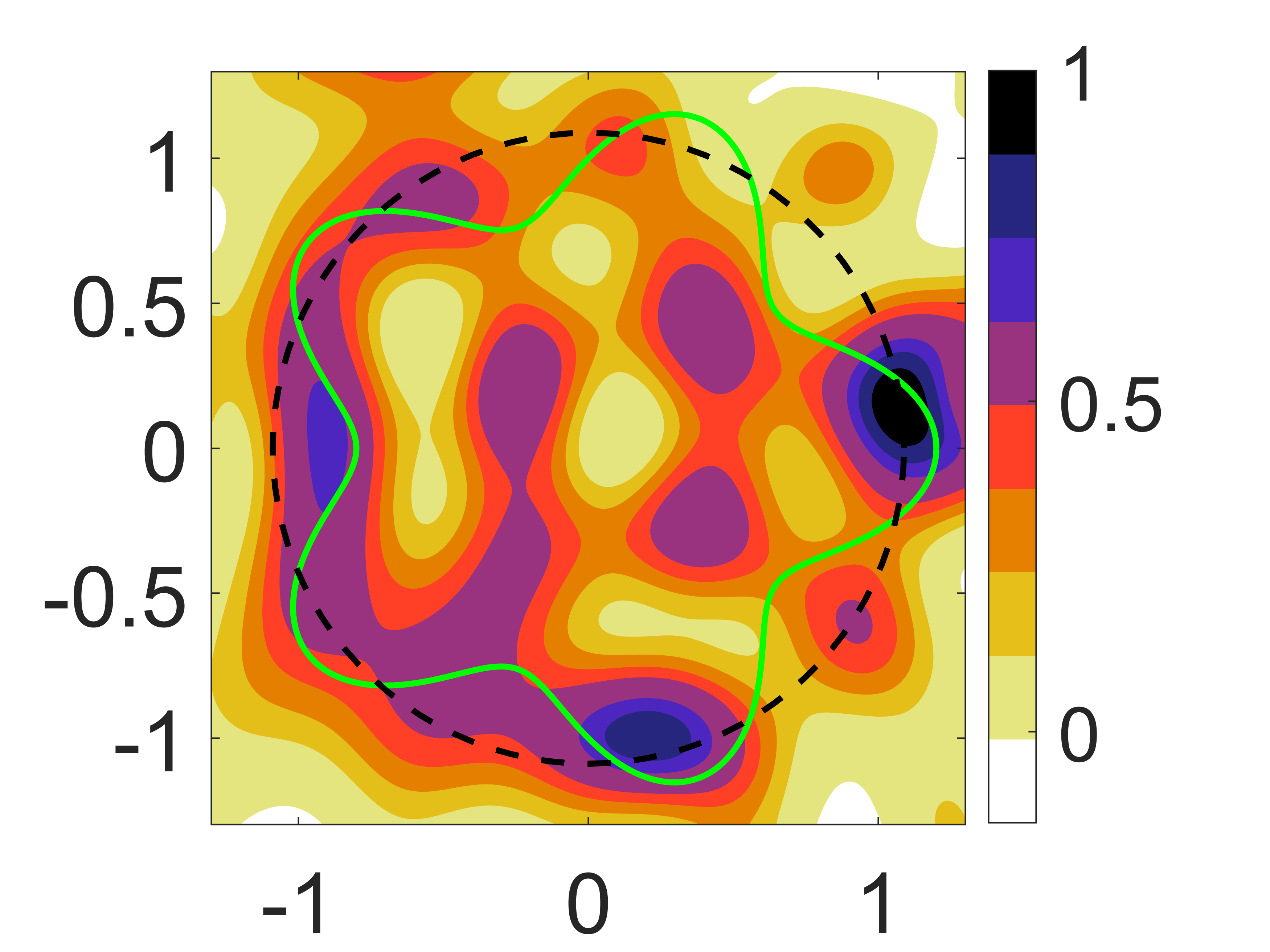}}
				\subfigure[]{\includegraphics[width=0.24\linewidth]{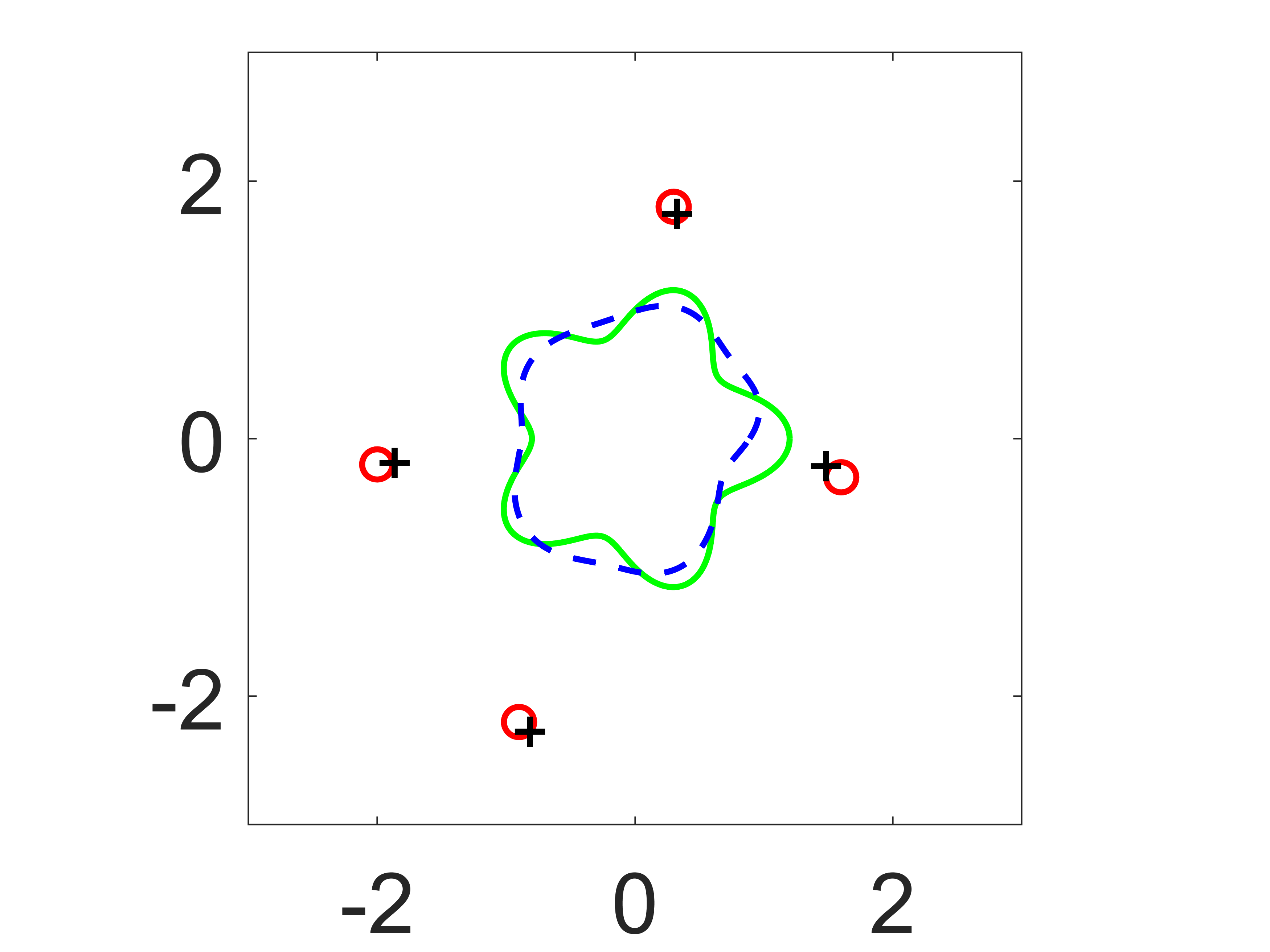}}
				
				\subfigure[]{\includegraphics[width=0.24\linewidth]{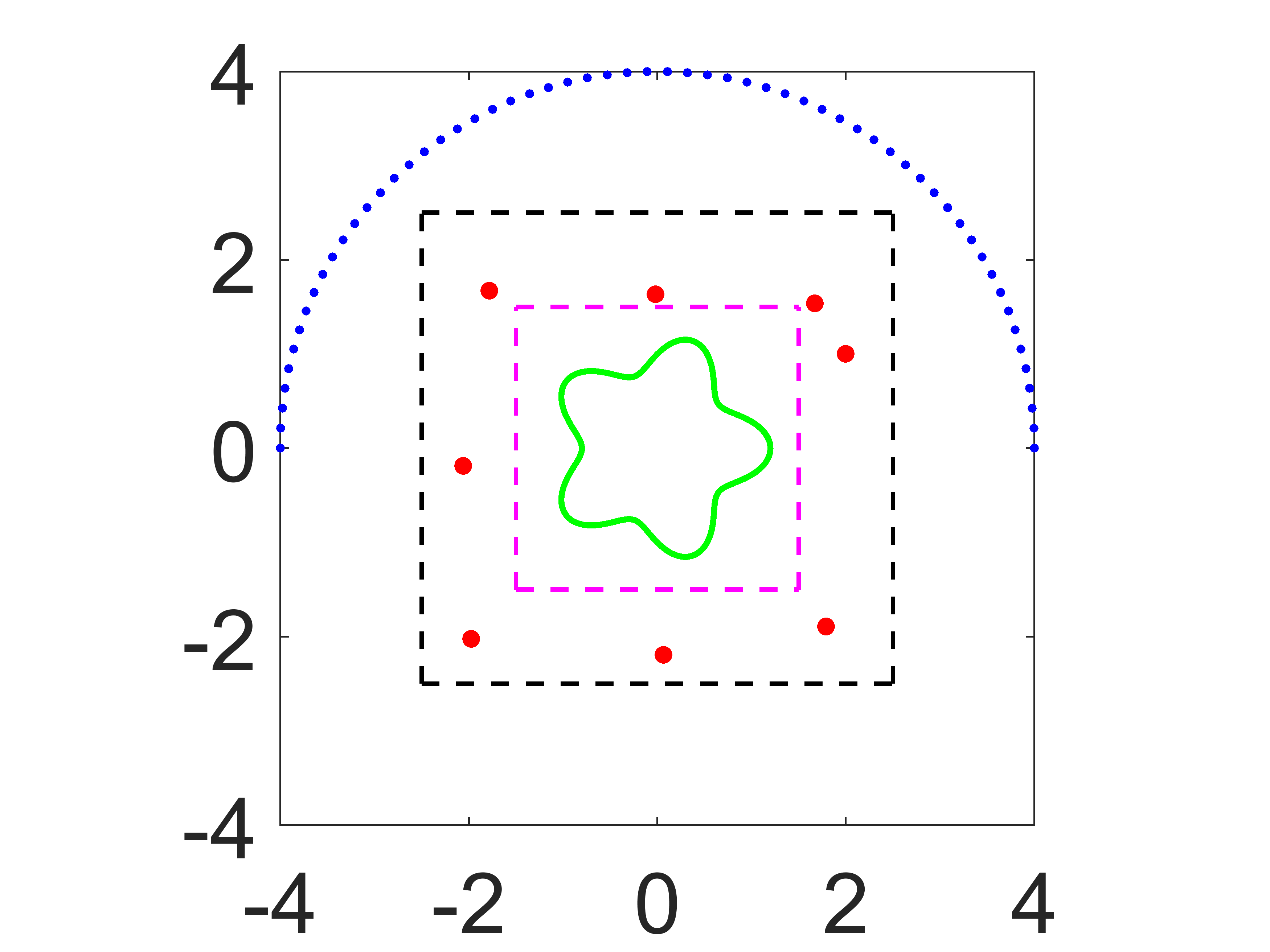}}
				\subfigure[]{\includegraphics[width=0.24\linewidth]{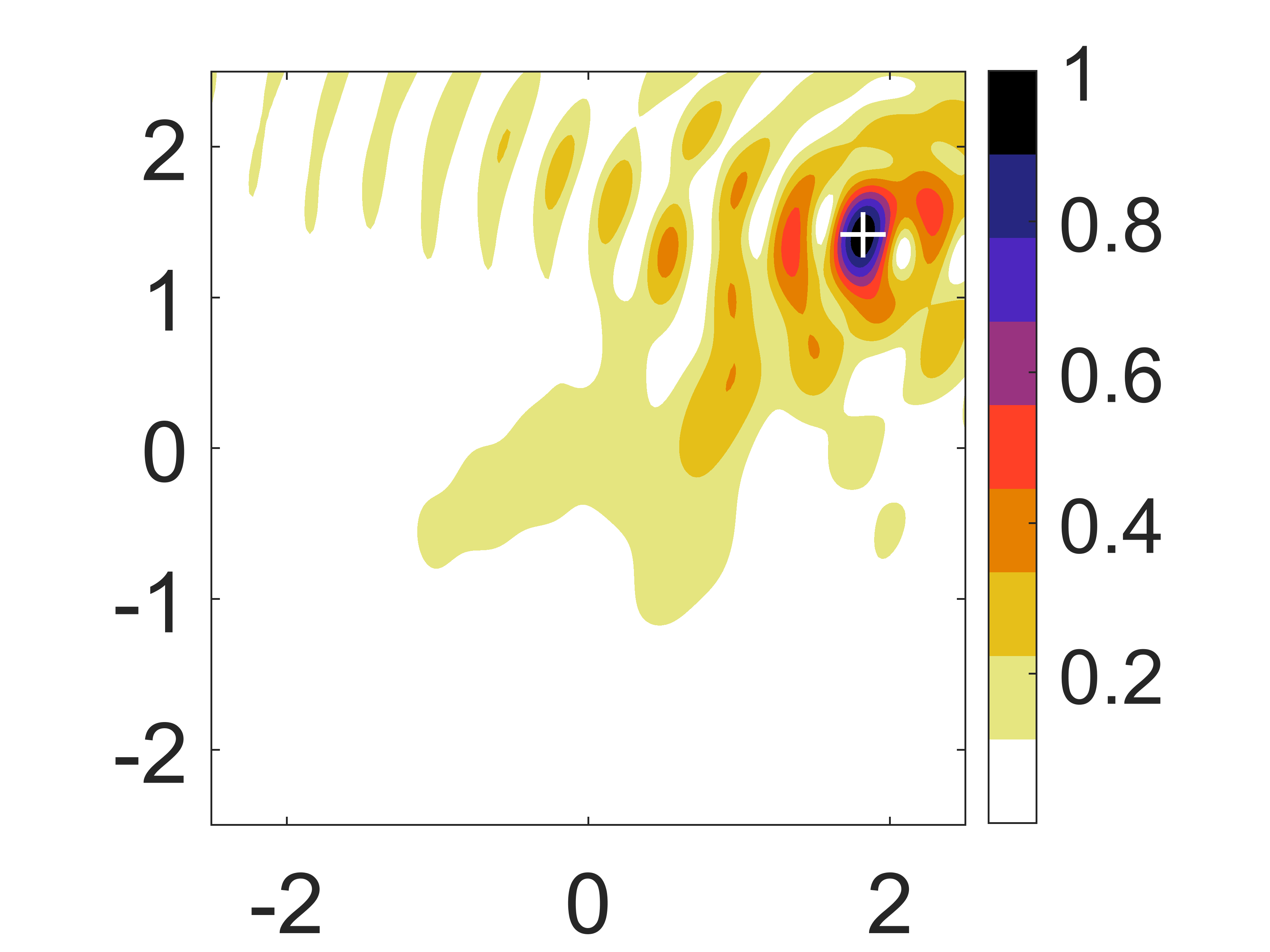}}
				\subfigure[]{\includegraphics[width=0.24\linewidth]{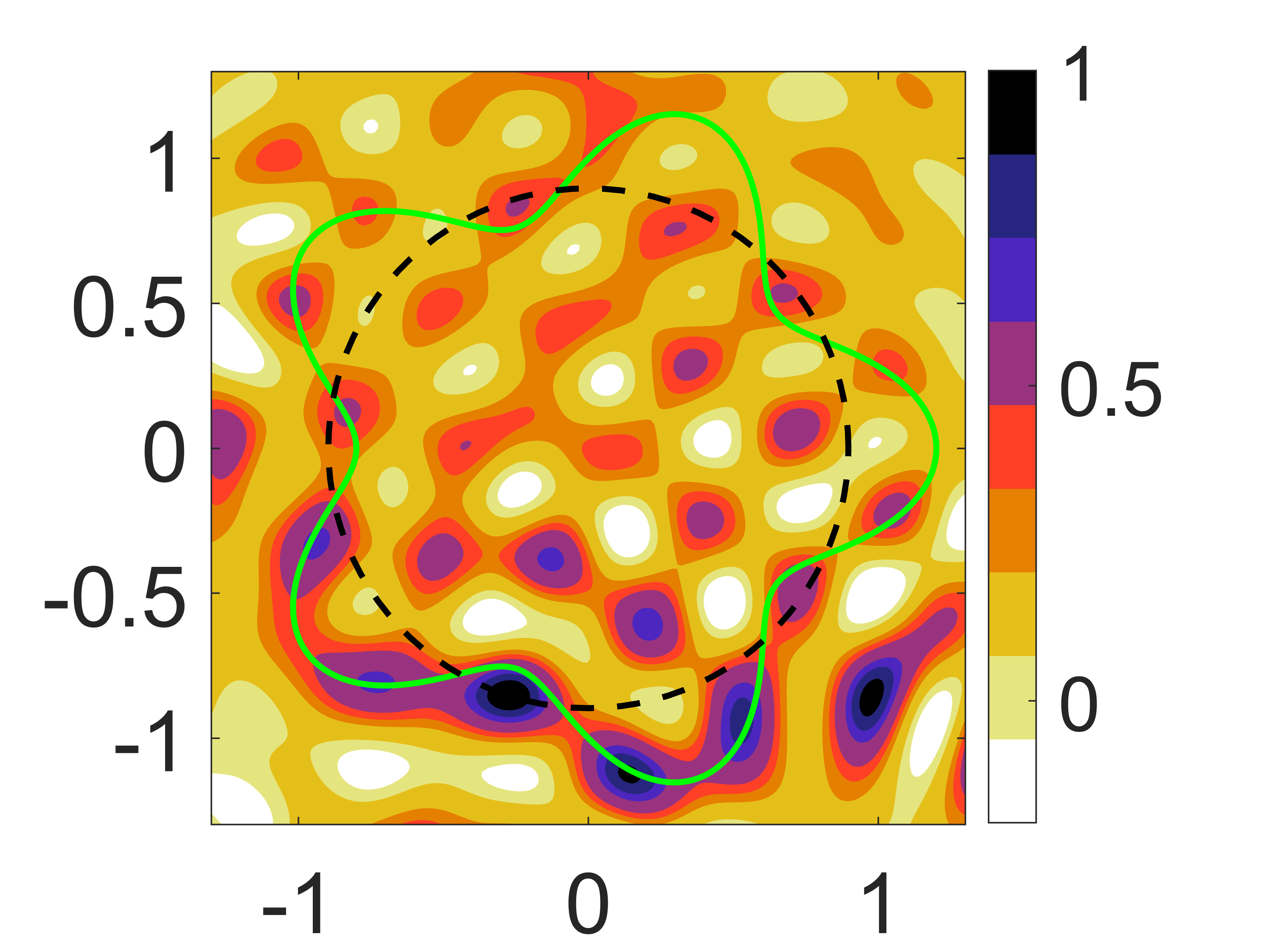}}
				\subfigure[]{\includegraphics[width=0.24\linewidth]{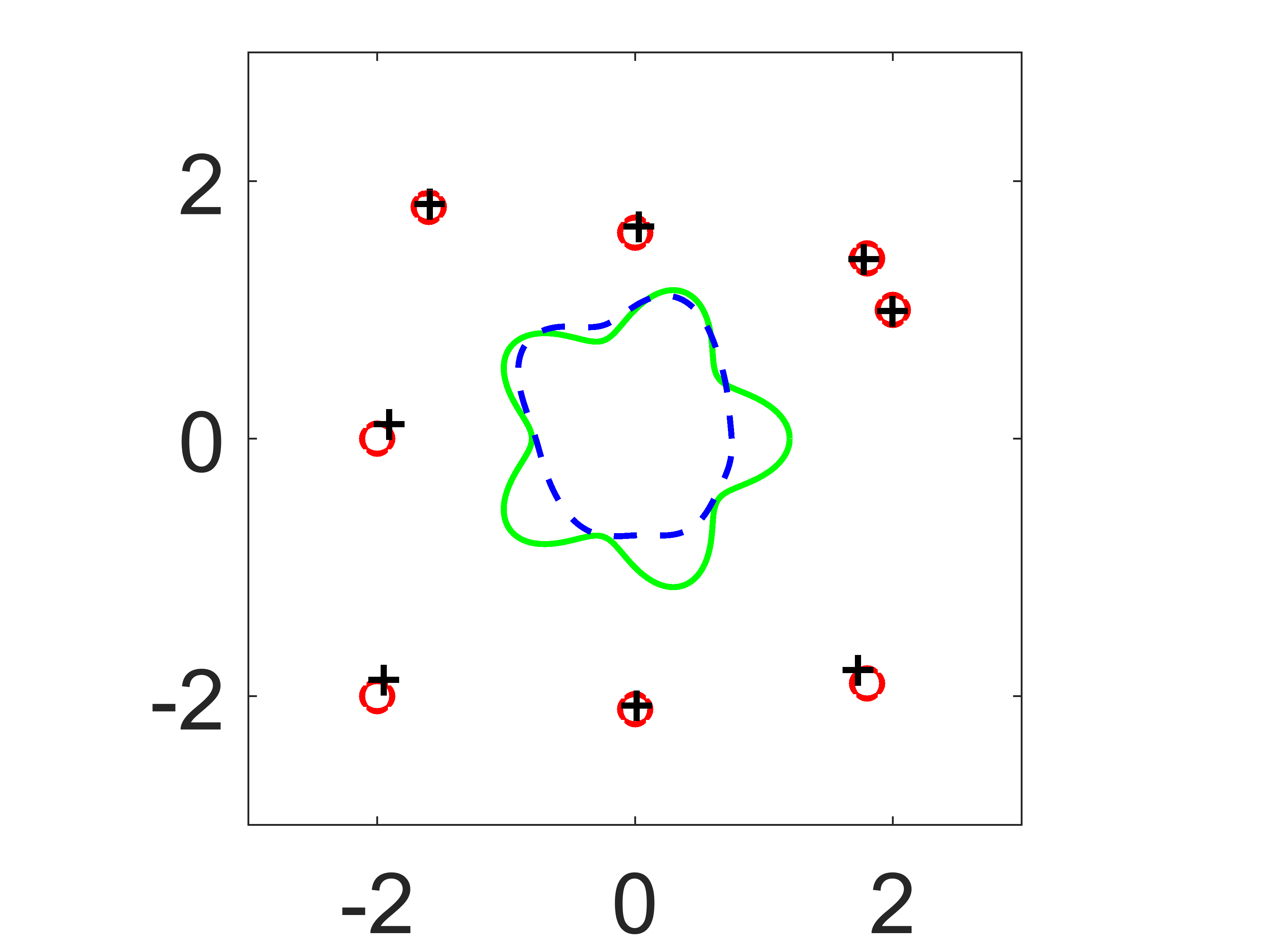}}
				\caption{Reconstructions of the starfish and different number of source points from limited aperture data. (The 1st column: problem geometry; the 2nd column: exact and reconstructed locations of the source points with DSM; the 3rd column: image of $I_D(y)$; the 4th column: reconstruction by the optimization method. The 1st row: $N=4,\theta=3\pi/2$; the 2nd row: $N=8, \theta=3\pi/2$; the 3rd row: $N=4, \theta=\pi$; the 4th row: $N=8, \theta=\pi$.)}\label{fig:limited}
			\end{figure}
		\end{example}

		\section{Conclusion}\label{sec:conclusions}
		In this work, we developed a novel optimization method to simultaneously reconstruct a sound-soft obstacle and multiple source points from near field scattering data. To obtain a good initial guess for the optimization method, a two-step sampling method is also developed to respectively image the rough source points and obstacle.  Theoretically, we analyze the convergence of the optimization method and the indicating behaviors of the sampling scheme. Numerical results are provided to demonstrate the feasibility and promising features of the proposed method. Our future work consists of the theoretical justification of the uniqueness issue, the extension to three-dimensional model and the other boundary conditions, as well as the applicability in the scenarios of electromagnetic or elastic waves.
		
		\section*{Acknowledgments}
		This work was supported by the NSFC grant under NO. 11971133 and the Fundamental Research Funds for the Central Universities. The authors would also like to thank Professor Xianchao Wang for valuable discussions.
		

	\end{document}